\documentclass[a4paper,10pt]{amsart}
\usepackage{mathrsfs}
\usepackage{amssymb} 
\usepackage{amsmath,amstext,amsthm,amscd}
\usepackage{dsfont}
\usepackage{mathrsfs}

\makeatletter
      \def\@setcopyright{}
      \def\serieslogo@{}
      \makeatother

\newcommand{\Complex}{\mathbb C}
\newcommand{\Real}{\mathbb R}

\newcommand{\ddbar}{\overline\partial}
\newcommand{\pr}{\partial}
\newcommand{\ol}{\overline}
\newcommand{\Td}{\widetilde}
\newcommand{\norm}[1]{\left\Vert#1\right\Vert}
\newcommand{\abs}[1]{\left\vert#1\right\vert}

\newcommand{\set}[1]{\left\{#1\right\}}
\newcommand{\To}{\rightarrow}
\DeclareMathOperator{\Ker}{Ker}
\newcommand{\pit}{\mathit\Pi}
\newcommand{\cali}[1]{\mathscr{#1}}
\newcommand{\cH}{\cali{H}}

\theoremstyle{plain}
\newtheorem{thm}{Theorem}[section]

\newtheorem{lem}[thm]{Lemma}
\newtheorem{prop}[thm]{Proposition}

\theoremstyle{definition}
\newtheorem{defn}[thm]{Definition}

\theoremstyle{remark}
\newtheorem{rem}[thm]{Remark}
\newtheorem{ex}[thm]{Example}
\newtheorem{que}[thm]{Question}
\newtheorem{con}[thm]{Conjecture}
\numberwithin{equation}{section}

\begin{document}
\title[]{Existence of CR sections for high power of semi-positive generalized Sasakian CR line bundles over generalized Sasakian CR manifolds}
\author[]{Chin-Yu Hsiao}
\address{Universit{\"a}t zu K{\"o}ln,  Mathematisches Institut,
    Weyertal 86-90,   50931 K{\"o}ln, Germany}
\thanks{The author is supported by the DFG funded project MA 2469/2-1}
\email{chinyu.hsiao@gmail.com} 
\subjclass[2010]{32V30, 32W10, 32W25.} 


\begin{abstract} 
Let $X$ be a compact generalized Sasakian CR manifold of dimension $2n-1$, $n\geqslant2$, and let
$L$ be a generalized Sasakian CR line bundle over $X$ equipped with a rigid semi-positive Hermitian fiber metric $h^L$. 
In this paper we prove that if $h^L$ is positive at some point of $X$ and 
conditions $Y(0)$ and $Y(1)$ hold at each point of $X$, then $L$ is big. 
\end{abstract}  

\maketitle \tableofcontents 

\section{Introduction and statement of the main results}

Let $X$ be a compact CR manifold of dimension $2n-1$, $n\geq 2$. When $X$ is strongly pseudoconvex and dimension of $X$ is greater than five, a classical theorem of 
L. Boutet de Monvel~\cite{BdM1:74b} asserts that $X$ can be globally CR embedded into $\Complex^N$, for some $N\in\mathbb N$. 
For a strongly pseudoconvex CR  manifold of dimension greater than five, the dimension of the kernel of the tangential Cauchy-Riemmann operator $\ddbar_b$ is infinite and we can find many CR functions to embed $X$ into complex space. 
When the Levi form of $X$ has negative eigenvalues, global embedding problems for $X$ are a 
very important subject in CR geometry, see 
Andreotti-Siu~\cite{AndS70}, Epstein-Henkin~\cite{EH00} and Marinescu~\cite{Ma96}. 
In this case, the dimension of the kernel of $\ddbar_b$ is finite and could be zero and in general, $X$ can not be globally CR embedded into complex space.  
Inspired by Kodaira, we introduced in \cite{HM09} (see also \cite{Ma96}) the idea of embedding CR manifolds by means of CR sections of tensor powers $L^k$ of a CR line bundle $L\To X$. If the dimension of the space $H^0_b(X,L^k)$ of CR sections of $L^k$ is large, when $k\To\infty$, one should find many CR sections to embed $X$ into projective space. In analogy to the Kodaira embedding theorem, it is natural to ask if $X$ can be globally embedded into projective space when it carries a CR line bundle with positive curvature? To understand this question, it is crucial to be able to know if ${\rm dim\,}H^0_b(X,L^k)\sim k^n$, for $k$ large? The following conjecture was implicit in~\cite{Ma96}

\begin{con}\label{con1} 
If $L$ is positive and the Levi form of $X$ has at least two negative and two positive eigenvalues, then 
\[{\rm dim\,}H^0_b(X,L^k)\sim k^n,\] 
for $k$ large.
\end{con} 

The difficulty of this conjecture comes from the presence of positive eigenvalues of the curvature of the line bundle and negative eigenvalues of the Levi form of $X$ and this causes the associated Kohn Laplacian 
to have no semi-classical spectral gap. 
This problem is also closely related to the fact that in the global $L^2$-estimates for the $\ddbar_b$-operator of Kohn-H\"{o}rmander there is a curvature term from the line bundle as well from the boundary and, in general, it is very difficult 
to control the sign of the total curvature contribution.  

In complex geometry, Demailly's holomorphic Morse inequalities~\cite{De:85} handled the corresponding analytical difficulties in a new way. Inspired by Demailly,  we established analogues of the holomorphic Morse inequalities of Demailly for CR manifolds
(see \cite{HM09})
\begin{thm} \label{HM}[Hsiao-Marinescu, 2009]
We assume that 
the Levi form of $X$ has at least two negative and two positive eigenvalues.
Then as $k\to\infty$,
\begin{equation}\label{*}
\begin{split}
&-{\rm dim\,}H^0_b(X, L^k)+{\rm dim\,}H^1_b(X, L^k)\\
&\quad\leqslant  \frac{k^{n}}{2(2\pi)^{n}}\Bigr(-\int_X\int_{\Real_{\phi(x),0}}\abs{\det(M^\phi_x+s\mathcal{L}_x)}ds\,dv_X(x)\\
&\quad+\int_X\int_{\Real_{\phi(x),1}}\abs{\det(M^\phi_x+s\mathcal{L}_x)}ds\,dv_X(x)\Bigr)+o(k^n),
\end{split}
\end{equation}
where $M^\phi_x$ is the associated curvature of $L$ at $x\in X$(see Definition~\ref{s1-d3}), $H^1_b(X, L^k)$ denotes the first $\ddbar_b$ cohomology group with values in $L^k$, $dv_X(x)$ is the volume form on $X$, $\mathcal{L}_x$ denotes the Levi form of $X$ at $x\in X$, and for $x\in X$, $q=0,1$, 
\begin{equation}\label{***}\begin{split}
&\Real_{\phi(x),q}=\{s\in\Real;\, \mbox{$M^\phi_x+s\mathcal{L}_x$ has exactly $q$ negative eigenvalues} \\
&\quad\mbox{and $n-1-q$ positive eigenvalues}\}.
\end{split}\end{equation}
\end{thm}

From \eqref{*}, we see that if 
\begin{equation}\label{**}
\int_X\int_{\Real_{\phi(x),0}}\abs{\det(M^\phi_x+s\mathcal{L}_x)}dsdv_X(x)>
\int_X\int_{\Real_{\phi(x),1}}\abs{\det(M^\phi_x+s\mathcal{L}_x)}dsdv_X(x)
\end{equation}
then $L$ is big, that is ${\rm dim\,}H^0_b(X,L^k)\sim k^n$. This is a very general criterion and it is desirable to refine it in some cases where \eqref{**} is not easy to verify. 
The problem still comes from the presence of positive eigenvalues of $M^\phi_x$ and negative eigenvalues of $\mathcal{L}_x$.

For the better understanding, let's see a simple example. We consider compact analogues of the Heisenberg group $H_n$. Let $\lambda_1,\ldots,\lambda_{n-1}$ be given non-zero integers. 
We assume that $\lambda_1<0,\ldots,\lambda_{n_-}<0,\lambda_{n_-+1}>0,\ldots,\lambda_{n-1}>0$.
Let $\mathscr CH_n=(\Complex^{n-1}\times\Real)/_\sim$\,, where
$(z, \theta)\sim(\Td z, \Td\theta)$ if
\[\begin{split}
&\Td z-z=(\alpha_1,\ldots,\alpha_{n-1})\in\sqrt{2\pi}\mathbb Z^{n-1}+i\sqrt{2\pi}\mathbb Z^{n-1},\\
&\Td\theta-\theta-i\sum^{n-1}_{j=1}\lambda_j(z_j\ol\alpha_j-\ol z_j\alpha_j)\in\pi\mathbb Z\,.
\end{split}\]
We can check that $\sim$ is an equivalence relation
and $\mathscr CH_n$ is a compact manifold of dimension $2n-1$. The equivalence class of $(z, \theta)\in\Complex^{n-1}\times\Real$ is denoted by
$[(z, \theta)]$.
For a given point $p=[(z, \theta)]$, we define the CR structure 
$T^{1, 0}_p\mathscr CH_n$ of $\mathscr CH_n$ to be the space spanned by
$\{\frac{\pr}{\pr z_j}+i\lambda_j\ol z_j\frac{\pr}{\pr\theta},\ \ j=1,\ldots,n-1\}$.
Then $(\mathscr CH_n,T^{1,0}_p\mathscr CH_n)$ is a compact CR manifold of dimension $2n-1$. With a suitable choose of a Hermitian metric on the complexified tangent bundle of $\mathscr CH_n$, the Levi form of $\mathscr CH_n$ at $p\in\mathscr CH_n$ is given by
$\mathcal{L}_p=\sum^{n-1}_{j=1}\lambda_jdz_j\wedge d\ol z_j$.
Let $L=(\Complex^{n-1}\times\Real\times\Complex)/_\equiv$ where $(z,\theta,\eta)\equiv(\Td z, \Td\theta, \Td\eta)$ if
\[(z,\theta)\sim(\Td z, \Td\theta),\ \
\Td\eta=\eta\exp(\sum^{n-1}_{j,t=1}\mu_{j,t}(z_j\ol\alpha_t+\frac{1}{2}\alpha_j\ol\alpha_t))\,,\quad\text{for $(\alpha_1,\ldots,\alpha_{n-1})=\Td z-z$},\]
where $\mu_{j,t}=\mu_{t,j}$, $j, t=1,\ldots,n-1$, are given integers. We can check that $\equiv$ is an equivalence relation and
$L$ is a CR line bundle over $\mathscr CH_n$. For $(z, \theta, \eta)\in\Complex^{n-1}\times\Real\times\Complex$ we denote
$[(z, \theta, \eta)]$ its equivalence class. Take the pointwise norm
$
\big\lvert[(z, \theta, \eta)]\big\rvert^2_{h^L}:=\abs{\eta}^2\exp\big(-\textstyle\sum^{n-1}_{j,t=1}\mu_{j,t}z_j\ol z_t\big)
$
on $L$. Then the associated curvature of $L$ is given by 
$M^\phi_x=\sum^{n-1}_{j,t=1}\mu_{j,t}dz_j\wedge d\ol z_t,\ \ \forall x\in\mathscr CH_n$.
In this simple example, Conjecture~\ref{con1} becomes 

\begin{que}\label{qpc1}
If $n_-\geq2$, $n-1-n_-\geq2$, and the matrix $\left(\mu_{j,t}\right)^{n-1}_{j,t=1}$ is positive definite, then ${\rm dim\,}H^0_b(\mathscr CH_n,L^k)\sim k^n$ ?
\end{que}

If $\mu_{j,t}=\abs{\lambda_j}\delta_{j,t}$, $j,t=1,\ldots,n-1$, and $n_-\geq2$, $n-1-n_-\geq2$, where $\delta_{j,t}=1$ if $j=t$, $\delta_{j,t}=0$ if $j\neq t$, then it is easy to see that $\Real_{\phi(x),1}=\emptyset$, where $\Real_{\phi(x),1}$ is given by \eqref{***}.
Combining this observation with Morse inequalities for CR manifolds(see \eqref{*}), we get

\begin{thm}\label{HM1} 
If $n_-\geq2$, $n-1-n_-\geq2$, and $\mu_{j,t}=\abs{\lambda_j}\delta_{j,t}$, $j,t=1,\ldots,n-1$, then ${\rm dim\,}H^0_b(\mathscr CH_n,L^k)\sim k^n$.
\end{thm}

The assumptions in Theorem~\ref{HM1} are somehow restrictive. It is clearly that we cannot go much further from Morse inequalities. By using Morse inequalities to approach Conjecture~\ref{con1}, we always have to impose extra conditions linking the Levi form and the curvature of the line bundle $L$. Similar problems also appear in the works of Marinescu~\cite{Ma96} and Berman~\cite{Be05} where they studied the $\ddbar$-Neumann cohomology groups associated to a high power of a given holomorphic line bundle on a compact complex manifold with boundary. In order to get many holomorphic sections, they also have to 
assume that, close to the boundary, the curvature of the line bundle is adapted to the Levi form of the boundary. In this work, by carefully studying semi-classical behaviour of microlocal Fourier transforms of the extreme functions for the spaces of lower energy forms of the associated Kohn Laplacian, we could solve Conjecture~\ref{con1} under rigidity conditions on $X$ and $L$ without any extra condition linking the Levi form of $X$ and the curvature of $L$. As an application, we solve Question~\ref{qpc1} completely. The proof of our main result presents a new way to overcome the analytic difficulty mentioned in the discussion after Conjecture~\ref{con1} under rigidity conditions. By using this new method, it is possible to remove 
the assumptions linking the curvatures of the line bundle and the boundary in the works of Marinescu~\cite{Ma96} and Berman~\cite{Be05} under rigidity conditions on the boundary and the line bundle.

The rigidity conditions we used in this work are inspired by the work of
Baouendi-Rothschild-Treves~\cite{BRT85}. They introduced rigidity condition on CR structure and proved that such a manifold can always be locally CR embedded in complex space as a generic submanifold. From their work, rigidity condition on CR structure seems suitable for our purpose. Initially, it is reasonable to first assume that $X$ can be locally embedded and study global embeddability of $X$. We 
can expect that the curvature of the line bundle and its transition functions have to satisfy some rigidity conditions (see Definition~\ref{s1-d1} and Definition~\ref{s1-d1-1}).  
Moreover, with these geometric conditions, it is possible to establish a micolocal asymptotic 
expansion of the Szeg\"{o} kernel and extend Kodaira embedding theorem to this situation. 

The geometric objects introduced in this paper form large classes of CR manifolds and CR line bundles. 
We hope that these geometric objects will be interesting for CR geometers and will be useful in CR geometry.

\subsection{Some standard notations}

We shall use the following notations: $\Real$ is the set of real numbers, $\ol\Real_+:=\set{x\in\Real;\, x\geq0}$, $\mathbb N=\set{1,2,\ldots}$, $\mathbb N_0=\mathbb N\bigcup\set{0}$. An element $\alpha=(\alpha_1,\ldots,\alpha_n)$ of $\mathbb N_0^n$ will be called a multiindex and the length of $\alpha$ is: $\abs{\alpha}=\alpha_1+\cdots+\alpha_n$. We write $x^\alpha=x_1^{\alpha_1}\cdots x^{\alpha_n}_n$, 
$x=(x_1,\ldots,x_n)$,
$\pr^\alpha_x=\pr^{\alpha_1}_{x_1}\cdots\pr^{\alpha_n}_{x_n}$, $\pr_{x_j}=\frac{\pr}{\pr x_j}$, $D^\alpha_x=D^{\alpha_1}_{x_1}\cdots D^{\alpha_n}_{x_n}$, $D_x=\frac{1}{i}\pr_x$, $D_{x_j}=\frac{1}{i}\pr_{x_j}$. 
Let $z=(z_1,\ldots,z_n)$, $z_j=x_{2j-1}+ix_{2j}$, $j=1,\ldots,n$, be coordinates of $\Complex^n$.  
We write $z^\alpha=z_1^{\alpha_1}\cdots z^{\alpha_n}_n$, $\ol z^\alpha=\ol z_1^{\alpha_1}\cdots\ol z^{\alpha_n}_n$,
$\frac{\pr^{\abs{\alpha}}}{\pr z^\alpha}=\pr^\alpha_z=\pr^{\alpha_1}_{z_1}\cdots\pr^{\alpha_n}_{z_n}$, $\pr_{z_j}=
\frac{\pr}{\pr z_j}=\frac{1}{2}(\frac{\pr}{\pr x_{2j-1}}-i\frac{\pr}{\pr x_{2j}})$, $j=1,\ldots,n$. 
$\frac{\pr^{\abs{\alpha}}}{\pr\ol z^\alpha}=\pr^\alpha_{\ol z}=\pr^{\alpha_1}_{\ol z_1}\cdots\pr^{\alpha_n}_{\ol z_n}$, $\pr_{\ol z_j}=
\frac{\pr}{\pr\ol z_j}=\frac{1}{2}(\frac{\pr}{\pr x_{2j-1}}+i\frac{\pr}{\pr x_{2j}})$, $j=1,\ldots,n$.

Let $\Omega$ be a $C^\infty$ paracompact manifold. 
We let $T\Omega$ and $T^*\Omega$ denote the tangent bundle of $\Omega$ and the cotangent bundle of $\Omega$ respectively.
The complexified tangent bundle of $\Omega$ and the complexified cotangent bundle of $\Omega$ will be denoted by $\Complex T\Omega$
and $\Complex T^*\Omega$ respectively. We write $\langle\,\cdot\,,\cdot\,\rangle$ to denote the pointwise duality between $T\Omega$ and $T^*\Omega$.
We extend $\langle\,\cdot\,,\cdot\,\rangle$ bilinearly to $\Complex T\Omega\times\Complex T^*\Omega$.
Let $E$ be a $C^\infty$ vector bundle over $\Omega$. The fiber of $E$ at $x\in\Omega$ will be denoted by $E_x$.
Let $F$ be another vector bundle over $\Omega$. We write 
$E\boxtimes F$ to denote the vector bundle over $\Omega\times\Omega$ with fiber over $(x, y)\in \Omega\times\Omega$ 
consisting of the linear maps from $E_x$ to $F_y$.  

\subsection{Generalized Sasakian CR manifolds and generalized Sasakian CR line bundles} 

Let $(X,T^{1,0}X)$ be a CR manifold of dimension $2n-1$, $n\geqslant2$, 
where $T^{1,0}X$ is a 
CR structure of $X$. That is, $T^{1,0}X$ is a complex $n-1$ dimensional subbundle of the complexified tangent bundle
$\Complex TX$, satisfying $T^{1,0}X\bigcap T^{0,1}X=\set{0}$, where $T^{0,1}X=\ol{T^{1,0}X}$, 
and $[\mathcal{V},\mathcal{V}]\subset\mathcal{V}$, where $\mathcal{V}=C^\infty(X,T^{1,0}X)$. In this section, we denote $Y:=X\times\Real$ and we write $t$ to denote the standard coordinate of $\Real$. We need  

\begin{defn} \label{s1-d-rigid}
We say that $(X,T^{1,0}X)$ is a generalized Sasakian CR manifold if there exists an integrable almost complex structure $J:TY\To TY$, $\Complex TY\To\Complex TY$, such that $Ju=iu$, $\forall u\in T^{1,0}X$.
\end{defn} 

Let $(X,T^{1,0}X)$ be a CR manifold of dimension $2n-1$, $n\geqslant2$, and let $J:TY\To TY$, $\Complex TY\To\Complex TY$, be an almost complex structure. We say that $J$ is a canonical complex structure on $Y$ if $J$ is integrable and $Ju=iu$, $\forall u\in T^{1,0}X$. Thus, $(X,T^{1,0}X)$ is a generalized Sasakian CR manifold if and only if there exists a canonical complex structure on $Y$. 

Let $(X,T^{1,0}X)$ be a generalized Sasakian CR manifold and let $J:TY\To TY$, $\Complex TY\To\Complex TY$ be any canonical complex structure on $Y$. From the Newlander-Nirenberg theorem, $J$ defines a complex structure $T^{1,0}Y\supset T^{1,0}X$. Put $T=J\frac{\pr}{\pr t}$. Then, $T\in C^\infty(X,TX)$, $T$ is a global real vector field on $X$. Since $J$ is integrable, it is easy to see that 
\begin{equation}\label{msmipI}\begin{split}
&\Complex TX=T^{1,0}X\oplus T^{0,1}X\oplus\set{\lambda T;\, \lambda\in\Complex},\\
&[T,\mathcal{V}]\subset\mathcal{V},\ \ \mathcal{V}:=C^\infty(X,T^{1,0}X).\end{split}\end{equation}

Conversely, let $(X,T^{1,0}X)$ be a CR manifold of dimension $2n-1$, $n\geqslant2$. We assume that there exists a global real vector field $T\in C^\infty(X,\Complex TX)$ such that \eqref{msmipI} hold. Then one can define a canonical complex structure on $Y$ by the rule: 
\[\begin{split}
&J:TY\To TY,\ \ \Complex TY\To\Complex TY\\
&Ju=iu,\ \ \forall u\in T^{1,0}X,\ \ Jv=-iv,\ \ \forall v\in T^{0,1}X,\ \ J\frac{\pr}{\pr t}=T.
\end{split}\] 
Thus, $(X,T^{1,0}X)$ is a generalized Sasakian CR manifold if and only if there exists a global real vector field $T\in C^\infty(X,\Complex TX)$ such that \eqref{msmipI} hold. We call $T$ a rigid global real vector field. 

Let's see some examples

\begin{ex}
(I) Let $M$ be an open subset with $C^\infty$ boundary $\pr M$ of a complex manifold $M'$ of dimension $n$. 
If for every $x_0\in\pr M$, we can find local holomorphic coordinates $(z_1,\ldots,z_n)$ defined in some neighborhood of $x_0$,
such that near $x_0$, $\pr M$ is given by the equation
\[{\rm Im\,}z_n=f(z_1,\ldots,z_{n-1}),\ \ \mbox{$f\in C^\infty$ is real valued},\]
then $\pr M$ is a generalized Sasakian CR manifold of dimension $2n-1$.
\\[2pt]
\ 
(II) Let $M$ be a complex manifold and $(E,h^E)$ be a holomorphic Hermitian line bundle on $M$, where the Hermitian 
fiber metric on $E$ is denoted by $h^E$. 
Let $(E^*,h^{E^*})$ be the dual bundle of $E$. We denote
\[G:=\set{v\in L^*;\, \abs{v}_{h^{L^*}}<1}\,,\quad\partial G=\set{v\in L^*;\, \abs{v}_{h^{L^*}}=1}\,.\]
The domain $G$ is called \emph{Grauert tube} associated to $E$. It is easy to see that $\pr G$ is a generalized Sasakian CR manifold. 
\\[2pt]
\
(III) The hypersurface
\[\set{(z_1,\ldots,z_n)\in\Complex^n;\, \sum^n_{j=1}\lambda_j\abs{z_j}^2=R}\]
is a generalized Sasakian CR manifolds, where $\lambda_j\in\Real$, $j=0,1,\ldots,n$, $R\in\Real$.
\\[2pt]
\
(IV) Heisenberg groups and compact Heisenberg groups (see section 9.1) are generalized Sasakian CR manifolds.
\end{ex}

From now on, we assume that $(X,T^{1,0}X)$ is a compact generalized Sasakian CR manifold and we let $\pi:Y\To X$ denote the standard projection. 

\begin{defn} \label{s1-d1}
Let $L$ be a complex line bundle over $X$. $(L,J)$ is a generalized Sasakian CR line bundle over $X$, where $J$ is a canonical complex structure on $Y$ if the pull-back $\pi^*L$ is a holomorphic line bundle over $Y$ with respect to $J$.
\end{defn} 

We need 

\begin{defn}\label{d-dhmilkI}
Let $T\in C^\infty(X,TX)$ be a rigid global real vector field on $X$. Let $U\Subset X$ be an open set. A function $u\in C^\infty (U)$ is said to be a $T$-rigid CR function on $U$ if $Tu=0$ and $Zu=0$ for all $Z\in C^\infty(U,T^{0,1}X)$.
\end{defn} 

From now on, we let $(L,J)$ be a generalized Sasakian CR line bundle over $X$ and we fix $T=J\frac{\pr}{\pr t}$. $T$ is a rigid global real vector field. Since $\pi^*L$ is a holomorphic line bundle over $Y$ with respect to the canonical complex structure $J$ on $Y$, it is easy to see that $X$ can be covered with open sets $U_j$ with trivializing sections $s_j$, $j=1,2,\ldots$, such that the corresponding transtion functions are $T$-rigid CR functions. In this paper, when trivializing sections $s$ are used, we will assume that they are of this special form. 

Fix a Hermitian fiber metric $h^L$ on $L$ and we will denote by
$\phi$ the local weights of the Hermitian metric $h^L$. More precisely, if
$s$ is a local trivializing
section of $L$ on an open subset $D\subset X$, then the local weight of $h^L$ with respect to $s$ is the function $\phi\in C^\infty(D,\Real)$ for which
\begin{equation} \label{s1-e6}
\abs{s(x)}^2_{h^L}=e^{-\phi(x)}\,,\quad x\in D.
\end{equation} 
We write $h^{\pi^*L}$ to denote the pull-back of $h^L$ by the projection $\pi$. Then, $h^{\pi^*L}$ is a Hermitian fiber metric on the holomorphic line bundle $\pi^*L$. Let $R^{\pi^*L}$ be the canonical curvature induced by $h^{\pi^*L}$. Let $\ddbar_J$ and $\pr_J$ be the $(0,1)$ and $(1,0)$ part of the exterior differential operator $d$ on functions with respect to $J$. If $s$ is a local trivializing
section of $L$ on an open subset $D\subset X$, $\abs{s}^2_{h^L}=e^{-\phi(x)}$, then 
\begin{equation}\label{msmiI}
R^{\pi^*L}(y)=\pr_J\ddbar_J\phi(\pi(y))\ \ \mbox{on $D\times\Real$}.
\end{equation}
We need 

\begin{defn} \label{s1-d3}
For $p\in X$ we define the Hermitian quadratic form $M^\phi_p$ on $T^{1,0}_pX$ by
\begin{equation} \label{msmiII}
M^\phi_p(U, \ol V)=\Big\langle U\wedge\ol V, R^{\pi^*L}(y)\Big\rangle,\ \ \pi(y)=p,\ \ U, V\in T^{1,0}_pX.
\end{equation}
\end{defn}

\begin{rem}\label{r-msmiIa}
Let $s$ be a local trivializing section of $L$ on an open subset $D\subset X$ and $\phi$ the corresponding local weight as in \eqref{s1-e6}. Let $\ddbar_b$ denote the tangential Cauchy-Riemann operator on functions (see chapter 7 in~\cite{CS01}). It is not difficult to see that for every $p\in D$, we have
\begin{equation} \label{s1-e14}
M^\phi_p(U, \ol V)=\frac{1}{2}\Big\langle U\wedge\ol V, d\big(\ddbar_b\phi-\pr_b\phi\big)(p)\Big\rangle,\ \ U, V\in T^{1,0}_pX,
\end{equation}
where $d$ is the usual exterior derivative and $\ol{\pr_b\phi}=\ddbar_b\ol\phi$.
\end{rem}

For $p\in X$, let $\mathcal{L}_p$ be the \emph{Levi form} (with respect to $T$) at $p$ (see  Definition~\ref{s1-d2}, for the precise meaning). 

\begin{defn} \label{s1-d-positive} 
We say that $h^L$ is positive at $x_0\in X$ if the Hermitian quadratic form $M^\phi_{x_0}$ is positive, 
$h^L$ is semi-positive if there is a positive constant $\delta>0$ such that for every $x\in X$ and $s\in[-\delta, \delta]$, the Hermitian quadratic form $M^\phi_x+2s\mathcal{L}_x$ is semi-positive. 
\end{defn} 

Since the transition functions are $T$-rigid CR functions, we can check that $T\phi$ is a well-defined global smooth function on $X$. 

\begin{defn} \label{s1-d1-1}
$h^L$ is said to be a $T$-rigid Hermitian fiber metric on $(L,J)$ if
\begin{equation}\label{pclr}
\mbox{$T\phi=C$ on $X$, for some constant $C$},
\end{equation}
where $\phi$ denotes the corresponding local weight as in \eqref{s1-e6}.
\end{defn} 

Note that the constant $C$ in \eqref{pclr} can be non-zero. (See section 9.1).

\begin{defn} \label{d-msmilkaI}
We say that $(L,J,h^L)$ is a rigid generalized Sasakian CR line bundle over X if $(L,J)$ is a generalized Sasakian CR line bundle over $X$ and $h^L$ is a $T$-rigid Hermitian fiber metric on $(L,J)$, $T=J\frac{\pr}{\pr t}$. 
\end{defn} 

\subsection{Hermitian CR geometry and the main results} 

Fix a smooth Hermitian metric $\langle\,\cdot\,|\,\cdot\,\rangle$ on $\Complex TX$ so that $T^{1,0}X$
is pointwise orthogonal to $T^{0,1}X$, $T$ is pointwise orthogonal to $T^{1,0}X\oplus T^{0,1}X$, $\langle T|T\rangle:=\norm{T}^2=1$ and $\langle u|v\rangle$ is real if $u$, $v$ are real tangent vectors.

Define 
\[\begin{split}
&T^{*1,0}X:=\set{e\in\Complex T^*X;\, \langle e,u\rangle=0,\forall u\in T^{0,1}X\oplus\set{\lambda T;\, \lambda\in\Complex}},\\
&T^{*0,1}X:=\set{f\in\Complex T^*X;\, \langle f,v\rangle=0,\forall v\in T^{1,0}X\oplus\set{\lambda T;\, \lambda\in\Complex}}.\end{split}\] $T^{*1,0}X$ and $T^{*0,1}X$ are 
subbundles of the complexified cotangent bundle $\Complex T^*X$. Define the vector bundle of $(0, q)$ forms of $X$ by $\Lambda^{0, q}T^*X:=\Lambda^{q}T^{*0,1}X$. 
Let $D\subset X$ be an open set. Let $\Omega^{0,q}(D)$ denote the space of smooth sections of $\Lambda^{0,q}T^*X$ over $D$. Similarly, if $E$ is a vector bundle over $D$, then we let $\Omega^{0,q}(D, E)$
denote the space of smooth sections of $\Lambda^{0,q}T^*X\otimes E$ over $D$. Let $\Omega^{0,q}_0(D, E)$ be the subspace of
$\Omega^{0,q}(D, E)$ whose elements have compact support in $D$. Let
\begin{equation} \label{s1-e5}
\ddbar_b:\Omega^{0,q}(X)\To\Omega^{0,q+1}(X)
\end{equation}
be the tangential Cauchy-Riemann operator (see chapter 7 in~\cite{CS01}).

The Hermitian metric $\langle\,\cdot\,|\,\cdot\,\rangle$ on $\Complex TX$ induces,
by duality, a Hermitian metric on $\Complex T^*X$ and also on  $\Lambda^{0, q}T^*X$ the bundle of $(0, q)$ forms of $X$. We shall also denote all these induced metrics by $\langle\,\cdot\,|\,\cdot\,\rangle$. For $f\in\Omega^{0,q}(X)$, we denote the pointwise norm $\abs{f(x)}^2:=\langle f(x)|f(x)\rangle$. Locally there is a real $1$-form $\omega_0$ of length one which is orthogonal to
$T^{*1,0}X\oplus T^{*0,1}X$. The form $\omega_0$ is unique up to the choice of sign.
We choose $\omega_0$ so that $\langle T, \omega_0\rangle=-1$.
Therefore $\omega_0$ is uniquely determined. We call $\omega_0$ the uniquely determined global real $1$-form. 
We have the
pointwise orthogonal decompositions:
\begin{equation} \label{s1-e3}\begin{split}
\Complex T^*X&=T^{*1,0}X\oplus T^{*0,1}X\oplus\set{\lambda\omega_0;\,
\lambda\in\Complex},  \\
\Complex TX&=T^{1,0}X\oplus T^{0,1}X\oplus\set{\lambda T;\,\lambda\in\Complex}.
\end{split}\end{equation}

We recall 

\begin{defn} \label{s1-d2}
For $p\in X$, the \emph{Levi form} $\mathcal{L}_p$ is the Hermitian quadratic form on $T^{1,0}_pX$ defined as follows. For any $U,\ V\in T^{1,0}_pX$, pick $\mathcal{U},\mathcal{V}\in
C^\infty(X, T^{1,0}X)$ such that
$\mathcal{U}(p)=U$, $\mathcal{V}(p)=V$. Set
\begin{equation} \label{s1-e4}
\mathcal{L}_p(U,\ol V)=\frac{1}{2i}\big\langle\big[\mathcal{U}\ ,\ol{\mathcal{V}}\,\big](p)\ ,\omega_0(p)\big\rangle\,,
\end{equation}
where $\big[\mathcal{U}\ ,\ol{\mathcal{V}}\,\big]=\mathcal{U}\ \ol{\mathcal{V}}-\ol{\mathcal{V}}\ \mathcal{U}$ denotes the commutator of $\mathcal{U}$ and $\ol{\mathcal{V}}$.
Note that $\mathcal{L}_p$ does not depend of the choices of $\mathcal{U}$ and $\mathcal{V}$. 

Since $\mathcal{L}_p$ is a Hermitian form there is a local orthonormal basis $\{\mathcal{U}_1,\ldots,\mathcal{U}_{n-1}\}$ of
$T^{1,0}X$ with respect to $\langle\,\cdot\,|\,\cdot\,\rangle$ such that $\mathcal{L}_p$ is diagonal in this basis,
$\mathcal{L}_p(\mathcal{U}_j,\ol{\mathcal{U}}_t)=\delta_{j,t}\lambda_j(p)$, 
$j,t=1,\ldots,n-1$, $\delta_{j,t}=1$ if $j=t$, $\delta_{j,t}=0$ if $j\neq t$, $\lambda_j(p)\in\Real$, $j=1,\ldots,n-1$.
The diagonal entries $\{\lambda_1(p),\ldots,\lambda_{n-1}(p)\}$ are
called the \emph{eigenvalues} of the Levi form
at $p\in X$ with respect to $\langle\,\cdot\,|\,\cdot\,\rangle$.

Given $q\in\{0,\ldots,n-1\}$, the Levi form is said to satisfy \emph{condition $Y(q)$} at $p\in X$, if $\mathcal{L}_p$ has at least either $\max{(q+1, n-q)}$ eigenvalues of the same sign or $\min{(q+1,n-q)}$ pairs of eigenvalues with opposite signs. Note that the sign of the eigenvalues does not depend on the choice of the metric $\langle\,\cdot\,|\,\cdot\,\rangle$.
\end{defn} 

Let $L^k$, $k>0$, be the $k$-th tensor power of the line bundle $L$. 
We write $\ddbar_{b,k}$ to denote the tangential
Cauchy-Riemann operator acting on forms with values in $L^k$, defined locally by:
\begin{equation} \label{s1-e7}
\ddbar_{b,k}:\Omega^{0,q}(X, L^k)\To\Omega^{0,q+1}(X, L^k)\,,\quad \ddbar_{b,k}(s^ku):=s^k\ddbar_bu,
\end{equation}
where $s$ is a local trivialization of $L$ on an open subset $D\subset X$
and $u\in\Omega^{0,q}(D)$.
We obtain a $\ddbar_{b,k}$-complex $(\Omega^{0,\bullet}(X, L^k),\ddbar_{b,k})$ with cohomology
\begin{equation}\label{db-cohom}
H^{\bullet}_b(X,L^k):=\ker\ddbar_{b,k}/\operatorname{Im} \ddbar_{b,k}.
\end{equation}
We assume that $X$ is compact and $Y(0)$ holds. By \cite[7.6-7.8]{Ko65}, \cite[5.4.11-12]{FK72}, \cite[Props.\,8.4.8-9]{CS01}, condition $Y(0)$ implies that $\dim H^0_b(X,L^k)<\infty$. 

Our main result is the following

\begin{thm}\label{main} 
Let $(X,T^{1,0}X)$ be a compact generalized Sasakian CR manifold of dimension $2n-1$, $n\geqslant2$ and let $(L,J,h^L)$ be a rigid generalized Sasakian CR line bundle over $X$. Assume that $h^L$ is semi-positive and positive at some point of $X$. Suppose conditions $Y(0)$ and $Y(1)$ hold at each point of $X$. Then, for $k$ large, there is a constant $c>0$ independnet of $k$, such that 
\[{\rm dim\,}H^0_b(X,L^k)\geq ck^n.\]
\end{thm} 

It should be mentioned that the Levi curvature assumptions in Theorem~\ref{main} are a bit more general than the ones in 
Conjecture~\ref{con1}. 

\begin{rem}\label{almsmiII}
It should be mention that Theorem~\ref{main} implies the famous Grauert-Riemenschneider conjecture in complex geometry. 
Let $M$ be a compact complex manifold of complex dimension $n$ and let $E\To M$ be a holomorphic line bundle with a Hermitian fiber metric $h^E$. Let $R^E$ denotes the canonical curvature on $E$ induced by $h^E$. We assume that $R^E$ is semi-positive and positive at some point of $M$. Then Grauert-Riemenschneider conjecture claims that $L$ is big, that is, ${\rm dim\,}H^0(M, E^k)\sim k^n$, where $H^0(M,E^k)$ denotes the space of global holomorphic sections of $E^k$ the $k$-th power of $E$. This conjecture was first solved by Siu~\cite{Si1:84}. Let's see how to obtain this conjecture from Theorem~\ref{main}. With the notations used above, let $(\Td X,T^{1,0}\Td X)$ be a compact generalized Sasakian CR manifold of dimension $2m-1$, $m\geqslant2$, such that the Levi form of $\Td X$ has at least two negative and two positive eigenvalues and let $(\Td L,\Td J,h^{\Td L})$ be a rigid generalized Sasakian CR line bundle over $\Td X$ with $h^{\Td L}$ is positive at every point of $\Td X$. We can find such $(\Td X,T^{1,0}\Td X)$ and $(\Td L,\Td J,h^{\Td L})$ (see setion 9). Consider $X=M\oplus\Td X$, $T^{1,0}X:=T^{1,0}M\oplus T^{1,0}\Td X$, where $T^{1,0}M$ denotes the holomorphic tangent bundle of $M$. Then, $(X,T^{1,0}X)$ is a compact generalized Sasakian CR manifold of dimension $2(m+n)-1$ and the Levi form of $X$ has at least two negative and two positive eigenvalues. Thus, conditions $Y(0)$ and $Y(1)$ hold at each point of $X$. Put $L:=E\otimes\Td L$. Then, $L$ is a complex line bundle over $X$. Let $J$ be the canonical complex structure on $X\times\Real$ induced by $\Td J$ and the complex structure on $M$. It is obviously that $(L,J)$ is a generalized Sasakian CR line bundle over $X$. 
Put $h^{L}=h^{E}\otimes h^{\Td L}$. Then $h^L$ is a Hermitian fiber metric on $L$ and $(L,J,h^L)$ is a rigid generalized Sasakian CR line bundle over $X$. Moreover, it is easy to check that $h^L$ is semi-positive and positive at some point of $X$. From Theorem~\ref{main}, we conclude that for $k$ large, there is a constant $C_0>0$ such that
\begin{equation}\label{e-lolkmiI}
{\rm dim\,}H^0_b(X,L^k)\geq C_0k^{n+m}.
\end{equation}
We notice that ${\rm dim\,}H^0_b(X,L^k)={\rm dim\,}H^0(M,E^k)\times{\rm dim\,}H^0_b(\Td X,\Td L^k)$ and it is well-known that there is a constant $C_1>0$ such that ${\rm dim\,}H^0_b(\Td X,\Td L^k)\leq C_1k^{m}$ (see~\cite{HM09}). Combining this observation and \eqref{e-lolkmiI}, we conclude that there is a constant $c>0$ such that
${\rm dim\,}H^0(M,E^k)\geq ck^{n}$.
\end{rem} 

We investigate Theorem~\ref{main} on generalized torus CR manifolds. Let $\Phi^t(x)$ be the $T$-flow. That is, 
$\Phi^t(x)$ is a differentiable mapping: 
\[t\To\Phi^t(x)\in X:I\To X,\] 
$I$ is an open interval in $\Real$, $0\in I$, such that 
$\Phi^0(x)=x$, $\forall x\in X$, and $\frac{d\Phi^t(x)}{dt}=T(\Phi^t(x))$. We need 

\begin{defn}\label{dpc} 
We say that $(X,T^{1,0}X)$ is a generalized torus CR manifold if there is a constant $\gamma_0>0$ such that 
$\Phi^t(x)$ is well-defined, $\forall \abs{t}\leq\gamma_0$, $\forall x\in X$, and
$\Phi^{\gamma_0}(x)=x$ for every $x\in X$.
\end{defn} 

\begin{defn}\label{dpc1} 
We say that $(L,J)$ is an admissible generalized Sasakian CR line bundle over a compact generalized torus CR manifold $X$ if we can find an open covering $\set{U_j}^N_{j=1}$ of $X$ such that $L$ is trivial on $U_j$, for each $j$, and \[\set{\Phi^t(x);\, x\in U_j, \abs{t}\leq\gamma_0}=U_j,\]
for each $j$, where $\gamma_0>0$ is as in Definition~\ref{dpc}.
\end{defn} 

Let $(L,J)$ be an admissible generalized Sasakian CR line bundle over a compact generalized torus CR manifold $(X,T^{1,0}X)$. Take any 
Hermitian fiber metric $h^{L}$ on $L$ and let $\phi$ denotes the corresponding local weight as in \eqref{s1-e6}. Let $h^{L}_1$ be the Hermitian fiber metric on $L$ locally given by 
$\abs{s}^2_{h^{L}_1}=e^{-\phi_1}$, where $\phi_1=\frac{1}{\gamma_0}\int^{\gamma_0}_0\phi(\Phi^t(x))dt$,  $\gamma_0>0$ is as in Definition~\ref{dpc}, $s$ is a local trivializing section of $L$ with the special form in Definition~\ref{dpc1}. It is easy to check that $h^{L}_1$
is well-defined and $T\phi_1=0$. Thus, $(L,J,h^{L}_1)$ is a rigid generalized Sasakian CR line bundle over $(X,T^{1,0}X)$. 
Moreover, we can show that if $M^\phi_x$ is positive on $X$ then $M^{\phi_1}_x$ is positive on $X$ (see Proposition~\ref{ppcf}, for the proof). Combining this with Theorem~\ref{main}, we obtain 

\begin{thm}\label{main1} 
Let $(X,T^{1,0}X)$ be a compact generalized torus CR manifold of dimension $2n-1$, $n\geqslant2$ and let $(L,J)$ 
be an admissible generalized Sasakian CR line bundle over $X$ with a Hermitian fiber metric $h^L$. We assume that $h^L$ is positive on $X$ and conditions $Y(0)$ and $Y(1)$ hold at each point of $X$. Then, for $k$ large, there is a constant $c>0$ independnet of $k$, such that 
\[{\rm dim\,}H^0_b(X,L^k)\geq ck^n.\]
\end{thm} 

\subsection{The outline of the proof of Theorem~\ref{main}}

Let $\Box^{(q)}_{b,k}$ denote the Kohn Laplacian with values in $L^k$ (see section 2). Fix $q=0,1,\ldots,n-1$. We assume that  $Y(q)$ holds. It is well-known that $\Box^{(q)}_{b,k}$ has a discrete
spectrum, each eigenvalues occurs with finite multiplicity and all eigenforms are smooth and 
${\rm Ker\,}\Box^{(q)}_{b,k}:=\cH^q_b(X,L^k)\cong H^q_b(X,L^k)$.
For $\lambda\geq0$, let $\cH^q_{b,\leq\lambda}(X,L^k)$ denote the space spanned by the eigenforms of $\Box^{(q)}_{b,k}$ 
whose eigenvalues are bounded by $\lambda$. Now, we assume that $Y(0)$ and $Y(1)$ hold and $(L,h^L)$ is semi-positive and positive at some point of $X$. Take $\delta_0>0$ be a small constant so that 
$M^\phi_x+2s\mathcal{L}_x\geq0$, $\forall\abs{s}\leq\delta_0$, $\forall x\in X$. 
Take $\psi(\eta)\in C^\infty_0(]-\delta_0,\delta_0[,\ol\Real_+)$ so that $\psi(\eta)=1$ if $-\frac{\delta_0}{2}\leq \eta\leq\frac{\delta_0}{2}$. Take $\chi(t)\in C^\infty_0(]-2,2[,\ol\Real_+)$ so that $0\leq\chi(t)\leq1$ and $\chi(t)=1$ if $-1\leq t\leq 1$ and $\chi(-t)=\chi(t)$ for all $t\in\Real$. Fix $M>0$. 
Under the rigidity assumptions in Theorem~\ref{main}, we can construct global continuous operators $Q^{(0)}_{M,k}:C^\infty(X,L^k)\To C^\infty(X,L^k)$ and $Q^{(1)}_{M,k}:\Omega^{0,1}(X,L^k)\To\Omega^{0,1}(X,L^k)$ such that  
\begin{equation}\label{epcI}
\ddbar_{b,k}Q^{(0)}_{M,k}=Q^{(1)}_{M,k}\ddbar_{b,k}\ \ \mbox{on $C^\infty(X,L^k)$}
\end{equation}
and $Q^{(0)}_{M,k}$, $Q^{(1)}_{M,k}$ are formally given by the following. Let $s$ be a local section of $L$ on $D\subset X$, $\abs{s}^2_{h^L}=e^{-\phi}$, and let $\Phi^t(x)$ be the $T$-flow. Then, 
\begin{equation}\label{epcII}
\begin{split}
&(Q^{(0)}_{M,k}f)(x)=s^ke^{\frac{k}{2}\phi(x)}\int e^{-it\eta}\psi(\eta)\chi(\frac{t}{M})e^{-\frac{k}{2}\phi(\Phi^{\frac{t}{k}}(x))}\Td f(\Phi^{\frac{t}{k}}(x))dtd\eta\ \ \mbox{on $D$},\\
&(Q^{(1)}_{M,k}g)(x)=s^ke^{\frac{k}{2}\phi(x)}\int e^{-it\eta}\psi(\eta)\chi(\frac{t}{M})e^{-\frac{k}{2}\phi(\Phi^{\frac{t}{k}}(x))}\Td g(\Phi^{\frac{t}{k}}(x))dtd\eta\ \ \mbox{on $D$},
\end{split}
\end{equation}
where $f=s^k\Td f\in C^\infty_0(D,L^k)$, $g=s^k\Td g\in\Omega^{0,1}_0(D,L^k)$. (See section 5, for the precise definitions of the operators $Q^{(0)}_{M,k}$, $Q^{(1)}_{M,k}$.) Let $\langle\,\cdot\,|\,\cdot\,\rangle_{h^{L^k}}$ denote the Hermitian metric on $\Lambda^{0,q}T^*X\otimes L^k$ induced by $h^{L}$ and $\langle\,\cdot\,|\,\cdot\,\rangle$. Let $dv_X=dv_X(x)$ be the volume form on $X$ induced by $\langle\,\cdot\,|\,\cdot\,\rangle$ and let $(\,\cdot\,|\,\cdot\,)_{h^{L^k}}$ be the $L^2$ inner product on $\Omega^{0,q}(X,L^k)$ induced by $\langle\,\cdot\,|\,\cdot\,\rangle_{h^{L^k}}$ and $dv_X$. For $\lambda\geq0$, define 
\begin{equation}\label{epcIII} 
\begin{split}
&(Q^{(0)}_{M,k}\pit^{(0)}_{k,\,\leq\lambda})(x):=\sum^{m_k}_{j=1}\langle Q^{(0)}_{M,k}f_j(x)| f_j(x)\rangle_{h^{L^k}}, \\
&(Q^{(1)}_{M,k}\pit^{(1)}_{k,\,\leq\lambda}\ol{Q^{(1)}_{M,k}})(x):=\sum^{p_k}_{j=1}\langle Q^{(1)}_{M,k}g_j(x)| Q^{(1)}_{M,k}g_j(x)\rangle_{h^{L^k}},
\end{split}
\end{equation} 
where $f_j(x)\in C^\infty(X, L^k)$, $j=1,\ldots,m_k$, is an orthonormal frame for the space $\cH^0_{b,\,\leq\lambda}(X, L^k)$ with respect to $(\,\cdot\,|\,\cdot\,)_{h^{L^k}}$, $g_j(x)\in\Omega^{0,1}(X, L^k)$, $j=1,\ldots,p_k$, is an orthonormal frame for $\cH^1_{b,\,\leq\lambda}(X, L^k)$ with respect to $(\,\cdot\,|\,\cdot\,)_{h^{L^k}}$. It is straightforward to see that the definitions \eqref{epcIII} are independent of the choices of orthonormal frames. The point of our proof is that
there exists a sequence $\nu_k>0$ with $\nu_k\To0$ as $k\To\infty$, such that 
\begin{equation}\label{epcIII-I}
\mbox{For each $x\in X$, $\lim_{k\To\infty}k^{-n}(Q^{(0)}_{M,k}\pit^{(0)}_{k,\leq k\nu_k})(x)$ exists and is real valued}, 
\end{equation}
\begin{equation}\label{epcIV}
\begin{split}
&\lim_{k\To\infty}k^{-n}(Q^{(0)}_{M,k}\pit^{(0)}_{k,\leq k\nu_k})(x)\\
&\quad\geq(2\pi)^{1-n}\int\psi(\xi)\det(M^\phi_x+2\xi\mathcal{L}_x)\mathds{1}_{\Real_{x,0}}(\xi)d\xi-\frac{C_1}{M^2},\ \ \forall x\in X,
\end{split}
\end{equation}
\begin{equation}\label{epcV}
\limsup_{k\To\infty}k^{-n}(Q^{(1)}_{M,k}\pit^{(1)}_{k,\leq k\nu_k}\ol{Q^{(1)}_{M,k}})(x)\leq\frac{C_1}{M^2},\ \ \forall x\in X,
\end{equation} 
and 
\begin{equation}\label{epcV-I}
\begin{split}
&\sup\set{k^{-n}\abs{(Q^{(0)}_{M,k}\pit^{(0)}_{k,\leq k\nu_k})(x)};\, k>0, x\in X}<\infty,\\ 
&\sup\set{k^{-n}(Q^{(1)}_{M,k}\pit^{(0)}_{k,\leq k\nu_k}\ol{Q^{(1)}_{M,k}})(x);\, k>0, x\in X}<\infty,
\end{split}
\end{equation}
where $\Real_{x,0}$ is given by \eqref{e-seaI} and $\mathds{1}_{\Real_{x,0}}(\xi)=1$ if $\xi\in\Real_{x,0}$, 
$\mathds{1}_{\Real_{x,0}}(\xi)=0$ if $\xi\notin\Real_{x,0}$ and $C_1>0$ is a constant independent of $k$ and $M$.

From \eqref{epcV-I}, we can apply Lebesgue dominate Theorem and Fatou's lemma and  we get by using \eqref{epcIV} and \eqref{epcV}, 
\begin{equation}\label{epcVI}
\begin{split}
&\abs{\int_X(Q^{(0)}_{M,k}\pit^{(0)}_{k,\leq k\nu_k})(x)dv_X(x)}\\
&\geq k^n\Bigr((2\pi)^{1-n}\int_X\bigr(\int\psi(\xi)\det(M^\phi_x+2\xi\mathcal{L}_x)\mathds{1}_{\Real_{x,0}}(\xi)d\xi\bigr)dv_X(x)-\frac{C_2}{M^2}\Bigr)+o(k^n),
\end{split}
\end{equation}
\begin{equation}\label{epcVII}
\int_X(Q^{(1)}_{M,k}\pit^{(1)}_{k,\leq k\nu_k}\ol{Q^{(1)}_{M,k}})(x)dv_X(x)\leq k^n\frac{C_2}{M^2}+o(k^n),
\end{equation}
where $C_2>0$ is a constant independent of $M$ and $k$. 

Let
$f_{1,k},f_{2,k},\ldots,f_{d_k,k}$
be an orthonormal basis for $\cH^0_b(X,L^k)$, where $d_{k}={\rm dim\,}\cH^0_{b}(X,L^k)$. 
Let $\Td f_{1,k},\Td f_{2,k},\ldots,\Td f_{n_k,k}$
be an orthonormal basis for the space $\cH^0_{b,0<\mu\leq k\nu_k}(X,L^k)$. From \eqref{epcVI} and \eqref{epcIII}, we see that if $M$ is large enough, then
\begin{equation}\label{epcVIII}
\begin{split}
&\sum^{d_k}_{j=1}\abs{\int_X\langle Q^{(0)}_{M,k}f_{j,k}|f_{j,k}\rangle_{h^{L^k}}(x)dv_X(x)}+\sum^{n_k}_{j=1}\abs{\int_X\langle Q^{(0)}_{M,k}\Td f_{j,k}|\Td f_{j,k}\rangle_{h^{L^k}}(x)dv_X(x)}\\
&\quad\geq\frac{k^n}{2}(2\pi)^{1-n}\int_X\Bigr(\int\psi(\xi)\det(M^\phi_x+2\xi\mathcal{L}_x)
\mathds{1}_{\Real_{x,0}}(\xi)d\xi\Bigr)dv_X(x)
\end{split}
\end{equation}
for $k$ large. From \eqref{epcI} and \eqref{epcIII}, it is not difficult to check that 
\begin{equation}\label{epcIX}
\begin{split}
&\sum^{n_k}_{j=1}\abs{\int _X\langle Q^{(0)}_{M,k}\Td f_{j,k}|\Td f_{j,k}\rangle_{h^{L^k}}(x)dv_X(x)}\\
&\leq\Bigr(\int_X(Q^{(1)}_{M,k}\pit^{(1)}_{k,\leq k\nu_k}\ol{Q^{(1)}_{M,k}})(x)dv_X(x)\Bigr)^{\frac{1}{2}}\Bigr(\sum^{n_k}_{j=1}\int_X\langle\Td f_{j,k}|\Td f_{j,k}\rangle_{h^{L^k}}(x)dv_X(x)\Bigr)^{\frac{1}{2}}.
\end{split}
\end{equation}
It is well-known (see~\cite{HM09}) that
\begin{equation}\label{epcVIII-I}
\sup\set{k^{-n}\sum^{n_k}_{j=1}\langle\Td f_{j,k}|\Td f_{j,k}\rangle_{h^{L^k}}(x);\, k>0, x\in X}<\infty.
\end{equation}
From \eqref{epcVIII-I}, \eqref{epcVII}, \eqref{epcIX} and \eqref{epcVIII}, it is straightforward to see that if $M$ 
is large enough, then 
\begin{equation}\label{epcX}
\begin{split}
&\sum^{d_k}_{j=1}\abs{\int_X\langle Q^{(0)}_{M,k}f_{j,k}|f_{j,k}\rangle_{h^{L^k}}(x)dv_X(x)}\\
&\quad\geq\frac{k^n}{4}(2\pi)^{1-n}\int_X\Bigr(\int\psi(\xi)\det(M^\phi_x+2\xi\mathcal{L}_x)
\mathds{1}_{\Real_{x,0}}(\xi)d\xi\Bigr)dv_X(x)
\end{split}
\end{equation}
for $k$ large. Moreover, it is straightforward to see that there is a constant $C_M>0$ independent of $k$ such that 
$\abs{\int_X\langle Q^{(0)}_{M,k}u|u\rangle_{h^{L^k}}(x)dv_X(x)}\leq C_M\int_X\langle u|u\rangle_{h^{L^k}}(x)dv_X(x)$, for all $u\in C^\infty(X,L^k)$.
Combining this with \eqref{epcX}, we have
\[\begin{split}
C_Md_k&=C_M\sum^{d_k}_{j=1}\int_X\langle f_{j,k}|f_{j,k}\rangle_{h^{L^k}}(x)dv_X(x)
\geq\sum^{d_k}_{j=1}\abs{\int_X\langle Q^{(0)}_{M,k}f_{j,k}|f_{j,k}\rangle_{h^{L^k}}(x)dv_X(x)}\\
&\geq\frac{k^n}{4}(2\pi)^{1-n}\int_X\Bigr(\int\psi(\xi)\det(M^\phi_x+2\xi\mathcal{L}_x)
\mathds{1}_{\Real_{x,0}}(\xi)d\xi\Bigr)dv_X(x).\end{split}\]
Theorem~\ref{main} follows. 

The paper is organized as follows. In section 2, we review the results in~\cite{HM09} about the asymptotic behaviour of the Szeg\"{o} kernel for lower energy forms in order to prove \eqref{epcIII-I}, \eqref{epcIV} and \eqref{epcV-I}. We introduce 
the extremal function for the space of lower energy forms with respect to a given continuous operator and 
relate it to the function $Q^{(1)}_{M,k}\pit^{(1)}_{k,\leq\lambda}\ol{Q^{(1)}_{M,k}}$ (see Lemma~\ref{s2-l1}). 
This result will be used in the proof of \eqref{epcV}. In section 3, we introduce canonical coordinates on generalized Sasakian CR manifolds and prove that locally we can always find canonical coordinates and local section such that the corresponding local weight has a simple form (see Proposition~\ref{can-p0}). Canonical coordinates will be used in the constructions of the operaors $Q^{(0)}_{M,k}$ and $Q^{(1)}_{M,k}$ and Proposition~\ref{can-p0} will be used in section 4 and the proofs of \eqref{epcIII-I}, \eqref{epcIV} and \eqref{epcV}. In section 4, we modify the scaling technique developed in~\cite{HM09} in order to establish the semi-classical Kohn estimates(see Propositions~\ref{scat-p2}) and a result about the asymptotic behaviour of a sequence of forms with small energy(see Proposition~\ref{scat-p3}). These results play important roles in the proofs of \eqref{epcIII-I}, \eqref{epcIV} and \eqref{epcV}. In section 5, we construct the operators $Q^{(0)}_{M,k}$ and $Q^{(1)}_{M,k}$. 
In section 6, we prove \eqref{epcV-I}, \eqref{epcIII-I}, \eqref{epcIV} and \eqref{epcVI}. In section 7, we prove \eqref{epcV} and \eqref{epcVII}. In section 8, 
we first prove the inequality \eqref{epcIX} and then we complete the proof of Theorem~\ref{main}.
In section 9, we exemplify our main result in two concrete examples, one of a quotient of the Heisenberg group 
and the other of a Grauert tube over the torus. 

\smallskip

\noindent
{\small\emph{
\textbf{Acknowledgements.} The author is grateful to Prof. Rapha\"el Ponge for comments and useful suggestions on an early draft of the manuscript.}}

\section{Szeg\"{o} kernels for lower energy forms} 

We will use the same notations as section 1. From now on we assume that $(L,J,h^L)$ is a rigid generalized Sasakian CR line bundle over $X$.

The Hermitian fiber metric on $L$ induces a Hermitian
fiber metric on $L^k$ that we shall denote by $h^{L^k}$. If $s$ is a local trivializing section
of $L$ then $s^k$ is a local trivializing section of $L^k$.
The Hermitian metrics $\langle\,\cdot\,|\,\cdot\,\rangle$ on $\Lambda^{0,q}T^*X$ and $h^{L^k}$ induce
Hermitian metrics on $\Lambda^{0,q}T^*X\otimes L^k$. We shall denote these induced metrics by $\langle\,\cdot\,|\,\cdot\,\rangle_{h^{L^k}}$.
For $f\in\Omega^{0,q}(X, L^k)$, we denote the pointwise norm $\abs{f(x)}^2_{h^{L^k}}:=\langle f(x)|f(x)\rangle_{h^{L^k}}$. As \eqref{s1-e7}, let 
\begin{equation} \label{s2-e1}
\ddbar_{b,k}:\Omega^{0,q}(X, L^k)\To\Omega^{0,q+1}(X, L^k)
\end{equation}
denote the tangential Cauchy-Riemann operator acting on forms with values in $L^k$.
We denote by $dv_X=dv_X(x)$ the volume form on $X$ induced by the fixed 
Hermitian metric $\langle\,\cdot\,|\,\cdot\,\rangle$ on $\Complex TX$. Then we get natural global $L^2$ inner products $(\ |\ )_{h^{L^k}}$, $(\ |\ )$
on $\Omega^{0,q}(X, L^k)$ and $\Omega^{0,q}(X)$ respectively. We denote by $L^2_{(0,q)}(X, L^k)$ the completion of $\Omega^{0,q}(X, L^k)$ with respect to $(\ |\ )_{h^{L^k}}$. For $f\in\Omega^{0,q}(X,L^k)$, we denote $\norm{f}^2_{h^{L^k}}:=(f\ |\ f)_{h^{L^k}}$. Similarly, for $f\in\Omega^{0,q}(X)$, we denote $\norm{f}^2:=(f\ |\ f)$. Let
\begin{equation} \label{s2-e2}
\ol{\pr}^{\,*}_{b,k}:\Omega^{0,q+1}(X, L^k)\To\Omega^{0,q}(X, L^k)
\end{equation}
be the formal adjoint of $\ddbar_{b,k}$ with respect to $(\ |\ )_{h^{L^k}}$\,. The \emph{Kohn-Laplacian} with values in $L^k$ is given by
\begin{equation} \label{s2-e3}
\Box_{b,k}^{(q)}=\ol{\pr}^{\,*}_{b,k}\ddbar_{b,k}+\ddbar_{b,k}\ol{\pr}^{\,*}_{b,k}:
\Omega^{0,q}(X, L^k)\To\Omega^{0,q}(X, L^k).
\end{equation} 
We extend 
$\ddbar_{b,k}$ to $L^2_{(0,r)}(X,L^k)$, $r=0,1,\ldots,n-1$, by 
\begin{equation}\label{e:ddbar1}
\ddbar_{b,k}:{\rm Dom\,}\ddbar_{b,k}\subset L^2_{(0,r)}(X, L^k)\To L^2_{(0,r+1)}(X, L^k)\,,
\end{equation}
where ${\rm Dom\,}\ddbar_{b,k}:=\{u\in L^2_{(0,r)}(X, L^k);\, \ddbar_{b,k}u\in L^2_{(0,r+1)}(X, L^k)\}$, where for any $u\in L^2_{(0,r)}(X,L^k)$, $\ddbar_{b,k} u$ is defined in the sense of distribution. 
We also write 
\begin{equation}\label{e:ddbar_star1}
\ol{\pr}^{\,*}_{b,k}:{\rm Dom\,}\ol{\pr}^{\,*}_{b,k}\subset L^2_{(0,r+1)}(X, L^k)\To L^2_{(0,r)}(X, L^k)
\end{equation}
to denote the Hilbert space adjoint of $\ddbar_{b,k}$ in the $L^2$ space with respect to $(\ |\ )_{h^{L^k}}$.
Let $\Box^{(q)}_{b,k}$ also denote the Gaffney extension of the Kohn Laplacian given by 
\begin{equation}\label{Gaf1}
\begin{split}
{\rm Dom\,}\Box^{(q)}_{b,k}=&\{s\in L^2_{(0,q)}(X,L^k);\, s\in{\rm Dom\,}\ddbar_{b,k}\cap{\rm Dom\,}\ol{\pr}^{\,*}_{b,k},\\
 &\quad\ddbar_{b,k}u\in{\rm Dom\,}\ol{\pr}^{\,*}_{b,k},\ \ol{\pr}^{\,*}_{b,k}u\in{\rm Dom\,}\ddbar_{b,k}\}\,,
 \end{split}
\end{equation}
and $\Box^{(q)}_{b,k}s=\ddbar_{b,k}\ol{\pr}^{\,*}_{b,k}s+\ol{\pr}^{\,*}_{b,k}\ddbar_{b,k}s$ for $s\in {\rm Dom\,}\Box^{(q)}_{b,k}$. We notice that $\Box^{(q)}_{b,k}$ is a positive self-adjoint operator. 
For a Borel set $B\subset\Real$ we denote by $E(B)$ the spectral projection of $\Box^{(q)}_{b,k}$ corresponding to the set $B$, where $E$ is the spectral measure of $\Box^{(q)}_{b,k}$ (see section 2 in Davies~\cite{Dav95}, for the precise meanings of spectral projection and spectral measure). We notice that the spectrum of $\Box^{(q)}_{b,k}$ is contained in $\ol\Real_+$. For $\lambda\geq0$, we set 
\begin{equation} \label{s1-specsp}
\begin{split}
&\cH^q_{b,\leq\lambda}(X, L^k):=\operatorname{Range}E\big((-\infty,\lambda]\big)\subset L^2_{(0,q)}(X,L^k)\,,\\
&\cH^q_{b,>\lambda}(X,L^k):=\operatorname{Range}E\big((\lambda,\infty)\big)\subset L^2_{(0,q)}(X,L^k).
\end{split}
\end{equation}  
It is well-known (see section 2 in~\cite{Dav95}) that for all $\lambda>0$,
\begin{equation} \label{s5-sp1}
L^2_{(0,q)}(X,L^k)=\cH^q_{b,\leq\lambda}(X, L^k)\oplus\cH^q_{b,>\lambda}(X, L^k)
\end{equation} 
and 
\begin{equation} \label{s5-sp2}
\norm{u}^2_{h^{L^k}}\leq\frac{1}{\lambda}(\Box^{(q)}_{b,k}u\ |\ u)_{h^{L^k}},\ \ \forall u\in\cH^q_{b,>\lambda}(X,L^k)\bigcap{\rm Dom\,}\Box^{(q)}_{b,k}.
\end{equation}
For $\lambda=0$, we denote 
\begin{equation} \label{s2-e4}
\cH_b^q(X, L^k):=\cH^q_{b,\leq0}(X,L^k)=\Ker\Box^{(q)}_{b,k}\,.
\end{equation}
Now, fix $q\in\set{0,1,\ldots,n-1}$ and until further notice we assume that $Y(q)$ holds. By \cite[7.6-7.8]{Ko65}, \cite[5.4.11-12]{FK72}, \cite[Props.\,8.4.8-9]{CS01}, we know that
$\Box^{(q)}_{b,k}$ is hypoelliptic, has compact resolvent, the strong Hodge decomposition holds and
$\Box^{(q)}_{b,k}$ has a discrete
spectrum, each eigenvalues occurs with finite multiplicity and all eigenforms are smooth. Hence, for any $\lambda\geq0$, 
\begin{equation} \label{s2-e5}
\dim\cH_{b,\leq\lambda}^q(X, L^k)<\infty,\ \cH_{b,\leq\lambda}^q(X, L^k)\subset\Omega^{0,q}(X, L^k),\ \cH_b^q(X, L^k)\cong H_b^q(X, L^k).
\end{equation} 
Let $g_j(x)\in\Omega^{0,q}(X, L^k)$, $j=1,\ldots,d_k$, $d_k={\rm dim\,}\cH^q_{b,\,\leqslant \lambda}(X,L^k)$, be 
any orthonormal frame for the space $\cH_{b,\,\leq\lambda}^q(X, L^k)$ with respect to $(\,\cdot\,|\,\cdot\,)_{h^{L^k}}$.  The Szeg\"{o} kernel function $\pit^{(q)}_{k,\leq\lambda}(x)$ of
the space $\cH^q_{b,\leqslant \lambda}(X,L^k)$ is given by 
\begin{equation}\label{e-seI} 
\pit^{(q)}_{k,\,\leqslant \lambda}(x):=\sum^{d_k}_{j=1}\abs{g_j(x)}^2_{h^{L^k}}.
\end{equation} 
Let 
\[A:\Omega^{0,q}(X,L^k)\To\Omega^{0,q}(X,L^k)\]
be a continuous operator. We define 
\begin{equation}\label{e-seII} 
(A\pit^{(q)}_{k,\,\leqslant \lambda})(x):=\sum^{d_k}_{j=1}\langle Ag_j(x)| g_j(x)\rangle_{h^{L^k}}, 
\end{equation} 
\begin{equation}\label{e-seIII} 
(A\pit^{(q)}_{k,\,\leqslant \lambda}\ol A)(x):=\sum^{d_k}_{j=1}\abs{Ag_j(x)}^2_{h^{L^k}}.
\end{equation} 
It is straightforward to see that the definitions \eqref{e-seI}, \eqref{e-seII} and \eqref{e-seIII} are independent of the choices of orthonormal frame $g_j$, $j=1,\ldots,d_k$. 

For $q=0,1,\ldots,n-1$ and $x\in X$, set
\begin{equation}\label{e-seaI}
\begin{split}
\Real_{x,q}=\big\{s\in\Real;\, \text{$M^\phi_x+2s\mathcal{L}_x$ has exactly $q$ negative eigenvalues} \\
\text{and $n-1-q$ positive eigenvalues}\big\}\,,
\end{split}
\end{equation}
where $M^\phi_x$ is given by \eqref{s1-e14} and the eigenvalues of the Hermitian quadratic form $M^\phi_x+2s\mathcal{L}_x$, $s\in\Real$,
are calculated with respect to the Hermitian metric $\langle\,\cdot\,|\,\cdot\,\rangle$.
It is not difficult to see that if $Y(q)$ holds at each point of $X$ then there is a constant $C>0$ such that
\begin{equation} \label{e-seaII}
\mbox{$\Real_{x,\,q}\subset[-C,C]$ for all $x\in X$.}
\end{equation}
Denote by $\det(M^\phi_x+2s\mathcal{L}_x)$ the product of all the eigenvalues of $M^\phi_x+2s\mathcal{L}_x$. 
Assuming \eqref{e-seaII} holds, the function
\begin{equation} \label{e-seaIII}
X\longrightarrow{\Real}\,,\quad x\longmapsto\int_{\Real_{x,q}}\abs{\det(M^\phi_x+2s\mathcal{L}_x)}ds
\end{equation}
is well-defined.
Since $M^\phi_x$ and $\mathcal{L}_x$ are continuous functions of $x\in X$, we conclude that the function \eqref{e-seaIII}
is continuous.

The following is well-known (see Theorem 1.6 in Hsiao-Marinescu~\cite{HM09})

\begin{thm}\label{s2-t1}
Assume that condition $Y(q)$ holds at each point of $X$. Then
for any sequence $\nu_k>0$ with $\nu_k\To0$ as $k\To\infty$, there is a constant $C_0>0$ independent of $k$,
such that
\begin{equation} \label{s2-e6}
k^{-n}\pit^{(q)}_{k,\,\leqslant  k\nu_k}(x)\leqslant  C_0
\end{equation}
for all $x\in X$. Moreover, there is a sequence $\mu_k>0$, $\mu_k\To0$ as $k\To\infty$, such that for any sequence
$\nu_k>0$ with $\lim_{k\To\infty}\frac{\mu_k}{\nu_k}=0$ and $\nu_k\To0$ as $k\To\infty$, we have
\begin{equation} \label{s2-e6I}
\lim_{k\To\infty}k^{-n}\pit^{(q)}_{k,\,\leqslant  k\nu_k}(x)=(2\pi)^{-n}\int_{\Real_{x,q}}\abs{\det(M^\phi_x+2s\mathcal{L}_x)}ds,
\end{equation}
for all $x\in X$.
\end{thm} 

We introduce some notations.
For $p\in X$, we can choose a smooth orthonormal frame
$e_1,\ldots,e_{n-1}$ of $T^{*0,1}X$ over a neighborhood $U$ of $p$. We say that a multiindex $J=(j_1,\ldots,j_q)\in\{1,\ldots,n-1\}^q$ has length $q$ and write $\abs{J}=q$. We say that $J$ is strictly increasing if $1\leqslant  j_1<j_2<\cdots<j_q\leqslant  n-1$. For $J=(j_1,\ldots,j_q)$ we define $e_J:=e_{j_1}\wedge\cdots\wedge e_{j_q}$.
Then $\set{e_J;\, \mbox{$\abs{J}=q$, $J$ strictly increasing}}$
is an orthonormal frame for $\Lambda^{0,q}T^*X$ over $U$.

For $f\in\Omega^{0,q}(X, L^k)$, we may write
\[
f|_U=\sideset{}{'}\sum_{\abs{J}=q} f_Je_J\,,\quad \text{with $f_J=\langle f|e_J\rangle\in C^\infty(U,L^k)$}\,,
\]
where $\sum'$ means that the summation is performed only over strictly increasing multiindices. We call $f_J$ the component of $f$ along $e_J$. It will be clear from the context what frame is being used. For $q>0$, the \emph{extremal function} $S^{(q)}_{k,\leq\lambda,J}$ for the space $\cH^q_{b,\leq\lambda}(X,L^k)$ along the direction $e_J$ is defined by
\begin{equation} \label{s2-e7}
S^{(q)}_{k,\leq\lambda,J}(y)=\sup_{\alpha\in\,\cH_{b,\leq\lambda}^q(X, L^k),\,\norm{\alpha}_{h^{L^k}}=1}\abs{\alpha_J(y)}^2_{h^{L^k}}\,,
\end{equation} 
where $\alpha_J$ denotes the component of $\alpha$ along $e_J$. 
Let 
\[A:\Omega^{0,q}(X,L^k)\To\Omega^{0,q}(X,L^k)\]
be a continuous operator. For $\abs{J}=q$, $J$ is strictly increasing, we define 
\begin{equation}\label{s2-e7-1} 
(AS^{(q)}_{k,\,\leqslant \lambda,J}\ol A)(y):=\sup_{\alpha\in\,\cH_{b,\leq\lambda}^q(X, L^k),\,\norm{\alpha}_{h^{L^k}}=1}\abs{(A\alpha)_J(y)}^2_{h^{L^k}}\,,
\end{equation} 
where $(A\alpha)_J$ denotes the component of $A\alpha$ along $e_J$.
Similarly, when $q=0$, we define
\begin{equation} \label{s2-e8}
\begin{split}
&S^{(0)}_{k,\leq\lambda}(y)=\sup_{\alpha\in\,\cH_{b,\leq\lambda}^0(X, L^k),\,\norm{\alpha}_{h^{L^k}}=1}\abs{\alpha(y)}^2_{h^{L^k}}, \\
&(AS^{(0)}_{k,\,\leqslant \lambda}\ol A)(y):=\sup_{\alpha\in\,\cH_{b,\leq\lambda}^0(X, L^k),\,\norm{\alpha}_{h^{L^k}}=1}\abs{(A\alpha)(y)}^2_{h^{L^k}}.
\end{split}
\end{equation} 

We need the following
\begin{lem} \label{s2-l1}
Fix $\lambda\geq0$. Let $A:\Omega^{0,q}(X,L^k)\To\Omega^{0,q}(X,L^k)$
be a continuous operator. For every local orthonormal frame 
\[\set{e_J(y); \mbox{$\abs{J}=q$,\,$J$ strictly increasing}}\]
of $\Lambda^{0,q}T^*X$ over an open set $U\subset X$,
we have when $q>0$,
\begin{equation}\label{s2-e9}
(A\pit^{(q)}_{k,\leq\lambda}\ol A)(y)=\sideset{}{'}\sum_{\abs{J}=q}(AS^{(q)}_{k,\leq\lambda,J}\ol A)(y),
\end{equation}
for every $y\in U$. 

Similarly, when $q=0$, we have  
\begin{equation}\label{s2-e10}
(A\pit^{(0)}_{k,\leq\lambda}\ol A)(y)=(AS^{(0)}_{k,\leq\lambda}\ol A)(y),
\end{equation}
for every $y\in U$. 

We remind that $A\pit^{(q)}_{k,\leq\lambda}\ol A$ is given by \eqref{e-seIII}.
\end{lem}

\begin{proof}
Let $(f_j)_{j=1,\ldots,d_k}$ be an orthonormal frame for the space
$\cH_{b,\leq\lambda}^q(X, L^k)$. Let $s$ be a local section of $L$ on $U$, $\abs{s}^2_{h^L}=e^{-\phi}$. 
On $U$, we write 
\[\begin{split}
&Af_j=s^k\Td g_j,\ \ \Td g_j\in\Omega^{0,q}(U),\ \ j=1,\ldots,d_k,\\
&\Td g_j=\sideset{}{'}\sum_{\abs{J}=q}\Td g_{j,J}e_J,\ \ j=1,\ldots,d_k.
\end{split}\]
On $U$ we write 
\begin{equation}\label{eallerI}
(A\pit^{(q)}_{k,\leq\lambda}\ol A)(y)=\sideset{}{'}\sum_{\abs{J}=q}(A\pit^{(q)}_{k,\leq\lambda,J}\ol A)(y),
\end{equation}
where 
\[(A\pit^{(q)}_{k,\leq\lambda,J}\ol A)(y):=e^{-\phi(y)}\sum_j\abs{\Td g_{j,J}(y)}^2.\]
It is easy to see that
$(A\pit^{(q)}_{k,\leq\lambda,J}\ol A)(y)$ is independent of the choice of the orthonormal frame $(f_j)_{j=1,\ldots,d_k}$. Take
$\alpha\in\cH^q_{b,\leq\lambda}(X, L^k)$ of unit norm. Since $\alpha$ is contained in an orthonormal base, obviously $|(A\alpha)_J(y)|^2_{h^{L^k}}\leqslant(A\pit^{(q)}_{k,\leq\lambda,J}\ol A)(y)$, where $(A\alpha)_J$ denotes the component of $A\alpha$ along $e_J$. Thus,
\begin{equation} \label{s2-e11}
(AS^{(q)}_{k,\leq\lambda,J}\ol A)(y)\leqslant(A\pit^{(q)}_{k,\leq\lambda,J}\ol A)(y)\,,\quad\text{for all strictly increasing $J$, $\abs{J}=q$.}
\end{equation}
Fix a point $p\in U$ and a strictly increasing multiindex $J$ with $\abs{J}=q$. We may assume that 
$\sum^{d_k}_{j=1}\abs{\Td g_{j,J}(p)}^2\neq0$. Put
\[
\textstyle
u(y)=\Big(\sum^{d_k}_{j=1}\abs{\Td g_{j,J}(p)}^2\Big)^{-1/2}\cdot\sum^N_{j=1}\ol{\Td g_{j,J}(p)}f_j(y)\,.
\]
We can easily check that $u\in\cH^q_{b,\leq\lambda}(X, L^k)$ and $\norm{u}_{h^{L^k}}=1$. Hence, 
\[\abs{(Au)_{J}(p)}^2_{h^{L^k}}\leqslant(AS^{(q)}_{k,\leq\lambda,J}\ol A)(p),\] 
therefore
\[
(A\pit^{(q)}_{k,\leq\lambda,J}\ol A)(p)=\sum^{d_k}_{j=1}e^{-\phi(p)}\abs{\Td g_{j,J}(p)}^2=|(Au)_{J}(p)|^2_{h^{L^k}}\leqslant(AS^{(q)}_{k,\leq\lambda,J}\ol A)(p)\,.
\]
From this and \eqref{s2-e11}, we conclude that $A\pit^{(q)}_{k,\leq\lambda,J}\ol A=AS^{(q)}_{k,\leq\lambda,J}\ol A$ for all strictly increasing multiindices $J$ with $\abs{J}=q$. Combining this with \eqref{eallerI}, \eqref{s2-e9} follows.

The proof of \eqref{s2-e10} is the same. The lemma follows.
\end{proof}

\section{Canonical coordinates of generalized Sasakian CR manifolds} 

In this work, we need the following beautiful result due to Baouendi-Rothschild-Treves~\cite[section1]{BRT85} 

\begin{thm}\label{t-can}
We recall that we work with the assumption that $(X,T^{1,0}X)$ is a generalized Sasakian CR manifold and we fix a rigid global real vector field $T=J\frac{\pr}{\pr t}$. 
For every point $x_0\in X$, there exists local coordinates $x=(x_1,\ldots,x_{2n-1})=(z,\theta)=(z_1,\ldots,z_{n-1},\theta)$, 
$z_j=x_{2j-1}+ix_{2j}$, $j=1,\ldots,n-1$, $\theta=x_{2n-1}$, defined in some small neighborhood $U$ of $x_0$ such that 
\begin{equation}\label{e-can}
\begin{split}
&T=\frac{\pr}{\pr\theta},\\
&Z_j=\frac{\pr}{\pr z_j}+i\frac{\pr\varphi}{\pr z_j}(z)\frac{\pr}{\pr\theta},\ \ j=1,\ldots,n-1,
\end{split}
\end{equation}
where $Z_j(x)$, $j=1,\ldots,n-1$, form a basis of $T^{1,0}_xX$, for each $x\in U$, and $\varphi(z)\in C^\infty(U,\Real)$
independent of $\theta$.
\end{thm} 

Let $x=(x_1,\ldots,x_{2n-1})$ be local coordinates of $X$ defined in some open set in $X$. In this paper, when we write 
$x=(x_1,\ldots,x_{2n-1})=(z,\theta)$ we mean that $z=(z_1,\ldots,z_{n-1})$, $z_j=x_{2j-1}+ix_{2j}$, $j=1,\ldots,n-1$, $\theta=x_{2n-1}$. We call $x$ canonical coordinates if $x$ satisfies \eqref{e-can}. 

We also need 

\begin{prop}\label{can-p0}
For a given point $p\in X$, we can find cannonical coordinates $x=(x_1,\ldots,x_{2n-1})=(z,\theta)$ and local section $s$, $\abs{s}^2_{h^L}=e^{-\phi}$, defined in some small neighborhood $D$ of $p$ such that 
\begin{equation}\label{s-canpe0}
\begin{split}
&x(p)=0,\\
&Z_j=\frac{\pr}{\pr z_j}+i\lambda_j\ol z_j\frac{\pr}{\pr\theta}+O(\abs{z}^2)\frac{\pr}{\pr\theta},\ \ j=1,\ldots,n-1,\\ 
&\mbox{$\frac{\pr}{\pr z_1},\ldots,\frac{\pr}{\pr z_{n-1}}$ is an orthonormal frame for $T^{1,0}_pX$},\\
&\phi(z,\theta)=\beta\theta+\sum^{n-1}_{j,t=1}\mu_{j,t}\ol z_j z_t+O(\abs{z}\abs{\theta})+O(\abs{\theta}^2)+O(\abs{(z,\theta)}^3),
\end{split}
\end{equation}
where $Z_1(x),\ldots, Z_{n-1}(x)$
form a basis of $T^{1, 0}_xX$ varying smoothly with $x$ in a neighborhood of\, $p$,
$\lambda_1,\ldots,\lambda_{n-1}$ are the eigenvalues of $\mathcal{L}_p$ with respect to $\langle\,\cdot\,|\,\cdot\,\rangle$, $\beta\in\Real$, $\mu_{j,t}\in\Complex$, $\mu_{j,t}=\ol{\mu_{t,j}}$, $j,t=1,\ldots,n-1$.
\end{prop} 

\begin{proof}
Fix $p\in X$. Let $\Td x=(\Td x_1,\ldots,\Td x_{2n-1})=(\Td z,\Td \theta)=(\Td z_1,\ldots,\Td z_{n-1},\Td \theta)$, $\Td z_j=\Td x_{2j-1}+i\Td x_{2j}$, $j=1,\ldots,n-1$, $\Td \theta=\Td x_{2n-1}$ be canonical coordinates 
of $X$ defined in some small neighborhood $D$ of $p$. We have 
\begin{equation}\label{e-can1}
\begin{split}
&T=\frac{\pr}{\pr\Td \theta},\\
&Z_j=\frac{\pr}{\pr\Td z_j}+i\frac{\pr\Td \varphi}{\pr\Td z_j}(\Td z)\frac{\pr}{\pr\Td\theta},\ \ j=1,\ldots,n-1,
\end{split}
\end{equation}
where $Z_j(\Td x)$, $j=1,\ldots,n-1$, form a basis of $T^{1,0}_{\Td x}X$, for each $\Td x\in D$, and $\Td\varphi(\Td z)\in C^\infty(D,\Real)$ independent of $\Td\theta$. It is easy to see that we can take $\Td x$ so that $\Td x(p)=0$. Near $p$, we write 
\begin{equation}\label{e-can2-1-0}
\Td\varphi(\Td z)=a+\sum^{n-1}_{j=1}(\alpha_j\Td z_j+\ol\alpha_j\ol{\Td z_j})+O(\abs{\Td z}^2),
\end{equation}
where $a\in\Complex$, $\alpha_j\in\Complex$, $j=1,\ldots,n-1$. Let $\hat z=\Td z$, $\hat \theta=\Td\theta-\sum^{n-1}_{j=1}(i\alpha_j\Td z_j-i\ol\alpha_j\ol{\Td z_j})$. Then, $(\hat z,\hat\theta)$ form canocical coordinates of $X$ near $p$ and we can check that 
\begin{equation}\label{e-canr4}
\begin{split}
&\frac{\pr}{\pr\hat\theta}=\frac{\pr}{\pr\Td\theta},\\
&\frac{\pr}{\pr\hat z_j}=\frac{\pr}{\pr\Td z_j}+i\alpha_j\frac{\pr}{\pr\Td\theta},\ \ j=1,\ldots,n-1.
\end{split}
\end{equation} 
From \eqref{e-canr4}, \eqref{e-can2-1-0} and \eqref{e-can1}, we see that
\begin{equation}\label{e-can4}
\begin{split}
&T=\frac{\pr}{\pr\hat\theta},\\
&Z_j=\frac{\pr}{\pr\hat z_j}+i\frac{\pr\hat\varphi}{\pr\hat z_j}(\hat z)\frac{\pr}{\pr\hat\theta},\ \ j=1,\ldots,n-1,
\end{split}
\end{equation}
where 
\[\hat\varphi(\hat z)=\Td\varphi(\hat z)-\sum^{n-1}_{j=1}(\alpha_j\hat z_j+\ol\alpha_j\ol{\hat z_j})
=a+O(\abs{\hat z}^2).\] 
Thus, $\frac{\pr}{\pr\hat z_1},\ldots,\frac{\pr}{\pr\hat z_{n-1}}$ is a basis of $T^{1,0}_pX$. By taking some linear transformation, we can take $\hat z$ so that $\frac{\pr}{\pr\hat z_j}$, $j=1,\ldots,n-1$, is an orthonormal 
frame for $T^{1,0}_pX$ and the Levi form is diagonal at $p$ with respect to $\frac{\pr}{\pr\hat z_j}$, $j=1,\ldots,n-1$. We write 
\begin{equation}\label{eTW0}
\hat\varphi(\hat z)=\alpha+\sum^{n-1}_{j,t=1}(\beta_{j,t}\hat z_j\hat z_t+\ol\beta_{j,t}\ol{\hat z_j}\,\ol{\hat z_t})+\sum^{n-1}_{j,t=1}\gamma_{j,t}\ol{\hat z_j}\hat z_t+O(\abs{\hat z}^3),
\end{equation}
where $\beta_{j,t}\in\Complex$, $\gamma_{j,t}\in\Complex$, $\gamma_{j,t}=\ol{\gamma_{t,j}}$, $j,t=1,\ldots,n-1$. 
Since the Levi form is diagonal at $p$ with respect to $\frac{\pr}{\pr\hat z_j}$, $j=1,\ldots,n-1$, we can check that
\begin{equation}\label{eTW1}
\gamma_{j,t}=\lambda_j\delta_{j,t},\ \ j,t=1,\ldots,n-1,
\end{equation}
where $\lambda_1,\ldots,\lambda_{n-1}$ are the eigenvalues of $\mathcal{L}_p$ with respect to $\langle\,\cdot\,|\,\cdot\,\rangle$.
Let $z=\hat z$, $\theta=\hat\theta-\sum^{n-1}_{j,t=1}i(\beta_{j,t}\hat z_j\hat z_t-\ol\beta_{j,t}\ol{\hat z_j}\,\ol{\hat z_t})$. Then, $(z,\theta)$ form canonical coordinates of $X$ near $p$ and we can check that 
\begin{equation}\label{eTW2}
\begin{split}
&\frac{\pr}{\pr\theta}=\frac{\pr}{\pr\hat\theta},\\
&\frac{\pr}{\pr z_j}=\frac{\pr}{\pr\hat z_j}+i\sum^{n-1}_{t=1}\beta_{j,t}\hat z_t\frac{\pr}{\pr\hat\theta},\ \ j=1,\ldots,n-1.
\end{split}
\end{equation}

From \eqref{e-can4}, \eqref{eTW0}, \eqref{eTW1} and \eqref{eTW2}, we can check that 
\begin{equation*}
\begin{split}
&T=\frac{\pr}{\pr\theta},\\
&Z_j=\frac{\pr}{\pr z_j}+i\lambda_j\ol z_j\frac{\pr}{\pr\theta}+O(\abs{z}^2)\frac{\pr}{\pr\theta},\ \ j=1,\ldots,n-1.
\end{split}
\end{equation*}
Since $\frac{\pr}{\pr z_j}$, $j=1,\ldots,n-1$, is an orthonormal frame of $T^{1,0}_pX$, 
we conclude that $x=(z,\theta)$ satisfies the first three properties in \eqref{s-canpe0}. 

Let $\hat s$ be a local section defined in some neighborhood of $p$, $\abs{\hat s}^2_{h^L}=e^{-\hat\phi}$. Near $p$, we write 
\begin{equation}\label{e-can5s}
\hat\phi(z,\theta)=c+\beta\theta+\sum^{n-1}_{j=1}(a_jz_j+\ol a_j\ol z_j)+O(\abs{(z,\theta)}^2),
\end{equation}
where $c\in\Real$, $\beta\in\Real$ and $a_j\in\Complex$, $j=1,\ldots,n-1$. Let 
\begin{equation}\label{e-can6s}
g(z)=e^{\frac{c}{2}}(1+\sum^{n-1}_{j=1}a_jz_j).
\end{equation}
Then, $g(z)$ is a rigid CR function. We may replace $\hat s$ by $g\hat s:=\Td s$. We have 
\begin{equation}\label{e-can7s}
\abs{\Td s}^2_{h^L}=e^{-\Td\phi}=\abs{g}^2e^{-\hat\phi}=e^{2\log\abs{g}-\hat\phi}.
\end{equation} 
From \eqref{e-can6s}, we can check that
\[2\log\abs{g}=c+\sum^{n-1}_{j=1}(a_jz_j+\ol a_j\ol z_j)+O(\abs{z}^2).\] 
Combining this with \eqref{e-can7s} and \eqref{e-can5s}, we conclude that 
\[\Td\phi(z,\theta)=\beta\theta+O(\abs{(z,\theta)}^2).\] 
Near $p$, we write 
\begin{equation}\label{e-can8s}
\begin{split}
\Td\phi(z,\theta)&=\beta\theta+\sum^{n-1}_{j,t=1}(c_{j,t}z_jz_t+\ol{c_{j,t}}\,\ol z_j\ol z_t)+\sum^n_{j,t=1}\mu_{j,t}\ol z_jz_t\\
&\quad+O(\abs{z}\abs{\theta})+O(\abs{\theta}^2)+O(\abs{(z,\theta)}^3),
\end{split}
\end{equation}
where $c_{j,t}\in\Complex$, $\mu_{j,t}\in\Complex$, $\mu_{j,t}=\ol{\mu_{t,j}}$, $j,t=1,\ldots,n-1$. 
Let 
\begin{equation}\label{e-can9s}
g_1(z)=1+\sum^{n-1}_{j,t=1}c_{j,t}z_jz_t.
\end{equation}
Then, $g_1(z)$ is a rigid CR function. We may replace $\Td s$ by $g_1\Td s:=s$. We have 
\begin{equation}\label{e-can10s}
\abs{s}^2_{h^L}=e^{-\phi}=\abs{g_1}^2e^{-\Td\phi}=e^{2\log\abs{g_1}-\Td\phi}.
\end{equation} 
From \eqref{e-can9s}, we can check that
\[2\log\abs{g_1}=\sum^{n-1}_{j,t=1}(c_{j,t}z_jz_t+\ol{c_{j,t}}\,\ol z_j\ol z_t)+O(\abs{z}^3).\] 
Combining this with \eqref{e-can10s} and \eqref{e-can8s}, we conclude that 
\[\phi(z,\theta)=\beta\theta+\sum^{n-1}_{j,t=1}\mu_{j,t}\ol  z_jz_t+O(\abs{z}\abs{\theta})+O(\abs{\theta}^2)+O(\abs{(z,\theta)}^3).\]

The proposition follows.
\end{proof} 

\begin{prop}\label{ppcf} 
We assume that $X$ is a generalized torus CR manifold and $(L,J)$ is an admissible generalized Sasakian CR line bundle over $X$ 
(see Definition~\ref{dpc} and Definition~\ref{dpc1}).
Let $\phi$ and $\phi_1$ be as in the discussion after Definition~\ref{dpc1}. If $M^{\phi}_x$ is positive on $X$, then $M^{\phi_1}_x$ is positive on $X$.
\end{prop}

\begin{proof} 
Let $\set{W_1\subset W'_1,\ldots,W_N\subset W'_N}$ be open sets of $X$ such that $X=\bigcup^N_{j=1}W_j$ and there exist canonical coordinates on $W'_j$, for each $j$ and there is a constant $\epsilon_0>0$ such that for each $x\in X$, $\Phi^t(x)$ is well-defined, $\forall\abs{t}\leq\epsilon_0$, and 
\[\set{\Phi^t(x);\, x\in W_j,\abs{t}\leq\epsilon_0}\subset W'_j,\] 
for each $j$, where $\Phi^t(x)$ is the $T$-flow. Fix $t_0\in[-\epsilon_0,\epsilon_0]$. 
Put $\Td\phi(x)=\phi(\Phi^{t_0}x)$. It is obviously that $\Td\phi(x)$ also define a Hermitian fiber metric on $L$. By using 
canonical coordinates \eqref{e-can}, we can check that 
\[d(\ddbar_b\Td\phi-\pr_b\Td\phi)(x)=d(\ddbar_b\phi-\pr\phi)(\Phi^{t_0}x),\ \ \forall x\in X.\]
Thus, 
\[M^{\Td\phi}_x=M^\phi_{\Phi^{t_0}(x)},\ \ \forall x\in X.\]
Similarly, fix $t_1\in[-\epsilon_0,\epsilon_0]$ and put $\hat\phi(x)=\Td\phi(\Phi^{t_1}x)=\phi(\Phi^{t_0+t_1}x)$. We have \[M^{\hat\phi}_{x}=M^{\Td\phi}_{\Phi^{t_1}(x)}=M^{\phi}_{\Phi^{t_0+t_1}(x)},\ \ \forall x\in X.\]
Continuing in this way, we obtain for any $t\in[0,\gamma_0]$, $M^{\phi(\Phi^t(x))}_{x}=M^{\phi}_{\Phi^t(x)}$, $\forall x\in X$, where $\gamma_0$ is as in Definition~\ref{dpc}. Thus, 
\[M^{\phi_1}_x=\frac{1}{\gamma_0}\int^{\gamma_0}_0M^\phi_{\Phi^t(x)}dt,\ \ \forall x\in X.\]
The proposition follows.
\end{proof}

\section{The scaling technique} 

In this section we modify the scaling technique developed in~\cite{HM09} in order to prove \eqref{epcIII-I}, \eqref{epcIV} and \eqref{epcV}.  

Fix a point $p\in X$. Let $x=(x_1,\ldots,x_{2n-1})=(z,\theta)$ be canonical coordinates of $X$ defined in 
some small neighborhood $D$ of $p$ and let $s$ be a local section of $L$ on $D$, $\abs{s}^2_{h^L}=e^{-\phi}$. 
We take $x$ and $s$ so that 
\begin{equation}\label{scat-e1}
\begin{split}
&x(p)=0,\\
&Z_j=\frac{\pr}{\pr z_j}+i\lambda_j\ol z_j\frac{\pr}{\pr\theta}+O(\abs{z}^2)\frac{\pr}{\pr\theta},\ \ j=1,\ldots,n-1,\\
&\mbox{$\frac{\pr}{\pr z_1},\ldots,\frac{\pr}{\pr z_{n-1}}$ is an orthonormal frame for $T^{1,0}_pX$},\\
&\phi(z,\theta)=\beta\theta+\sum^{n-1}_{j,t=1}\mu_{j,t}\ol z_j z_t+O(\abs{z}\abs{\theta})+O(\abs{\theta}^2)+O(\abs{(z,\theta)}^3),
\end{split}
\end{equation}
where $Z_1(x),\ldots, Z_{n-1}(x)$
form a basis of $T^{1, 0}_xX$ varying smoothly with $x$ in a neighborhood of\, $p$, $\lambda_1,\ldots,\lambda_{n-1}$ are the eigenvalues of $\mathcal{L}_p$ with respect to $\langle\,\cdot\,|\,\cdot\,\rangle$, $\beta\in\Real$, $\mu_{j,k}\in\Complex$, $\mu_{j,t}=\ol{\mu_{t,j}}$, $j,t=1,\ldots,n-1$. 
By Proposition~\ref{can-p0}, this is always possible. Fix $q\in\set{0,1,\ldots,n-1}$. In this section, we work on 
$(0,q)$ forms and we work with this local coordinates $x=(z,\theta)$. 

Let $(\ |\ )_{k\phi}$ be the inner product on the space
$\Omega^{0,q}_0(D)$ defined as follows:
\[
(f\ |\ g)_{k\phi}=\int_D\!\langle f|g\rangle e^{-k\phi}dv_X,\]
where $f, g\in\Omega^{0,q}_0(D)$. Let
$\ddbar^{\,*,k\phi}_b:\Omega^{0,q+1}(D)\To\Omega^{0,q}(D)$
be the formal adjoint of $\ddbar_b$ with respect to $( \ |\ )_{k\phi}$. Put
\[
\Box^{(q)}_{b,k\phi}=\ddbar_b\ddbar^{\,*,k\phi}_b+\ddbar^{\,*,k\phi}_b\ddbar_b:\Omega^{0,q}(D)\To\Omega^{0,q}(D)\,.
\]
Let $u\in\Omega^{0,q}(D, L^k)$. Then there exists $\Td u\in\Omega^{0,q}(D)$ such that $u=s^k\Td u$ and we have
\begin{equation}\label{scat-e2}
\Box^{(q)}_{b,k}u=s^k\Box^{(q)}_{b,k\phi}\Td u.
\end{equation}

Let $U_1(z,\theta),\ldots,U_{n-1}(z,\theta)$ be an orthonormal frame of $T^{1,0}_{(z,\theta)}X$ varying smoothly 
with $(z,\theta)$ in a neighborhood of $p$. We take $U_1,\ldots,U_{n-1}$ so that $U_j(0,0)=\frac{\pr}{\pr z_j}$, 
$j=1,\ldots,n-1$. Put 
\begin{equation}\label{epch}
U_j(z,\theta)=\sum^{n-1}_{t=1}a_{j,t}(z,\theta)Z_t,\ \ j=1,\ldots,n-1,
\end{equation}
where $a_{j,t}\in C^\infty$, $j,t=1,\ldots,n-1$, $Z_1,\ldots,Z_{n-1}$ are as in \eqref{scat-e1}. Then, we have 
\begin{equation}\label{epchI}
a_{j,t}(z,\theta)=\delta_{j,t}+O(\abs{(z,\theta)}),\ \ j,t=1,\ldots,n-1.
\end{equation}
Let $(e_j(z, \theta))_{j=1,\ldots,n-1}$ denote the basis of $T^{*0,1}_{(z,\theta)}X$,
dual to $(\ol U_j(z,\theta))_{j=1,\ldots,n-1}$.
If $w\in T^{*0,1}_zX$, let
$(w\wedge)^*: \Lambda^{0,q+1}T^*_zX\To \Lambda^{0,q}T^*_zX,\ q\geqslant0$,
be the adjoint of the left exterior multiplication
$w\wedge: \Lambda^{0,q}T^*_zX\To \Lambda^{0,q+1}T^*_zX$, $u\mapsto w\wedge u$\,:
\begin{equation} \label{scat-e3}
\langle w\wedge u|v\rangle=\langle u|(w\wedge)^*v\rangle\,,
\end{equation}
for all $u\in\Lambda^{0,q}T^*_zX$, $v\in\Lambda^{0,q+1}T^*_zX$.
Notice that $(w\wedge)^*$ depends $\mathbb{C}$-anti-linearly on $w$. It is easy to see that 
\begin{equation} \label{scat-e4}
\ddbar_b=\sum^{n-1}_{j=1}e_j\wedge\ol U_j
+\sum^{n-1}_{j=1}(\ddbar_b e_j)\wedge (e_j\wedge)^*
\end{equation}
and correspondingly
\begin{equation} \label{scat-e5}
\ddbar^{\,*,k\phi}_b=\sum^{n-1}_{j=1}(e_j\wedge)^*\ol U^{\,*,k\phi}_j
+\sum^{n-1}_{j=1}e_j\wedge(\ddbar_b e_j\wedge)^*,
\end{equation}
where $\ol U^{\,*,k\phi}_j$ is the formal adjoint of $\ol U_j$ with respect to $(\ |\ )_{k\phi}$, $j=1,\ldots,n-1$. We can check that for $j=1,\ldots,n-1$,
\begin{equation} \label{scat-e6}
\ol U_j^{\,*,k\phi}=-U_j+k(U_j\phi)+s_j(z, \theta),
\end{equation}
where $s_j\in C^\infty(D)$, $s_j$ is independent of $k$, $j=1,\ldots,n-1$. 

For $r>0$, let
$D_r=\set{x=(z, \theta)\in\Real^{2n-1};\, \abs{x_j}<r,\ j=1,\ldots,2n-1}$.
Let $F_k$ be the scaling map:
$F_k(z, \theta)=(\frac{z}{\sqrt{k}}, \frac{\theta}{k})$.
From now on, we assume that $k$ is large enough so that $F_k(D_{\log k})\subset D$.
We define the scaled bundle $F^*_k\Lambda^{0,q}T^*X$ on $D_{\log k}$ to be the bundle whose fiber at $(z, \theta)\in D_{\log k}$ is
\[
F^*_k\Lambda^{0,q}T^*_{(z,\theta)}X:=\Bigr\{\textstyle\sideset{}{'}\sum_{\abs{J}=q}a_Je_J(\frac{z}{\sqrt{k}},\frac{\theta}{k});\, a_J\in\Complex,\abs{J}=q\Bigl\}\,.
\] 
We take the Hermitian metric $\langle\,\cdot\,|\,\cdot\,\rangle_{F^*_k}$ on $F^*_k\Lambda^{0,q}T^*X$ so that at each point $(z, \theta)\in D_{\log k}$\,,
\[
\Bigr\{e_J\big(\tfrac{z}{\sqrt{k}}\;,\tfrac{\theta}{k}\big)\,; \text{$\abs{J}=q$, $J$ strictly increasing}\Bigl\}\,,
\]
is an orthonormal basis for $F^*_k\Lambda^{0,q}T^*_{(z,\theta)}X$. For $r>0$, let $F^*_k\Omega^{0,q}(D_r)$
denote the space of smooth sections of $F^*_k\Lambda^{0,q}T^*X$ over $D_r$. Let $F^*_k\Omega^{0,q}_0(D_r)$ be the subspace of
$F^*_k\Omega^{0,q}(D_r)$ whose elements have compact support in $D_r$.
Given $f\in\Omega^{0,q}(F_k(D_{\log k}))$ we write
$f=\sum'_{\abs{J}=q}f_Je_J$.
We define the scaled form $F_k^*f\in F^*_k\Omega^{0,q}(D_{\log k})$ by:
\begin{equation}\label{scat-e6-1}
F_k^*f=\sideset{}{'}\sum_{\abs{J}=q}f_J\Big(\frac{z}{\sqrt{k}}\;, \frac{\theta}{k}\Big)e_J\Big(\frac{z}{\sqrt{k}}\;,\frac{\theta}{k}\Big)\,.
\end{equation}

Let $P$ be a partial differential operator of order one on $F_k(D_{\log k})$ with $C^\infty$ coefficients. We write
$P=a(z, \theta)\frac{\pr}{\pr\theta}+\sum^{2n-2}_{j=1}a_j(z, \theta)\frac{\pr}{\pr x_j}$, $a, a_j\in C^\infty(F_k(D_{\log k}))$, $j=1,\ldots,2n-2$. The partial diffferential operator
$P_{(k)}$ on $D_{\log k}$ is given by
\begin{equation} \label{scat-e7}
P_{(k)}=\sqrt{k}F^*_ka\frac{\pr}{\pr\theta}+\sum^{2n-2}_{j=1}F^*_ka_j\frac{\pr}{\pr x_j}
=\sqrt{k}a\Big(\frac{z}{\sqrt{k}}\,,\frac{\theta}{k}\Big)\frac{\pr}{\pr\theta}+\sum^{2n-2}_{j=1}a_j\Big(\frac{z}{\sqrt{k}}\,, \frac{\theta}{k}\Big)\frac{\pr}{\pr x_j}\,.
\end{equation}
Let $f\in C^\infty(F_k(D_{\log k}))$. We can check that
\begin{equation} \label{scat-e8}
P_{(k)}(F^*_kf)=\frac{1}{\sqrt{k}}F^*_k(Pf).
\end{equation} 

The scaled differential operator $\ddbar_{b,(k)}:F^*_k\Omega^{0,q}(D_{\log k})\To F^*_k\Omega^{0,q+1}(D_{\log k})$ is given by (compare to the formula \eqref{scat-e4} for $\ddbar_b$):
\begin{equation} \label{scat-e9}
\begin{split}
\ddbar_{b,(k)}=&\sum^{n-1}_{j=1}e_j\Big(\frac{z}{\sqrt{k}}\,,\frac{\theta}{k}\Big)\wedge\ol U_{j(k)}\\
&+\sum^{n-1}_{j=1}\frac{1}{\sqrt{k}}(\ddbar_b e_j)\Big(\frac{z}{\sqrt{k}}\,,\frac{\theta}{k}\Big)\wedge \Big(e_j\Big(\frac{z}{\sqrt{k}}\,,\frac{\theta}{k}\Big)\wedge\Big)^*.
\end{split}
\end{equation}
From \eqref{scat-e4} and \eqref{scat-e8}, we can check that if $f\in\Omega^{0,q}(F_k(D_{\log k}))$, then
\begin{equation} \label{scat-e10}
\ddbar_{b,(k)}F^*_kf=\frac{1}{\sqrt{k}}F^*_k(\ddbar_bf).
\end{equation} 

Let $(\ |\ )_{kF^*_k\phi}$ be the inner product on the space $F^*_k\Omega^{0,q}_0(D_{\log k})$
defined as follows:
\[
(f\ |\ g)_{kF^*_k\phi}=\int_{D_{\log k}}\!\langle f|g\rangle_{F^*_k}e^{-kF^*_k\phi}(F^*_km)(z, \theta)dv(z)d\theta\,,
\]
where $dv_X=mdv(z)d\theta$ is the volume form, $dv(z)=2^{n-1}dx_1\cdots dx_{2n-2}$. Note that $m(0,0)=1$. Let $\ddbar^{\,*,k\phi}_{b,(k)}:F^*_k\Omega^{0,q+1}(D_{\log k})\To F^*_k\Omega^{0,q}(D_{\log k})$ be the formal adjoint of $\ddbar_{b,(k)}$ with respect to $(\ |\ )_{kF^*_k\phi}$. Then, we can check that (compare the formula for $\ddbar^{\,*,k\phi}_b$, see \eqref{scat-e5} and \eqref{scat-e6})
\begin{equation} \label{scat-e11}
\begin{split}
\ddbar^{\,*,k\phi}_{b,(k)}=&\sum^{n-1}_{j=1}\Big(e_j\Big(\frac{z}{\sqrt{k}},\frac{\theta}{k}\Big)\wedge\Big)^*\Bigr(-U_{j(k)}+\sqrt{k}F^*_k(U_j\phi)+\frac{1}{\sqrt{k}}F^*_ks_j\Bigr)\\
&+\sum^{n-1}_{j=1}\frac{1}{\sqrt{k}}\,e_j\Big(\frac{z}{\sqrt{k}}\,,\frac{\theta}{k}\Big)\wedge\Big((\ddbar_b e_j)\Big(\frac{z}{\sqrt{k}}\,,\frac{\theta}{k}\Big)\wedge\Big)^*,
\end{split}
\end{equation}
where $s_j\in C^\infty(D_{\log k})$, $j=1,\ldots,n-1$, are independent of $k$.
We also have
\begin{equation} \label{scat-e12}
\ddbar^{\,*,k\phi}_{b,(k)}F^*_kf=\frac{1}{\sqrt{k}}F^*_k(\ddbar^{\,*,k\phi}_bf),\,\quad f\in\Omega^{0,q+1}(F_k(D_{\log k}))\,.
\end{equation}
We define now the \emph{scaled Kohn-Laplacian}:
\begin{equation} \label{scat-e13}
\Box^{(q)}_{b,k\phi,(k)}:=\ddbar^{\,*,k\phi}_{b,(k)}\ddbar_{b,(k)}+\ddbar_{b,(k)}\ddbar^{\,*,k\phi}_{b,(k)}:F^*_k\Omega^{0,q}(D_{\log k})\To F^*_k\Omega^{0,q}(D_{\log k}).
\end{equation}
From \eqref{scat-e10} and \eqref{scat-e12}, we see that if $f\in\Omega^{0,q}(F_k(D_{\log k}))$, then
\begin{equation} \label{scat-e14}
(\Box^{(q)}_{b,k\phi,(k)})F^*_kf=\frac{1}{k}F^*_k(\Box^{(q)}_{b,k\phi}f).
\end{equation} 

From \eqref{epch}, \eqref{epchI} and \eqref{scat-e1}, we can check that
\begin{equation} \label{scat-e15}
\ol U_{j(k)}=\frac{\pr}{\pr\ol z_j}-i\lambda_jz_j\frac{\pr}{\pr \theta}+\epsilon_kZ_{j,k}\,,\quad j=1,\ldots,n-1,
\end{equation}
on $D_{\log k}$, where $\epsilon_k$ is a sequence tending to zero with $k\To\infty$ and $Z_{j,k}$ is a first order differential operator and all the derivatives of the coefficients of $Z_{j,k}$ are uniformly bounded in $k$ on $D_{\log k}$, $j=1,\ldots,n-1$. Similarly, from \eqref{epch}, \eqref{epchI} and \eqref{scat-e1}, it is straightforward to see that
\begin{equation} \label{scat-e16}
\begin{split}
&-U_{t(k)}+\sqrt{k}F^*_k(U_t\phi)+\frac{1}{\sqrt{k}}F^*_ks_t\\
&=-\frac{\pr}{\pr z_t}-i\lambda_t\ol z_t\frac{\pr}{\pr\theta}+i\lambda_t\ol z_t\beta+\sum^{n-1}_{j=1}\mu_{j,\,t}\,\ol{z}_j+\delta_kV_{t,\,k}\,,\quad  t=1,\ldots,n-1,
\end{split}
\end{equation}
on $D_{\log k}$, where $\delta_k$ is a sequence tending to zero with $k\To\infty$ and $V_{t,k}$ is a first order differential operator and all the derivatives of the coefficients of $V_{t,k}$ are uniformly bounded in $k$ on $D_{\log k}$, $t=1,\ldots,n-1$.
From \eqref{scat-e16}, \eqref{scat-e15} and \eqref{scat-e13}, \eqref{scat-e11}, \eqref{scat-e9}, it is straightforward to obtain the following 

\begin{prop} \label{scat-p1}
We have that
\begin{equation*} 
\begin{split}
\Box^{(q)}_{b,k\phi,(k)}&=\sum^{n-1}_{j=1}\Bigr[\Bigr(-\frac{\pr}{\pr z_j}-i\lambda_j\ol z_j\frac{\pr}{\pr \theta}+i\lambda_j\ol z_j\beta+\sum^{n-1}_{t=1}\mu_{t,\,j}\,\ol z_t\Bigr)\Bigr(\frac{\pr}{\pr\ol z_j}-i\lambda_jz_j\frac{\pr}{\pr\theta}\Bigr)\Bigr]\\
&
+\sum^{n-1}_{j,\,t=1}e_j\Big(\frac{z}{\sqrt{k}}\,, \frac{\theta}{k}\Big)\wedge \Big(e_t\Big(\frac{z}{\sqrt{k}}\,,\frac{\theta}{k}\Big)\wedge\Big)^*\Bigr(\Bigr(\mu_{j,\,t}+i\lambda_j\delta_{j,\,t}\beta\Bigr)-2i\lambda_j\delta_{j,\,t}\,\frac{\pr}{\pr \theta}\Bigr)\\
&\quad+\varepsilon_kP_k,
\end{split}
\end{equation*}
on $D_{\log k}$, where $\varepsilon_k$ is a sequence tending to zero with $k\To\infty$, $P_k$ is a second order differential operator and all the derivatives of the coefficients of $P_k$ are uniformly bounded in $k$ on $D_{\log k}$.
\end{prop} 

Let $D\subset D_{\log k}$ be an open set and let $W^s_{kF^*_k\phi}(D,F^*_k\Lambda^{0, q}T^*X)$,
$s\in\mathbb N_0$, denote the Sobolev space of order $s$ of sections of $F^*_k\Lambda^{0,q}T^*X$
over $D$ with respect to the weight $e^{-kF^*_k\phi}$. The Sobolev norm on this space is given by
\begin{equation} \label{scat-e17}
\norm{u}^2_{kF^*_k\phi,s,D}
=\sideset{}{'}\sum_{\substack{\alpha\in\mathbb{N}^{2n-1}_0,\;\abs{\alpha}\leqslant  s\;\\{\abs{J}=q}}}
\int_{D}\!\abs{\pr^\alpha_{x}u_J}^2e^{-kF^*_k\phi}(F^*_km)(z, \theta)dv(z)d\theta,
\end{equation}
where
$u=\sum'_{\abs{J}=q}u_Je_J\big(\frac{z}{\sqrt{k}},\frac{\theta}{k}\big)\in W^s_{kF^*_k\phi}
(D,F^*_k\Lambda^{0,q}T^*X)$ and $m$ is the volume form.
If $s=0$, we write $\norm{\cdot}_{kF^*_k\phi,D}$ to denote $\norm{\cdot}_{kF^*_k\phi,0,D}$.
The following is well-known (see Proposition 2.4 and Lemma 2.6 in~\cite{HM09})

\begin{prop} \label{scat-p2}
Assume that $Y(q)$ holds at each point of $X$. 
For every $r>0$ with $D_{2r}\subset D_{\log k}$ and $s\in\mathbb N_0$, there are constants $C_{r,s}>0$, 
$C_r>0$, $C_{r,s}$ and $C_r$ are
independent of $k$, such that 
\begin{equation} \label{scat-e18}
\norm{u}^2_{kF^*_k\phi,s+1,D_{r}}\leqslant  C_{r,s}\Bigr(\norm{u}^2_{kF^*_k\phi,D_{2r}}+\big\|\Box^{(q)}_{b,k\phi,(k)}u\big\|^2_{kF^*_k\phi,s,D_{2r}}\Bigl),\ 
u\in F^*_k\Omega^{0,q}(D_{\log k})
\end{equation} 
and
\begin{equation} \label{scat-e19}
\begin{split}
&\sup_{x\in D_r}\abs{u(x)}^2\\
&\quad\leqslant  C_r\Bigr(\norm{u}^2_{kF^*_k\phi,D_{2r}}+\sum^n_{m=1}\big\|(\Box^{(q)}_{b,k\phi,(k)})^mu\big\|^2_{kF^*_k\phi,D_{2r}}\Bigl),\ 
u\in F^*_k\Omega^{0,q}(D_{\log k})\,.
\end{split}
\end{equation} 
\end{prop} 

We pause and introduce some notations. We identify $\Real^{2n-1}$ with the Heisenberg group $H_n:=\Complex^{n-1}\times\Real$. We also write $(z, \theta)$ to denote the coordinates of $H_n$, $z=(z_1,\ldots,z_{n-1})\in\Complex^{n-1}$, $z_j=x_{2j-1}+ix_{2j}$, $j=1,\ldots,n-1$, and $\theta\in\Real$. Then
\[
\begin{split}
&\Big\{U_{j,H_n}=\frac{\pr}{\pr z_j}+i\lambda_j\ol z_j\frac{\pr}{\pr\theta}\,;\, j=1,\ldots,n-1\Big\},\,\\
&\Big\{U_{j,H_n}\,,\ \ol{U}_{j,H_n}\,,\ T=\frac{\pr}{\pr\theta}\, ;\,j=1,\ldots,n-1\Big\}
\end{split}
\]
are orthonormal bases for the bundles $T^{1, 0}H_n$ and $\Complex TH_n$ respectively. Then
\[
\Big\{dz_j\, ,\ d\ol z_j\, ,\ \omega_0=-d\theta+\sum^{n-1}_{j=1}(i\lambda_j\ol z_jdz_j-i\lambda_jz_jd\ol z_j)\,;j=1,\ldots,n-1\Big\}
\]
is the basis of $\Complex T^*H_n$ which is dual to $\{U_{j,H_n},\ol U_{j,H_n}, -T;j=1,\ldots,n-1\}$.
We take the Hermitian metric $\langle\,\cdot\,|\,\cdot\,\rangle$ on $\Lambda^{0,q}T^*H_n$ such that 
\[\{d\ol z_J;\,  \text{$\abs{J}=q$, $J$ strictly increasing}\}\] 
is an orthonormal basis of $\Lambda^{0,q}T^*H_n$. The Cauchy-Riemann operator $\ddbar_{b,H_n}$ on $H_n$ is given by
\begin{equation} \label{scat-e20}
\ddbar_{b,H_n}=\sum^{n-1}_{j=1}d\ol z_j\wedge\ol U_{j,H_n}:\Omega^{0,q}(H_n)\To\Omega^{0,q+1}(H_n).
\end{equation}
Put
$\phi_0(z, \theta)=\beta\theta+\sum^{n-1}_{j,t=1}\mu_{j,\,t}\ol z_jz_t\in C^\infty(H_n,\Real)$,
where $\beta$ and $\mu_{j,\,t}$, $j,t=1,\ldots,n-1$, are as in \eqref{scat-e1}. Note that
\begin{equation} \label{scat-e21}
\sup_{(z, \theta)\in D_{\log k}}\abs{kF^*_k\phi-\phi_0}\To0,\ \ \mbox{as $k\To\infty$}.
\end{equation}

Let $(\ |\ )_{\phi_0}$
be the inner product on $\Omega^{0,q}_0(H_n)$ defined as follows:
\[
(f\ |\ g)_{\phi_0}=\int_{H_n}\!\langle f|g\rangle e^{-\phi_0}dv(z)d\theta\,, \quad f, g\in\Omega^{0,q}_0(H_n)\,,
\]
where $dv(z)=2^{n-1}dx_1dx_2\cdots dx_{2n-2}$. Let $\ddbar^{\,*,\phi_0}_{b,H_n}:\Omega^{0,q+1}(H_n)\To\Omega^{0,q}(H_n)$
be the formal adjoint of $\ddbar_{b,H_n}$ with respect to $(\ |\ )_{\phi_0}$. We have
\begin{equation} \label{scat-e22}
\ddbar^{\,*,\phi_0}_{b,H_n}=\sum^{n-1}_{t=1}(d\ol z_t\wedge)^*\,\ol U_{t,H_n}^{\,*,\phi_0}:\Omega^{0,q+1}(H_n)\To \Omega^{0,q}(H_n),
\end{equation}
where
\begin{equation} \label{scat-e23}
\ol U_{t,H_n}^{\,*,\phi_0}=-U_{t,H_n}+U_{t,H_n}\phi_0=-U_{t,H_n}+\sum^{n-1}_{j=1}\mu_{j,\,t}\ol z_j+i\lambda_t\ol z_t\beta\,.
\end{equation}
The Kohn-Laplacian on $H_n$ is given by
\begin{equation}\label{scat-e24}
\Box^{(q)}_{b,H_n}=\ddbar_{b,H_n}\ddbar^{\,*,\phi_0}_{b,H_n}+\ddbar^{\,*,\phi_0}_{b,H_n}\ddbar_{b,H_n}:\Omega^{0,q}(H_n)\To\Omega^{0,q}(H_n)\,.
\end{equation}
From \eqref{scat-e20}, \eqref{scat-e22} and \eqref{scat-e23}, we can check that
\begin{equation} \label{scat-e25}
\begin{split}
&\Box^{(q)}_{b,H_n}\\
&=\sum^{n-1}_{j=1}\ol U^{\,*,\phi_0}_{j,H_n}\ol U_{j,H_n}
        +\sum^{n-1}_{j,\,t=1}d\ol z_j\wedge (d\ol z_t\wedge)^*\Bigr[\Bigr(\mu_{j,\,t}+i\lambda_j\delta_{j,\,t}\beta\Bigr)-2i\lambda_j\delta_{j,\,t}\,\frac{\pr}{\pr\theta}\Bigr]\\
        &=\sum^{n-1}_{j=1}\Bigr[\Bigr(-\frac{\pr}{\pr z_j}-i\lambda_j\ol z_j\frac{\pr}{\pr\theta}+i\lambda_j\ol z_j\beta+\sum^{n-1}_{t=1}\mu_{t,j}\ol z_t\Bigr)\Bigr(\frac{\pr}{\pr\ol z_j}-i\lambda_jz_j\frac{\pr}{\pr\theta}\Bigr)\Bigr]\\
        &\quad+\sum^{n-1}_{j,t=1}d\ol z_j\wedge (d\ol z_t\wedge)^*\Bigr[\Bigr(\mu_{j,\,t}+i\lambda_j\delta_{j,\,t}\beta\Bigr)-2i\lambda_j\delta_{j,\,t}\,\frac{\pr}{\pr\theta}\Bigr].
\end{split}
\end{equation} 

Now, we can prove 

\begin{prop} \label{scat-p3} 
Assume that $Y(q)$ holds at each point of $X$. 
For each $k$, let $\alpha_k\in F^*_k\Omega^{0,q}(D_{\log k})$. We assume that $\norm{\alpha_k}_{kF^*_k\phi,D_{\log k}}\leqslant  1$ for each $k$, and there is a sequence $\nu_k>0$, $\nu_k\To0$ as $k\To\infty$, such that for each $k$,
\[\norm{(\Box^{(q)}_{b,k\phi,(k)})^m\alpha_k}_{kF^*_k\phi,D_{\log k}}\leq\nu_k^m,\ \ \forall m\in\mathbb N.\]
Identify $\alpha_k$ with a form on $H_n$ by extending it with zero and write \[\alpha_k=\sideset{}{'}\sum_{\abs{J}=q}\alpha_{k,J}e_J(\frac{z}{\sqrt{k}},\frac{\theta}{k}).\] 
Then there is a subsequence $\set{\alpha_{k_j}}$ of $\set{\alpha_k}$ such that for each strictly increasing multiindex $J$, $\abs{J}=q$, $\alpha_{k_j,\,J}$ converges uniformly with all its derivatives on any compact subset of $H_n$ to a smooth function $\alpha_J$. Furthermore, if we put $\alpha=\sum'_{\abs{J}=q}\alpha_Jd\ol z_J$, then
$\Box^{(q)}_{b,H_n}\alpha=0$.
\end{prop}

\begin{proof} 
From \eqref{scat-e18} and using induction, we get for any $r>0$ and for every $s\in\mathbb N_0$, there is a constant $C_{r,s}>0$ independent of $k$, such that
\begin{equation} \label{scat-e26}
\begin{split}
\norm{\alpha_k}^2_{kF^*_k\phi,s+1,D_r}
&\leqslant  C_{r,s}\Bigr(\norm{\alpha_k}^2_{kF^*_k\phi,D_{2r}}+\sum^{s+1}_{m=1}
\norm{(\Box^{(q)}_{b,k\phi,(k)})^m\alpha_k}^2_{kF^*_k\phi,D_{2r}}\,\Bigr)\\
&\leq C_{r,s}(1+\sum^{\infty}_{m=1}\nu_k^m)\leq \Td C_{r,s}
\end{split}
\end{equation} 
for $k$ large, where $\Td C_{r,s}>0$ is independent of $k$.
Fix a strictly increasing multiindex $J$, $\abs{J}=q$,
and $r>0$. Combining \eqref{scat-e26} with Rellich 's compactness theorem \cite[p.\,281]{Yo80}, we conclude that  there is a subsequence of $\set{\alpha_{k,J}}$, which converges in all Sobolev spaces $W^s(D_r)$ for $s>0$. From the Sobolev embedding theorem \cite[p.\,170]{Yo80}, we see that the sequence converges in all $C^l(D_r)$, $l\geqslant0$, $l\in\mathbb Z$, locally unformly. Choosing a diagonal sequence, with respect to a sequence of $D_r$ exhausting $H_n$, we get a subsequence $\set{\alpha_{k_j,J}}$ of $\set{\alpha_{k,J}}$ such that $\alpha_{k_j,J}$ converges uniformly with all derivatives on any compact subset of $H_n$ to a smooth function $\alpha_J$.

Let $J'$ be another strictly increasing multiindex, $\abs{J'}=q$. We can repeat the procedure above and get a subsequce $\set{\alpha_{k_{j_s},J'}}$ of $\set{\alpha_{k_j,J'}}$ such that $\alpha_{k_{j_s},J'}$ converges uniformly with all derivatives on any compact subset of $H_n$ to a smooth function $\alpha_{J'}$. Continuing in this way, we get the first statement of the proposition.

Now, we prove the second statement of the proposition. Let $P=(p_1,\ldots,p_q)$, $R=(r_1,\ldots,r_q)$ be multiindices, $\abs{P}=\abs{R}=q$. Define
\[\varepsilon^P_R=\left\{ \begin{array}{ll}
&0,\ \ \mbox{if $\set{p_1,\ldots,p_q}\neq\set{r_1,\ldots,r_q}$}, \\
&\mbox{the sign of permutation $\left(
\begin{array}[c]{c}
 P  \\
R
\end{array}\right)$
},\ \ \mbox{if $\set{p_1,\ldots,p_q}=\set{r_1,\ldots,r_q}$}.
\end{array}\right.\]
For $j, t=1,\ldots,n-1$, define
\[\sigma^{jtP}_{R}=\left\{ \begin{array}{ll}
&0,\ \ \mbox{if $d\ol z_j\wedge (d\ol z_t\wedge)^*(d\ol z^P)=0$}, \\
&\varepsilon^Q_R,\ \ \mbox{if $d\ol z_j\wedge (d\ol z_t\wedge)^*(d\ol z^P)=d\ol z^Q$, $\abs{Q}=q$}.
\end{array}\right.\]

We may assume that $\alpha_{k,J}$ converges uniformly with all derivatives on any compact subset of $H_n$ to a smooth function $\alpha_J$, for all strictly increasing $J$, $\abs{J}=q$. As \eqref{scat-e26}, we have for any 
$r>0$ and for every $s\in\mathbb N_0$, there is a constant $C_{r,s}>0$ independent of $k$, such that
\begin{equation} \label{scat-e27}
\begin{split}
&\norm{\Box^{(q)}_{k,k\phi,(k)}\alpha_k}^2_{kF^*_k\phi,s+1,D_r}\\
&\quad\leqslant  C_{r,s}\Bigr(\norm{\Box^{(q)}_{k,k\phi,(k)}\alpha_k}^2_{kF^*_k\phi,D_{2r}}+\sum^{s+1}_{m=1}
\norm{(\Box^{(q)}_{b,k\phi,(k)})^{m+1}\alpha_k}^2_{kF^*_k\phi,D_{2r}}\,\Bigr)\\
&\quad\leq C_{r,s}\sum^{\infty}_{m=1}\nu_k^m\To0\ \ \mbox{as $k\To\infty$}.
\end{split}
\end{equation} 
Put 
\[\beta_k:=\Box^{(q)}_{k,k\phi,(k)}\alpha_k=\sideset{}{'}\sum_{\abs{J}=q}\beta_{k,J}e_J(\frac{z}{\sqrt{k}},\frac{\theta}{k})\in F^*_k\Omega^{0,q}(D_{\log k}).\]
Combining \eqref{scat-e27} with Sobolev embedding theorem \cite[p.\,170]{Yo80}, we conclude that 
\begin{equation}\label{scat-e28}
\begin{split}
&\mbox{$\beta_{k,J}$ converges uniformly with all derivatives on any compact subset of $H_n$} \\
&\mbox{to zero, for all strictly increasing $J$, $\abs{J}=q$.}
\end{split}
\end{equation}

From the explicit formula of $\Box^{(q)}_{b,k\phi,(k)}$
(see Proposition~\ref{scat-p1}), it is not difficult to see that for all strictly increasing $J$, $\abs{J}=q$, we have
\begin{equation} \label{scat-e29}
\begin{split}
\sum^{n-1}_{j=1}\ol U^{\,*,\phi_0}_{j,H_n}\ol U_{j,H_n}\alpha_{k,J}=-\sideset{}{'}\sum_{\substack{\abs{P}=q,\\1\,\leqslant\, \, j\,,\,t\,\leqslant \, n-1}}
\sigma^{jtP}_{J}
\Bigr[\Bigr(\mu_{j,\,t}+i\lambda_j\delta_{j,\,t}\beta\Bigr)&-2i\lambda_j\delta_{j,\,t}\frac{\pr}{\pr\theta}\Bigr]\alpha_{k,P}\\
&+\epsilon_kP_{k,J}\alpha_k+\beta_{k,J}
\end{split}
\end{equation}
on $D_{\log k}$, where $\epsilon_k$ is a sequence tending to zero with $k\To\infty$ , $P_{k,J}$ is a second order differential operator and all the derivatives of the coefficients of $P_{k,J}$ are uniformly bounded in $k$ on $D_{\log k}$ and $\beta_{k,J}$ is as in \eqref{scat-e28}. By letting $k\To\infty$ in \eqref{scat-e29} we get
\begin{equation} \label{scat-e30}
\sum^{n-1}_{j=1}\ol U^{\,*,\phi_0}_{j,H_n}\ol U_{j,H_n}\alpha_{J}=-\sideset{}{'}\sum_{\substack{\abs{P}=q,\\1\,\leqslant\, \, j\,,\,t\,\leqslant \, n-1}}
\sigma^{jtP}_{J}
\Bigr[\Bigr(\mu_{j,\,t}+i\lambda_j\delta_{j,\,t}\beta\Bigr)-2i\lambda_j\delta_{j,\,t}\frac{\pr}{\pr\theta}\Bigr]\alpha_{P}
\end{equation}
on $H_n$, for all strictly increasing $J$, $\abs{J}=q$. From this and the explicit formula of $\Box^{(q)}_{b,H_n}$ (see \eqref{scat-e25}),
we conclude that $\Box^{(q)}_{b,H_n}\alpha=0$. The proposition follows.
\end{proof}

\section{The operators $Q^{(0)}_{M,k}$ and $Q^{(1)}_{M,k}$} 

From now on, we assume that $h^L$ is semi-positive on $X$ and positive at some point of $X$ and conditions $Y(0)$ and $Y(1)$ hold at each point of $X$.

Take $\delta_0>0$ be a small constant so that 
\begin{equation}\label{sp3-eI}
M^\phi_x+2s\mathcal{L}_x\geq0,\ \ \forall \abs{s}\leq\delta_0,\ \ \forall x\in X.
\end{equation} 
Take $\psi(\eta)\in C^\infty_0(]-\delta_0,\delta_0[,\ol\Real_+)$ so that $\psi(\eta)=1$ if $-\frac{\delta_0}{2}\leq \eta\leq\frac{\delta_0}{2}$. Let $\hat\psi(t)=\int e^{-it\eta}\psi(\eta)d\eta$ be the Fourier transform of $\psi$. 
Put 
\begin{equation}\label{sp3-eI0}
C_0:=\sup_{t\in\Real}t^2\abs{\hat\psi(t)}.
\end{equation}
Let $E>0$ be a small constant so that
\begin{equation}\label{change}
\begin{split}
&\sqrt{\int_XE^2dv_X(x)}\sqrt{2(2\pi)^{-n}\int_X\Bigr(\int\det(M^\phi_x+2\xi\mathcal{L}_x)\mathds{1}_{\Real_{x,0}}(\xi)d\xi\Bigr)dv_X(x)}\\
&\leq\frac{(2\pi)^{1-n}}{4}
\int_X\Bigr(\int\psi(\xi)\det(M^\phi_x+2\xi\mathcal{L}_x)\mathds{1}_{\Real_{x,0}}(\xi)d\xi\Bigr)dv_X(x).
\end{split}
\end{equation}
Fix $M>0$ be a large constant so that 
\begin{equation}\label{sp3-eI1}
\begin{split}
&\frac{2C_0}{M}(2\pi)^{-n}\int_X\Bigr(\int\det(M^\phi_x+2\xi\mathcal{L}_x)\mathds{1}_{\Real_{x,0}}(\xi)d\xi\Bigr)dv_X(x)\\
&<\frac{(2\pi)^{1-n}}{2}
\int_X\Bigr(\int\psi(\xi)\det(M^\phi_x+2\xi\mathcal{L}_x)\mathds{1}_{\Real_{x,0}}(\xi)d\xi\Bigr)dv_X(x)
\end{split}
\end{equation}
and
\begin{equation}\label{sp3-eI2}
\frac{2C_0}{M}\Bigr((2\pi)^{-n}\int\abs{\det(M^\phi_x+2\xi\mathcal{L}_x)}\mathds{1}_{\Real_{x,1}}(\xi)d\xi\Bigr)^{\frac{1}{2}}<\frac{E}{\sqrt{n-1}},\ \ \forall x\in X,
\end{equation}
where $\mathds{1}_{\Real_{x,1}}(\xi)=1$ if $\xi\in\Real_{x,1}$, $\mathds{1}_{\Real_{x,1}}(\xi)=0$ if $\xi\notin\Real_{x,1}$.
Take $\chi(t)\in C^\infty_0(]-2,2[,\ol\Real_+)$ so that $0\leq\chi(t)\leq1$ and $\chi(t)=1$ if $-1\leq t\leq 1$ and 
$\chi(-t)=\chi(t)$, for all $t\in\Real$. Put 
\begin{equation}\label{sp3-eM0}
\chi_M(t):=\chi(\frac{t}{M}).
\end{equation} 
As before, let $\Phi^t(x)$ be the $T$-flow. 
The operator $Q^{(0)}_{M,k}$ is a continuous operator $C^\infty(X,L^k)\To C^\infty(X,L^k)$ defined as follows. Let $u\in C^\infty(X,L^k)$. Let $D\Subset D'\Subset X$ be open sets of $X$ and let $s$ be a local section of $L$ on $D'$, $\abs{s}^2_{h^L}=e^{-\phi}$. On $D'$, we write 
$u=s^k\Td u$, $\Td u\in C^\infty(D')$. Then, 
\begin{equation}\label{sp3-eII}
(Q^{(0)}_{M,k}u)(x):=s^ke^{\frac{k}{2}\phi(x)}\int e^{-it\eta}\psi(\eta)\chi_M(t)e^{-\frac{k}{2}\phi(\Phi^{\frac{t}{k}}(x))}\Td u(\Phi^{\frac{t}{k}}(x))dtd\eta\ \ \mbox{on $D$}.
\end{equation}

We first notice that for $k$ large, $\Phi^{\frac{t}{k}}(x)$ is well-defined for all $t\in{\rm Supp\,}\chi_M$, every $x\in X$ and $\Phi^{\frac{t}{k}}(x)\in D'$ for all $t\in{\rm Supp\,}\chi_M$, every $x\in D$. We may assume that $\Phi^{\frac{t}{k}}(x)$ is well-defined for all $t\in{\rm Supp\,}\chi_M$, every $x\in X$ and 
$\Phi^{\frac{t}{k}}(x)\in D'$ for all $t\in{\rm Supp\,}\chi_M$, every $x\in D$. 
Now we check that the definition above is independent of the choice of local sections. Let $\hat s$ be another local section 
of $L$ on $D'$, $\abs{\hat s}^2_{h^L}=e^{-\hat\phi}$. Then, we have $\hat s=gs$ for some non-zero rigid CR function $g$.
We can check that 
\begin{equation}\label{spe-eM1}
\begin{split}
&\hat\phi=\phi-2\log\abs{g},\\
&e^{-\frac{k}{2}\hat\phi}=e^{-\frac{k}{2}\phi}\abs{g}^k.
\end{split}
\end{equation}
Let $u\in C^\infty(X,L^k)$. On $D$, we write $u=s^k\Td u=\hat s^k\hat u$. We have 
\begin{equation}\label{spe-eM-1}
\hat u=g^{-k}\Td u.
\end{equation}
From \eqref{spe-eM1} and \eqref{spe-eM-1}, we can check that
\begin{equation}\label{spe-eM2}
e^{-\frac{k}{2}\hat\phi}\hat u=e^{-\frac{k}{2}\phi}\abs{g}^kg^{-k}\Td u.
\end{equation}
Since $Tg=0$, we have $(\abs{g}^kg^{-k})(\Phi^{\frac{t}{k}}x)=(\abs{g}^kg^{-k})(x)$
for all $t\in{\rm Supp\,}\chi_M$, $x\in D$. From this observation and \eqref{spe-eM2}, it is easy to see that 
\begin{equation}\label{spe-eM3}\begin{split}
&\int e^{-it\eta}\psi(\eta)\chi_M(t)e^{-\frac{k}{2}\hat\phi(\Phi^{\frac{t}{k}}(x))}\hat u(\Phi^{\frac{t}{k}}(x))dtd\eta\\
&\quad=(\abs{g}^kg^{-k})(x)\int e^{-it\eta}\psi(\eta)\chi_M(t)e^{-\frac{k}{2}\phi(\Phi^{\frac{t}{k}}(x))}\Td u(\Phi^{\frac{t}{k}}(x))dtd\eta\ \ \mbox{on $D$}.\end{split}\end{equation} 
Furthermore, we can check that
\[\hat s^ke^{\frac{k}{2}\hat\phi}=\abs{g}^{-k}g^k s^ke^{\frac{k}{2}\phi}.\] 
Combining this with \eqref{spe-eM3}, we obtain 
\[\begin{split}
&\hat s^ke^{\frac{k}{2}\hat\phi(x)}\int e^{-it\eta}\psi(\eta)\chi_M(t)e^{-\frac{k}{2}\hat\phi(\Phi^{\frac{t}{k}}(x))}\hat u(\Phi^{\frac{t}{k}}(x))dtd\eta\\
&\quad=s^ke^{\frac{k}{2}\phi(x)}\int e^{-it\eta}\psi(\eta)\chi_M(t)e^{-\frac{k}{2}\phi(\Phi^{\frac{t}{k}}(x))}\Td u(\Phi^{\frac{t}{k}}(x))dtd\eta\ \ \mbox{on $D$}.\end{split}\]
Thus, the definition of $Q^{(0)}_{M,k}$ is well-defined. 

We consider $(0,1)$ forms. The operator $Q^{(1)}_{M,k}$ is a continuous operator 
\[Q^{(1)}_{M,k}:\Omega^{0,1}(X,L^k)\To\Omega^{0,1}(X,L^k)\]
defined as follows. Let $D$ be an open set of $X$. We assume 
that there exist canonical coordinates $x$ defined in some neighborhood $W$ of $\ol D$ and $L$ is trivial on $W$. Let $\psi(\eta)$ and $\chi_M$ be as in \eqref{sp3-eII}. For $k$ large, we have 
\[\set{\Phi^{\frac{t}{k}}(x)\in W;\, \forall x\in D,t\in{\rm Supp\,}\chi_M}.\]
Let $s$ be a local section of $L$ on $W$, $\abs{s}^2_{h^L}=e^{-\phi}$.
Let $x=(x_1,\ldots,x_{2n-1})=(z,\theta)$ be canonical coordinates on $W$. Then, 
\begin{equation}\label{sp3-eIV}\begin{split}
&T=\frac{\pr}{\pr\theta},\\
&Z_j=\frac{\pr}{\pr z_j}+i\frac{\pr\varphi}{\pr z_j}(z)\frac{\pr}{\pr\theta},\ \ j=1,\ldots,n-1,
\end{split}
\end{equation}
where $Z_j(x)$, $j=1,\ldots,n-1$, form a basis of $T^{1,0}_xX$, for each $x\in D$, and $\varphi(z)\in C^\infty(D,\Real)$
independent of $\theta$. We can check that $d\ol z_j$, $j=1,\ldots,n-1$, is 
the basis of $T^{*0,1}X$, dual to $\ol Z_j$, $j=1,\ldots,n-1$. Let $u\in\Omega^{0,1}(X,L^k)$. On $W$, we write 
\[u=s^k\sum^{n-1}_{j=1}\Td u_j(x)d\ol z_j,\ \ \Td u_j\in C^\infty(D),\ j=1,\ldots,n-1.\]
 Then, 
\begin{equation}\label{sp3-eV}
\begin{split}
&(Q^{(1)}_{M,k}u)(x)\\
&:=s^ke^{\frac{k}{2}\phi(x)}\sum^{n-1}_{j=1}(\int e^{-it\eta}\psi(\eta)\chi_M(t)e^{-\frac{k}{2}\phi(\Phi^{\frac{t}{k}}(x))}\Td u_j(\Phi^{\frac{t}{k}}(x))dtd\eta)d\ol z_j\ \ \mbox{on $D$}.
\end{split}
\end{equation}
As before, we can show that the definition \eqref{sp3-eV} is independent of the choices of local sections. 
Now we check that the definition \eqref{sp3-eV} is independent of the choice of canonical coordinates. Let 
$y=(y_1,\ldots,y_{2n-1})=(w,\gamma)$, $w_j=y_{2j-1}+iy_{2j}$, $j=1,\ldots,n-1$, $\gamma=y_{2n-1}$, be another canonicl coordinates on $W$. Then, 
\begin{equation}\label{sp3-eVI}\begin{split}
&T=\frac{\pr}{\pr\gamma},\\
&\Td Z_j=\frac{\pr}{\pr w_j}+i\frac{\pr\Td\varphi}{\pr w_j}(w)\frac{\pr}{\pr\gamma},\ \ j=1,\ldots,n-1,
\end{split}
\end{equation}
where $\Td Z_j(y)$, $j=1,\ldots,n-1$, form a basis of $T^{1,0}_yX$, for each $y\in D$, and $\Td\varphi(w)\in C^\infty(D,\Real)$ independent of $\gamma$. From \eqref{sp3-eVI} and \eqref{sp3-eIV}, it is not difficult to see that on $W$, we have
\begin{equation}\label{sp3-eVII}
\begin{split}
&w=(w_1,\ldots,w_{n-1})=(H_1(z),\ldots,H_{n-1}(z))=H(z),\ \ H_j(z)\in C^\infty,\ \ \forall j,\\
&\gamma=\theta+G(z),\ \ G(z)\in C^\infty,
\end{split}
\end{equation}
where for each $j=1,\ldots,n-1$, $H_j(z)$ is holomorphic. From \eqref{sp3-eVII}, we can check that 
\begin{equation}\label{spca}
d\ol w_j=\sum^{n-1}_{l=1}\ol{(\frac{\pr H_j}{\pr z_l})}d\ol z_l,\ \ j=1,\ldots,n-1.
\end{equation}
From this observation, we have for $u\in\Omega^{0,1}(X,L^k)$,
\begin{equation}\label{sp3-eVIII}
\begin{split}
&u=s^k\sum^{n-1}_{j=1}\Td u_j(x)d\ol z_j=s^k\sum^{n-1}_{j=1}\hat u_j(y)d\ol w_j\ \ \mbox{on $W$},\\
&\Td u_l(x)=\sum^{n-1}_{j=1}\hat u_j(H(z),\theta+G(z))\ol{\frac{\pr H_j}{\pr z_l}(z)},\ \ l=1,\ldots,n-1.
\end{split}
\end{equation}
On $D$, we have $\Phi^{\frac{t}{k}}(x)=(z,\frac{t}{k}+\theta)$, $\Phi^{\frac{t}{k}}(y)=(w,\frac{t}{k}+\gamma)$ and $\frac{\pr H_j}{\pr z_l}(\Phi^{\frac{t}{k}}(z))=\frac{\pr H_j}{\pr z_l}(z)$, $j,l=1,\ldots,n-1$, $t\in{\rm Supp\,}\chi_M$.
From this observation and \eqref{sp3-eVIII}, \eqref{spca}, it is straightforward to see that 
\[\begin{split}
&s^ke^{\frac{k}{2}\phi(x)}\sum^{n-1}_{l=1}(\int e^{-it\eta}\psi(\eta)\chi_M(t)e^{-\frac{k}{2}\phi(\Phi^{\frac{t}{k}}(x))}\Td u_l(\Phi^{\frac{t}{k}}(x))dtd\eta)d\ol z_l\\
&=s^ke^{\frac{k}{2}\phi(y)}\sum^{n-1}_{j=1}(\int e^{-it\eta}\psi(\eta)\chi_M(t)e^{-\frac{k}{2}\phi(\Phi^{\frac{t}{k}}(y))}\hat u_j(\Phi^{\frac{t}{k}}(y))dtd\eta)d\ol w_j.\end{split}\]
Thus, the definition \eqref{sp3-eV} is independent of the choice of canonical coordinates. The operator $Q^{(1)}_{M,k}$ is well-defined. 

Now, we claim that 
\begin{equation}\label{sp3-eX} 
Q^{(1)}_{M,k}\ddbar_{b,k}u=\ddbar_{b,k}Q^{(0)}_{M,k}u,\ \ \forall u\in C^\infty(X,L^k).
\end{equation}
We work with canonical coordinates $x=(z,\theta)$ as \eqref{sp3-eIV}.
For $u\in C^\infty(X,L^k)$, we can check that 
\begin{equation} \label{sp3-eIX}
\ddbar_{b,k}u=s^k\sum^n_{j=1}(\ol Z_j\Td u)d\ol z_j=s^k\sum^{n-1}_{j=1}(\frac{\pr\Td u}{\pr\ol z_j}-i\frac{\pr\varphi}{\pr \ol z_j}(z)\frac{\pr\Td u}{\pr\theta})d\ol z_j
\end{equation}
on $W$, where $u=s^k\Td u$ on $W$. Combining \eqref{sp3-eIX} with \eqref{sp3-eV}, \eqref{sp3-eII} and notice that $\frac{\pr\Td u}{\pr\theta}(\Phi^t(x))=\frac{\pr}{\pr\theta}\bigr(\Td u(\Phi^t(x))\bigr)$, it is easy to see that 
\begin{equation}\label{sp3-eXI}
\begin{split}
&Q^{(1)}_{M,k}\ddbar_{b,k}u-\ddbar_{b,k}Q^{(0)}_{M,k}u\\
&=-\frac{k}{2}s^ke^{\frac{k}{2}\phi(x)}\times\\
&\quad\sum^{n-1}_{j=1}(\int e^{-it\eta}\psi(\eta)\chi_M(t)e^{-\frac{k}{2}\phi(\Phi^{\frac{t}{k}}(x))}\ol Z_j\bigr(\phi(x)-\phi(\Phi^{\frac{t}{k}}(x))\bigr)\Td u(\Phi^{\frac{t}{k}}(x))dtd\eta)d\ol z_j.
\end{split}
\end{equation}
Since $\ddbar_bT\phi=0$, we have $\ol Z_j\phi(x)=\ol Z_j\phi(\Phi^{\frac{t}{k}}(x))$, $j=1,\ldots,n-1$, $t\in{\rm Supp\,}\chi_M$. From this and \eqref{sp3-eXI}, \eqref{sp3-eX} follows. 

\section{The asymptotic behaviour of $(Q^{(0)}_{M,k}\pit^{(0)}_{k,\leq k\nu_k})(x)$} 

We will use the same notations as before. We recall that we work with the assumption that $Y(0)$ and $Y(1)$ hold at each point of $X$. We first need

\begin{thm} \label{s4-t1} 
For any sequence $\nu_k>0$ with $\nu_k\To0$ as $k\To\infty$, there is a constant $C>0$ independent of $k$,
such that
\begin{equation} \label{s4-e0}
\abs{(Q^{(0)}_{M,k}\pit^{(0)}_{k,\,\leqslant  k\nu_k}\ol{Q^{(0)}_{M,k}})(x)}\leqslant  Ck^n
\end{equation} 
and
\begin{equation}\label{s4-e1}
\abs{(Q^{(0)}_{M,k}\pit^{(0)}_{k,\,\leqslant  k\nu_k})(x)}\leqslant Ck^n,
\end{equation}
for all $x\in X$, $k>0$. Recall that $(Q^{(0)}_{M,k}\pit^{(0)}_{k,\leq k\nu_k})(x)$ and $(Q^{(0)}_{M,k}\pit^{(0)}_{k,\leq k\nu_k}\ol{Q^{(0)}_{M,k}})(x)$ are given by \eqref{e-seII} and \eqref{e-seIII} respectively.
\end{thm} 

\begin{proof}
Let $\nu_k>0$ be any sequence with $\nu_k\To0$ as $k\To\infty$. 
Let $f_j\in C^\infty(X,L^k)$, $j=1,\ldots,d_k$, be an orthonormal frame for $\cH^0_{b,\leq k\nu_k}(X,L^k)$.  
From \eqref{s2-e10}, we see that for each $x\in X$, 
\begin{equation}\label{s4-e5}
\begin{split}
(Q^{(0)}_{M,k}\pit^{(0)}_{k,\leq k\nu_k}\ol{Q^{(0)}_{M,k}})(x)&=\sum^{d_k}_{j=1}\abs{(Q^{(0)}_{M,k}f_j)(x)}^2_{h^{L^k}}\\
&=\sup_{\alpha\in\cH^0_{b,\leq k\nu_k}(X,L^k),\norm{\alpha}_{h^{L^k}}=1}\abs{(Q^{(0)}_{M,k}\alpha)(x)}^2_{h^{L^k}}. 
\end{split}
\end{equation} 
In view of \eqref{s2-e6}, we see that there is a constant $C>0$ independent of $k$ such that
\begin{equation}\label{s4-e4}
\pit^{(0)}_{k,\leq k\nu_k}(x)=\sum^{d_k}_{j=1}\abs{f_j(x)}^2_{h^{L^k}}\leq Ck^n,\ \ \forall x\in X.
\end{equation} 
For $\alpha\in\cH^0_{b,\leq k\nu_k}(X,L^k)$, $\norm{\alpha}_{h^{L^k}}=1$, we have 
\begin{equation}\label{s4-e6}
\abs{\alpha(x)}^2_{h^{L^k}}\leq\pit^{(0)}_{k,\leq k\nu_k}(x)\leq Ck^n,\ \ \forall x\in X,
\end{equation}
where $C>0$ is a constant independent of $k$ and $\alpha$. From \eqref{s4-e6} and \eqref{sp3-eII}, it is easy to see that 
there is a constant $C_1>0$ independent of $k$ such that 
\begin{equation}\label{s4-e7}
\abs{(Q^{(0)}_{M,k}\alpha)(x)}^2_{h^{L^k}}\leq C_1k^n,\ \ \forall x\in X,\ \ \forall \alpha\in\cH^0_{b,\leq k\nu_k}(X,L^k), 
\norm{\alpha}_{h^{L^k}}=1.
\end{equation} 
From \eqref{s4-e7} and \eqref{s4-e5}, \eqref{s4-e0} follows.

We have 
\begin{equation}\label{s4-e3}
\begin{split}
\abs{(Q^{(0)}_{M,k}\pit^{(0)}_{k,\leq k\nu_k})(x)}&=\abs{\sum^{d_k}_{j=1}\langle(Q^{(0)}_{M,k}f_j)(x)|f_j(x)\rangle_{h^{L^k}}}\\
&\quad\leq \Bigr(\sum^{d_k}_{j=1}\abs{(Q^{(0)}_{M,k}f_j)(x)}^2_{h^{L^k}}\Bigr)^{\frac{1}{2}}
\Bigr(\sum^{d_k}_{j=1}\abs{f_j(x)}^2_{h^{L^k}}\Bigr)^{\frac{1}{2}}.
\end{split}
\end{equation} 
From \eqref{s4-e3}, \eqref{s4-e4}, \eqref{s4-e5} and \eqref{s4-e0}, \eqref{s4-e1} follows.
\end{proof}

Fix a point $p\in X$. 
Let $x=(x_1,\ldots,x_{2n-1})=(z,\theta)$ be canonical coordinates of $X$ defined in 
some small neighborhood $D$ of $p$ and let $s$ be a local section of $L$ on $D$, $\abs{s}^2_{h^L}=e^{-\phi}$. 
We take $x$ and $s$ so that \eqref{scat-e1} hold.
Until further notice, we work with the local coordinates $x$ and the local section $s$ and we will use the same notations as section 4. We identify $D$ with some open set in $\Complex^{n-1}\times\Real$. Put
\begin{equation} \label{s5-epf2}
\begin{split}
u(z,\theta)&=(2\pi)^{-\frac{n}{2}}\bigr(\int_{\Real_{p,0}}\det(M^\phi_p+2s\mathcal{L}_p)ds\bigr)^{-\frac{1}{2}}\\
&\quad\times
\int e^{i\theta\xi+\frac{\beta\theta}{2}+(-\xi+\frac{i}{2}\beta)\sum^{n-1}_{j=1}\lambda_j\abs{z_j}^2}\det(M^\phi_p+2\xi\mathcal{L}_p)\mathds{1}_{\Real_{p,0}}(\xi)d\xi.
\end{split}
\end{equation} 
$u(z,\theta)\in C^\infty(\Complex^{n-1}\times\Real)$. 
We remind that $\Real_{p,0}$ is given by \eqref{e-seaI}. Set 
\begin{equation}\label{s5-epf3}
\alpha_k=k^{\frac{n}{2}}s^k\chi_1(\frac{\sqrt{k}}{\log k}z,\frac{\sqrt{k}}{\log k}\theta)u(\sqrt{k}z,k\theta)\in C^\infty_0(D,L^k),
\end{equation}
where $\chi_1\in C^\infty$, $0\leq\chi_1\leq1$, 
\[{\rm Supp\,}\chi_1\subset\set{(z,\theta)\in\Complex^{n-1}\times\Real;\, \abs{z}\leq1,\abs{\theta}\leq1},\] $\chi_1(z,\theta)=1$ if $\abs{z}\leq\frac{1}{2}$, $\abs{\theta}\leq\frac{1}{2}$. We notice that
\[{\rm Supp\,}\alpha_k\subset\set{(z,\theta)\in\Complex^{n-1}\times\Real;\, \abs{z}\leq\frac{\log k}{\sqrt{k}}, \abs{\theta}\leq\frac{\log k}{\sqrt{k}}}.\] 
Thus, for $k$ large, ${\rm Supp\,}\alpha_k\in D$ and $\alpha_k$ is well-defined. 
The following is well-known (see section 5 in~\cite{HM09})

\begin{prop}\label{s5-ppf1} 
With the notations used above, we have 
\begin{equation}\label{s5-epf3-1}
\lim_{k\To\infty}k^{-n}\abs{\alpha_k(0)}^2_{h^{L^k}}=(2\pi)^{-n}\int_{\Real_{p,0}}\det(M^\phi_p+2s\mathcal{L}_p)ds,
\end{equation}
\begin{equation}\label{s5-epf4}
\lim_{k\To\infty}\norm{\alpha_k}_{h^{L^k}}=1,
\end{equation}
\begin{equation}\label{s5-epf5}
\lim_{k\To\infty}\norm{(\frac{1}{k}\Box^{(0)}_{b,k}\alpha_k)^m}_{h^{L^k}}=0,\ \ \forall m\in\mathbb N,
\end{equation}
and there is a sequence $\gamma_k>0$, independent of the point $p$ and tending to zero as $k\To\infty$, such that 
\begin{equation}\label{s5-epf6}
(\frac{1}{k}\Box^{(0)}_{b,k}\alpha_k\ |\ \alpha_k)_{h^{L^k}}\leq\gamma_k,\ \ \forall k>0.
\end{equation}
\end{prop} 

We have the following 

\begin{prop}\label{s5-ppf2}
Let $\nu_k>0$ be any sequence with $\lim_{k\To\infty}\frac{\gamma_k}{\nu_k}=0$ and $\nu_k\To0$ as $k\To\infty$, where $\gamma_k$ is as in \eqref{s5-epf6}. Let $\alpha_k$ be as in \eqref{s5-epf3}. Let 
\begin{equation}\label{s5-epf7}
\begin{split}
&\alpha_k=\alpha^1_k+\alpha^2_k,\\
&\alpha^1_k\in\cH^0_{b,\leq k\nu_k}(X,L^k),\ \ \alpha^2_k\in\cH^0_{b,>k\nu_k}(X,L^k).
\end{split} 
\end{equation}
Then, 
\begin{equation}\label{s5-epf8}
\lim_{k\To\infty}\norm{\alpha^1_k}_{h^{L^k}}=1
\end{equation} 
and 
\begin{equation}\label{s5-epf9}
\lim_{k\To\infty}k^{-n}\abs{\alpha^1_k(0)}^2_{h^{L^k}}=(2\pi)^{-n}\int_{\Real_{p,0}}\det(M^\phi_p+2s\mathcal{L}_p)ds.
\end{equation} 
Moreover, on $D$, we put 
\begin{equation}\label{s5-epf10}
\alpha^2_k=k^{\frac{n}{2}}s^k\beta^2_k,\ \ \beta^2_k\in C^\infty(D).
\end{equation} 
Fix $r>0$. Then, for every $\varepsilon>0$, there is a $k_0>0$ such that for all $k\geq k_0$, we have $F_k(D_{2r})\subset D$ and
\begin{equation}\label{s5-epf11}
\abs{\beta^2_k(\frac{z}{\sqrt{k}},\frac{\theta}{k})}\leq\varepsilon,\ \ \forall (z,\theta)\in D_r.
\end{equation} 
In particular, 
\begin{equation}\label{s5-epf12}
\lim_{k\To\infty}\abs{\beta^2_k(\frac{z}{\sqrt{k}},\frac{\theta}{k})}=0,\ \ \forall (z,\theta)\in D.
\end{equation}
\end{prop}

\begin{proof}
From \eqref{s5-sp2}, we have
\[\norm{\alpha^2_k}^2_{h^{L^k}}\leqslant \frac{1}{k\nu_k}\big(\Box^{(0)}_{b,k}\alpha^2_k\ \big|\ \alpha^2_k\big)_{h^{L^k}}\leqslant \frac{1}{k\nu_k}\big(\Box^{(0)}_{b,k}\alpha_k\ \big|\ \alpha_k\big)_{h^{L^k}}\leqslant \frac{\gamma_k}{\nu_k}\To 0,\]
as $k\To\infty$.
Thus, $\lim_{k\To\infty}\norm{\alpha^2_k}_{h^{L^k}}=0$. Since $\norm{\alpha_k}_{h^{L^k}}\To1$ as $k\To\infty$, \eqref{s5-epf8} follows.

Now, we prove \eqref{s5-epf11}. As \eqref{s5-epf10}, 
on $D$, we write $\alpha^2_k=s^kk^{\frac{n}{2}}\beta^2_k$,
$\beta^2_k\in C^\infty(D)$. From \eqref{scat-e19}, we know that
\begin{equation} \label{s5-epf13}
\begin{split}
\sup_{(z,\theta)\in D_r}\abs{F^*_k\beta^2_k(z,\theta)}^2&=\sup_{(z,\theta)\in D_r}\abs{\beta^2_k(\frac{z}{\sqrt{k}},\frac{\theta}{k})}^2\\
&\leqslant  C_{r}\Bigr(\norm{F^*_k\beta^2_k}^2_{kF^*_k\phi,D_{2r}}+\sum^n_{m=1}\norm{(\Box^{(q)}_{k,k\phi,(k)})^mF^*_k\beta^2_k}^2_{kF^*_k\phi,D_{2r}}\Bigr),
\end{split}
\end{equation}
where $C_r>0$ is independent of $k$. Now, we have
\begin{equation} \label{s5-epf14}
\norm{F^*_k\beta^2_k}^2_{kF^*_k\phi,D_{2r}}\leqslant \norm{\alpha^2_k}^2_{h^{L^k}}\To0,\ \ \mbox{as $k\To\infty$}.
\end{equation}
Moreover, from \eqref{scat-e14}, it is easy to can check that for all $m\in\mathbb N$,
\begin{equation} \label{s5-epf15}
\begin{split}
\Big\|(\Box^{(q)}_{b,k\phi,(k)})^mF^*_k\beta^2_k\Big\|^2_{kF^*_k\phi,D_{2r}}&\leqslant \Big\|(\tfrac{1}{k}\Box^{(q)}_{b,k})^m\alpha^2_k\Big\|^2_{h^{L^k}}\\
&\leqslant \Big\|(\tfrac{1}{k}\Box^{(q)}_{b,k})^m\alpha_k\Big\|^2_{h^{L^k}}\To0\ \ \text{as $k\To\infty$}.
\end{split}
\end{equation}
Here we used \eqref{s5-epf5}. Combining \eqref{s5-epf13} with \eqref{s5-epf14} and \eqref{s5-epf15}, 
\eqref{s5-epf11} follows.

From \eqref{s5-epf11}, we deduce
\begin{equation}\label{s5-ebis}
\lim_{k\To\infty}\abs{F^*_k\beta^2_k(0)}^2=\lim_{k\To\infty}\abs{\beta^2_k(0)}^2=
\lim_{k\To\infty}k^{-n}\abs{\alpha^2_k(0)}^2_{h^{L^k}}=0.
\end{equation}
From this and \eqref{s5-epf3-1}, \eqref{s5-epf9} follows.
\end{proof} 

Now, we can prove 

\begin{thm}\label{s6-t1}
Let $\delta_k=\min\set{\mu_k,\gamma_k}$, where $\mu_k$ is as in Theorem~\ref{s2-t1} and $\gamma_k$ is as in \eqref{s5-epf6}. Let $\nu_k>0$ be any sequence 
with $\lim_{k\To\infty}\frac{\delta_k}{\nu_k}=0$ and $\nu_k\To0$ as $k\To\infty$. Then, 
\begin{equation} \label{s4-e2}
\begin{split}
&\lim_{k\To\infty}k^{-n}(Q^{(0)}_{M,k}\pit^{(0)}_{k,\,\leqslant  k\nu_k})(x)\\
&=(2\pi)^{-n}\int e^{it\xi}\hat\psi(t)\chi_M(t)\det(M^\phi_x+2\xi\mathcal{L}_x)\mathds{1}_{\Real_{x,0}}(\xi)dtd\xi
\end{split}
\end{equation}
for all $x\in X$, where $\psi(\eta)$ is as in the discussion after \eqref{sp3-eI} and $\chi_M(t)$ is given by \eqref{sp3-eM0}, $\hat\psi(t)=\int e^{-it\eta}\psi(\eta)d\eta$. We remind that ${\rm Supp\,}\psi\bigcap \Real_{x,1}=\emptyset$, for every $x\in X$.
\end{thm}

\begin{proof}
Let $\nu_k>0$ be any sequence with $\lim_{k\To\infty}\frac{\delta_k}{\nu_k}=0$ and $\nu_k\To0$ as $k\To\infty$. 
Fix a point $p\in X$. Let $x=(x_1,\ldots,x_{2n-1})=(z,\theta)$ be canonical coordinates of $X$ defined in 
some small neighborhood $D$ of $p$ and let $s$ be a local section of $L$ on $D$, $\abs{s}^2_{h^L}=e^{-\phi}$. 
As before we take $x$ and $s$ so that \eqref{scat-e1} hold 
and let $\alpha^1_k\in\cH^0_{b,\leq k\nu_k}(X,L^k)$ be as in \eqref{s5-epf7}. 
We take 
\[f^1_k:=\frac{\alpha^1_k}{\norm{\alpha^1_k}_{h^{L^k}}},f^2_k,\ldots,f^{d_k}_k\]
to be an orthonormal frame for $\cH^0_{b,\leq k\nu_k}(X,L^k)$. From \eqref{s5-epf8}, \eqref{s5-epf9} and \eqref{s2-e6I}, we conclude that
\begin{equation}\label{s5-epf16}
\lim_{k\To\infty}k^{-n}\abs{f^1_k(0)}^2_{h^{L^k}}=\lim_{k\To\infty}k^{-n}\pit^{(0)}_{k,\leq k\nu_k}(0)
=(2\pi)^{-n}\int_{\Real_{p,0}}\det(M^\phi_p+2s\mathcal{L}_p)ds.
\end{equation} 
Thus, 
\begin{equation}\label{s5-epf17}
\lim_{k\To\infty}k^{-n}\sum^{d_k}_{j=2}\abs{f^j_k(0)}^2_{h^{L^k}}=0.
\end{equation} 
Now, 
\begin{equation}\label{s5-epf18}
(Q^{(0)}_{M,k}\pit^{(0)}_{k,\leq k\nu_k})(0)=\langle(Q^{(0)}_{M,k}f^1_k)(0)|f^1_k(0)\rangle_{h^{L^k}}+\sum^{d_k}_{j=2}
\langle(Q^{(0)}_{M,k}f^j_k)(0)|f^j_k(0)\rangle_{h^{L^k}}.
\end{equation} 
From \eqref{s4-e0} and \eqref{s5-epf17}, we have 
\begin{equation}\label{s5-epf19}
\begin{split}
&\lim_{k\To\infty}k^{-n}\abs{\sum^{d_k}_{j=2}\langle(Q^{(0)}_{M,k}f^j_k)(0)|f^j_k(0)\rangle_{h^{L^k}}}\\
&\leq\lim_{k\To\infty}k^{-n}\sqrt{\sum^{d_k}_{j=2}\abs{(Q^{(0)}_{M,k}f^j_k)(0)}^2_{h^{L^k}}}\sqrt{\sum^{d_k}_{j=2}\abs{f^j_k(0)}^2_{h^{L^k}}}\To0,\ \ \mbox{as $k\To\infty$}.
\end{split}
\end{equation}
Combining \eqref{s5-epf19} with \eqref{s5-epf18}, we conclude that
\begin{equation}\label{s5-epf20a}
\lim_{k\To\infty}k^{-n}(Q^{(0)}_{M,k}\pit^{(0)}_{k,\leq\nu_k})(0)=\lim_{k\To\infty}k^{-n}\langle(Q^{(0)}_{M,k}f^1_k)(0)|f^1_k(0)\rangle_{h^{L^k}}.
\end{equation} 
Let $\alpha^2_k$ be as in \eqref{s5-epf7}. From \eqref{s5-epf11} and the definition of $Q^{(0)}_{M,k}$(see \eqref{sp3-eII}) , it is not difficult to see that 
\begin{equation}\label{eTWa1}
\lim_{k\To\infty}k^{-n}\langle(Q^{(0)}_{M,k}\alpha^2_k)(0)|\alpha^2_k(0)\rangle_{h^{L^k}}=0.
\end{equation}
Combining \eqref{eTWa1} with \eqref{s5-epf20a} and \eqref{s5-epf8}, we deduce 
\begin{equation}\label{s5-epf20}
\lim_{k\To\infty}k^{-n}(Q^{(0)}_{M,k}\pit^{(0)}_{k,\leq\nu_k})(0)=\lim_{k\To\infty}
k^{-n}\langle(Q^{(0)}_{M,k}\alpha_k)(0)|\alpha_k(0)\rangle_{h^{L^k}},
\end{equation} 
where $\alpha_k$ is as in \eqref{s5-epf3}.
On $D$, we put 
\begin{equation}\label{s5-epf20-1}
Q^{(0)}_{M,k}\alpha_k=s^kq_k,\ \ q_k\in C^\infty(D).
\end{equation}
By the definitions of $Q^{(0)}_{M,k}$ and $\alpha_k$(see \eqref{sp3-eII} and \eqref{s5-epf3}), we can check that 
\begin{equation}\label{s5-epf21}
\begin{split}
&q_k(0)\\
&=k^{\frac{n}{2}}(2\pi)^{-\frac{n}{2}}\bigr(\int_{\Real_{p,0}}\det(M^\phi_p+2s\mathcal{L}_p)ds\bigr)^{-\frac{1}{2}}\times\\
&\int e^{-it\eta}\psi(\eta)\chi_M(t)e^{-\frac{k}{2}\phi(0,\frac{t}{k})}\chi_1(0,\frac{t}{\sqrt{k}\log k})e^{it\xi+\frac{\beta}{2}t}\mathds{1}_{\Real_{p,0}}(\xi)\det (M^\phi_p+2\xi\mathcal{L}_p)d\xi dtd\eta.
\end{split}
\end{equation} 
We notice that $\frac{k}{2}\phi(0,\frac{t}{k})=\frac{\beta}{2}t+\epsilon_k(t)$, where $\epsilon_k(t)\To0$ 
as $k\To\infty$, uniformly on ${\rm Supp\,}\chi_M$ and $\chi_1(0,\frac{t}{\sqrt{k}\log k})\To1$ as $k\To\infty$, uniformly on ${\rm Supp\,}\chi_M$. Combining this observation with \eqref{s5-epf21}, \eqref{s5-epf3} and \eqref{s5-epf2}, we can check that
\begin{equation}\label{s5-epf22}
\begin{split}
&\lim_{k\To\infty}k^{-n}\langle (Q^{(0)}_{M,k}\alpha_k)(0)|\alpha_k(0)\rangle_{h^{L^k}}\\
&=\lim_{k\To\infty}k^{-\frac{n}{2}}q_k(0)\ol{u(0,0)}e^{-k\phi(0)}\\
&=(2\pi)^{-n}\int e^{-it\eta+it\xi}\psi(\eta)\chi_M(t)\mathds{1}_{\Real_{p,0}}(\xi)\det (M^\phi_p+2\xi\mathcal{L}_p)d\xi dtd\eta\\
&=(2\pi)^{-n}\int e^{it\xi}\hat\psi(t)\chi_M(t)\mathds{1}_{\Real_{p,0}}(\xi)\det (M^\phi_p+2\xi\mathcal{L}_p)d\xi dt,
\end{split}
\end{equation} 
where $\hat\psi(t):=\int e^{-it\eta}\psi(\eta)d\eta$, $u$ is as in \eqref{s5-epf2}. From \eqref{s5-epf22}
and \eqref{s5-epf20}, \eqref{s4-e2} follows. We get Theorem~\ref{s6-t1}. 
\end{proof}

We need 

\begin{thm}\label{s6-t2}
Let $\delta_k>0$, $\delta_k\To0$, as $k\To\infty$, be as in Theorem~\ref{s6-t1} and let $\nu_k>0$ be any sequence with $\lim_{k\To\infty}\frac{\delta_k}{\nu_k}=0$ and $\nu_k\To0$ as $k\To\infty$. 
Then there is a $k_0>0$ such that for all $k\geq k_0$, 
\begin{equation}\label{s6-e1.1}
\begin{split}
&\abs{\int_X(Q^{(0)}_{M,k}\pit^{(0)}_{k,\leq k\nu_k})(x)dv_X(x)}\\
&\quad\geq\frac{k^n}{2}(2\pi)^{1-n}\int_X\Bigr(\int\psi(\xi)\det(M^\phi_x+2\xi\mathcal{L}_x)
\mathds{1}_{\Real_{x,0}}(\xi)d\xi\Bigr)dv_X(x).
\end{split}
\end{equation} 
\end{thm}

\begin{proof}
For each $x\in X$, put 
\begin{equation}\label{s6-eo1}
C(x):=(2\pi)^{-n}\int e^{it\xi}\hat\psi(t)\chi_M(t)\det(M^\phi_x+2\xi\mathcal{L}_x)\mathds{1}_{\Real_{x,0}}(\xi)d\xi dt.
\end{equation}
From \eqref{s4-e1}, \eqref{s4-e2} and the Lebesgue dominated Theorem, we conclude that 
\[
\int_X(Q^{(0)}_{M,k}\pit^{(0)}_{k,\leq k\nu_k})(x)dv_X(x)=k^n\int_XC(x)dv_X(x)+o(k^n)
\]
and hence 
\begin{equation}\label{s6-e0}
\abs{\int_X(Q^{(0)}_{M,k}\pit^{(0)}_{k,\leq k\nu_k})(x)dv_X(x)}\geq k^n\abs{\int_XC(x)dv_X(x)}+o(k^n).
\end{equation}
We first claim that for each $x\in X$, $C(x)$ is real. We notice that $\ol{\hat\psi(t)}=\hat\psi(-t)$ and $\chi_M(t)=\chi_M(-t)$. From this observation, we can check that
\[\begin{split}
\ol{C(x)}&=(2\pi)^{-n}\int e^{-it\xi}\ol{\hat\psi(t)}\chi_M(t)\det(M^\phi_x+2\xi\mathcal{L}_x)
\mathds{1}_{\Real_{x,0}}(\xi)d\xi dt\\
&=(2\pi)^{-n}\int e^{-it\xi}\hat\psi(-t)\chi_M(-t)\det(M^\phi_x+2\xi\mathcal{L}_x)
\mathds{1}_{\Real_{x,0}}(\xi)d\xi dt\\
&=(2\pi)^{-n}\int e^{it\xi}\hat\psi(t)\chi_M(t)\det(M^\phi_x+2\xi\mathcal{L}_x)
\mathds{1}_{\Real_{x,0}}(\xi)d\xi dt=C(x).
\end{split}\] 
Thus, $C(x)$ is real. 

Now, we claim that $\int_XC(x)dv_X(x)$ is positive and
\begin{equation}\label{s6-e1}
\int_XC(x)dv_X(x)>\frac{1}{2}(2\pi)^{1-n}\int_X\Bigr(\int\psi(\xi)\det(M^\phi_x+2\xi\mathcal{L}_x)
\mathds{1}_{\Real_{x,0}}(\xi)d\xi\Bigr)dv_X(x).
\end{equation} 
We have
\begin{equation}\label{s6-e2}
\begin{split}
&C(x)=(2\pi)^{-n}\int e^{it\xi}\hat\psi(t)\det(M^\phi_x+2\xi\mathcal{L}_x)
\mathds{1}_{\Real_{x,0}}(\xi)d\xi dt\\
&\quad+(2\pi)^{-n}\int e^{it\xi}\hat\psi(t)(\chi_M(t)-1)\det(M^\phi_x+2\xi\mathcal{L}_x)
\mathds{1}_{\Real_{x,0}}(\xi)d\xi dt\\
&=(2\pi)^{1-n}\int \psi(\xi)\det(M^\phi_x+2\xi\mathcal{L}_x)
\mathds{1}_{\Real_{x,0}}(\xi)d\xi \\
&\quad+(2\pi)^{-n}\int e^{it\xi}\hat\psi(t)(\chi_M(t)-1)\det(M^\phi_x+2\xi\mathcal{L}_x)
\mathds{1}_{\Real_{x,0}}(\xi)d\xi dt.
\end{split}
\end{equation} 
Here we used Fourier's inversion formula.
Since $0\leq\chi_M\leq 1$ and $\chi_M=1$ if $-M\leq t\leq M$, we have 
\begin{equation}\label{s6-e3}
\begin{split}
&\abs{\int e^{it\xi}\hat\psi(t)(\chi_M(t)-1)\det(M^\phi_x+2\xi\mathcal{L}_x)
\mathds{1}_{\Real_{x,0}}(\xi)d\xi dt}\\ 
&=\abs{\int_{\abs{t}\geq M}e^{it\xi}\hat\psi(t)(\chi_M(t)-1)\det(M^\phi_x+2\xi\mathcal{L}_x)\mathds{1}_{\Real_{x,0}}(\xi)d\xi dt}\\
&\leq\int_{\abs{t}\geq M}\abs{\hat\psi(t)}dt\abs{\int\det(M^\phi_x+2\xi\mathcal{L}_x)
\mathds{1}_{\Real_{x,0}}(\xi)d\xi}\\
&\leq\frac{2C_0}{M}\int \det(M^\phi_x+2\xi\mathcal{L}_x)
\mathds{1}_{\Real_{x,0}}(\xi)d\xi,
\end{split}
\end{equation} 
where $C_0=\sup_{t\in\Real}t^2\abs{\hat\psi(t)}$. Combining \eqref{s6-e3} with \eqref{s6-e2}, 
we get 
\[\begin{split}
C(x)&\geq (2\pi)^{1-n}\int\psi(\xi)\det(M^\phi_x+2\xi\mathcal{L}_x)
\mathds{1}_{\Real_{x,0}}(\xi)d\xi\\
&\quad-\frac{2C_0}{M}(2\pi)^{-n}\int\det(M^\phi_x+2\xi\mathcal{L}_x)
\mathds{1}_{\Real_{x,0}}(\xi)d\xi.
\end{split}\]
Combining this with \eqref{sp3-eI1}, \eqref{s6-e1} follows. 

From \eqref{s6-e1} and \eqref{s6-e0}, we obtain \eqref{s6-e1.1}.
\end{proof} 

\section{The asymptotic behaviour of $(Q^{(1)}_{M,k}\pit^{(1)}_{k,\leq k\nu_k}\ol{Q^{(1)}_{M,k}})(x)$} 

We will use the same notations as before. 
Fix $p\in X$. 
Let $x=(x_1,\ldots,x_{2n-1})=(z,\theta)$ be canonical coordinates of $X$ defined in 
some small neighborhood $D$ of $p$ and let $s$ be a local section of $L$ on $D$, $\abs{s}^2_{h^L}=e^{-\phi}$. 
We take $x$ and $s$ so that \eqref{scat-e1} hold.
Until further notice, we work with the local coordinates $x$ and the local section $s$. We also write $t$ to denote the coordinate $\theta$. We identify $D$ with some open set in $H_n=\Complex^{n-1}\times\Real$. Let $\nu_k>0$ be any sequence with $\nu_k\To0$ as $k\To\infty$. We are going to estimate 
$\limsup_{k\To\infty}k^{-n}(Q^{(1)}_{M,k}\pit^{(1)}_{k,\leq k\nu_k}\ol{Q^{(1)}_{M,k}})(p)$. For the convenience of the reader we recall some notations we used before. Let $e_j(z,\theta)$, $j=1,\ldots,n-1$, denote the basis of $T^{*(0,1)}X$, dual to $\ol U_j(z,\theta)$, $j=1,\ldots,n-1$, where $U_j$, $j=1,\ldots,n-1$, are as in \eqref{epch}. For $f\in\Omega^{0,1}(X,L^k)$, we write $f=\sum^{n-1}_{j=1}f_je_j$, $f_j\in C^\infty(X,L^k)$, 
$j=1,\ldots,n-1$. We call $f_j$ the component of $f$ along $e_j$. As \eqref{s2-e7-1}, for $j=1,\ldots,n-1$, we
define 
\begin{equation}\label{s7-e0} 
(Q^{(1)}_{M,k}S^{(1)}_{k,\,\leqslant k\nu_k,j}\ol{Q^{(1)}_{M,k}})(y):=\sup_{\alpha\in\,\cH_{b,\leq k\nu_k}^1(X, L^k),\,\norm{\alpha}_{h^{L^k}}=1}\abs{(Q^{(1)}_{M,k}\alpha)_j(y)}^2_{h^{L^k}}\,,
\end{equation} 
where $(Q^{(1)}_{M,k}\alpha)_j$ denotes the component of $Q^{(1)}_{M,k}\alpha$ along $e_j$. From \eqref{s2-e9}, we know that
\begin{equation}\label{s7-e0-1}
(Q^{(1)}_{M,k}\pit^{(1)}_{k,\,\leqslant k\nu_k}\ol{Q^{(1)}_{M,k}})(y)=\sum^{n-1}_{j=1}(Q^{(1)}_{M,k}S^{(1)}_{k,\,\leqslant k\nu_k,j}\ol{Q^{(1)}_{M,k}})(y),\ \ \forall y\in D.
\end{equation}

We consider $H_n$. Let $\psi(\eta)$ be as in the discussion after \eqref{sp3-eI} and let $\chi_M(t)$ be as in \eqref{sp3-eM0}.
The operator $Q^{(1)}_{M,H_n}$ is a continuous operator $\Omega^{0,1}(H_n)\To\Omega^{0,1}(H_n)$ defined as follows. Let $u\in\Omega^{0,1}(H_n)$. We write $u=\sum^{n-1}_{j=1}u_jd\ol z_j$, $u_j\in C^\infty(H_n)$, $j=1,\ldots,n-1$. Then, 
\begin{equation}\label{s7-eI}
\begin{split}
(Q^{(1)}_{M,H_n}u)(z,\theta)
&=\sum^{n-1}_{j=1}(\int e^{-it\eta}\psi(\eta)\chi_M(t)e^{-\frac{\beta}{2}(t+\theta)}u_j(z,t+\theta)dtd\eta)d\ol z_j\\
&:=\int e^{-it\eta}\psi(\eta)\chi_M(t)e^{-\frac{\beta}{2}(t+\theta)}u(z,t+\theta)dtd\eta.
\end{split}
\end{equation}
We remind that $\beta$ is as in \eqref{scat-e1}. For $j=1,\ldots,n-1$, put (Compare \eqref{s7-e0})
\begin{equation}\label{s7-eII}
\begin{split}
&(Q^{(1)}_{M,H_n}S^{(1)}_{j,H_n}\ol{Q^{(1)}_{M,H_n}})(0)\\
&=\sup\set{\abs{(Q^{(1)}_{M,H_n}\alpha)_j(0)}^2;\, \alpha\in\Omega^{0,1}(H_n),\Box^{(1)}_{b,H_n}\alpha=0,\norm{\alpha}_{\phi_0}=1},
\end{split}
\end{equation}
where 
\[(Q^{(1)}_{M,H_n}\alpha)(x)=\sum^{n-1}_{j=1}(Q^{(1)}_{M,H_n}\alpha)_j(x)d\ol z_j,\ \ (Q^{(1)}_{M,H_n}\alpha)_j\in 
C^\infty(H_n),\ \ j=1,\ldots,n-1,\] 
and 
\[\norm{\alpha}^2_{\phi_0}=\int\abs{\alpha(z,\theta)}^2e^{-\phi_0(z,\theta)}dv(z)d\theta,\ \ 
dv(z)=2^{n-1}dx_1dx_2\ldots dx_{2n-1}.\]
We recall that $\phi_0$ is as in the discussion after \eqref{scat-e20}. 
We first need 

\begin{thm} \label{s7-t1}
We have
\[
\limsup_{k\To\infty}k^{-n}(Q^{(1)}_{M,k}\pit^{(1)}_{k,\leq k\nu_k}\ol{Q^{(1)}_{M,k}})(0)\leqslant\sum^{n-1}_{j=1}
(Q^{(1)}_{M,H_n}S^{(1)}_{j,H_n}\ol{Q^{(1)}_{M,H_n}})(0).
\]
\end{thm}

\begin{proof}
Fix $j\in\set{1,2,\ldots,n-1}$. We claim that
\begin{equation} \label{s7-eIII}
\limsup_{k\To\infty}k^{-n}(Q^{(1)}_{M,k}S^{(1)}_{k,\leq k\nu_k,j}\ol{Q^{(1)}_{M,k}})(0)\leqslant(Q^{(1)}_{M,H_n}S^{(1)}_{j,H_n}\ol{Q^{(1)}_{M,H_n}})(0).
\end{equation}
The definition \eqref{s7-e0} of $(Q^{(1)}_{M,k}S^{(1)}_{k,\leq k\nu_k,j}\ol{Q^{(1)}_{M,k}})(0)$
yields a sequence 
\[\alpha_{k_s}\in\cH^1_{b,\leq k_s\nu_{k_s}}(X, L^{k_s}),\ \  k_1<k_2<\ldots,\]
such that $\norm{\alpha_{k_s}}_{h^{L^{k_s}}}=1$ and
\begin{equation}\label{s7-eIV}
 \lim_{s\To\infty}k_s^{-n}\abs{(Q^{(1)}_{M,k_s}\alpha_{k_s})_j(0)}^2_{h^{L^{k_s}}}=\limsup_{k\To\infty}k^{-n}(Q^{(1)}_{M,k}S^{(1)}_{k,\leq k\nu_k,j}\ol{Q^{(1)}_{M,k}})(0)\,,
\end{equation}
where $(Q^{(1)}_{M,k_s}\alpha_{k_s})_j$ is the component
of $Q^{(1)}_{M,k_s}\alpha_{k_s}$ along $e_j$. On $D$, we write 
\[\alpha_{k_s}=s^{k_s}\Td\alpha_{k_s},\ \ \Td\alpha_{k_s}\in\Omega^{0,1}(D),\]
and on $D_{\log k_s}$, put
\[\gamma_{k_s}=k_s^{-\frac{n}{2}}F^*_{k_s}\Td\alpha_{k_s}\in F^*_{k_s}\Omega^{0,1}(D_{\log k_s}).\]
We recall that $F^*_{k_s}$ is the scaling map given by \eqref{scat-e6-1}.
It is not difficult to see that
\[\norm{\gamma_{k_s}}_{k_sF^*_{k_s}\phi,D_{\log k_s}}\leqslant1.\] 
Moreover, from \eqref{scat-e14} and \eqref{scat-e2}, it is straightforward to see that
\[\norm{(\Box^{(1)}_{b,k_s\phi,(k_s)})^m\gamma_{k_s}}_{k_sF^*_{k_s}\phi,D_{\log k_s}}\leq
\frac{1}{k^m_s}\norm{(\Box^{(1)}_{b,k_s})^m\alpha_{k_s}}_{h^{L^{k_s}}}\leq\nu^m_{k_s},\ \ \forall m\in\mathbb N.\] 
Proposition~\ref{scat-p3} yields a subsequence $\set{\gamma_{k_{s_u}}}$ of $\set{\gamma_{k_s}}$ such that for each $t$ in the set $\set{1,2,\ldots,n-1}$, $\gamma_{k_{s_u},t}$ converges uniformly with all derivatives on any compact subset of $H_n$ to a smooth function $\gamma_t$, where $\gamma_{k_{s_u},t}$ denotes the component of $\gamma_{k_{s_u}}$ along $e_t(\frac{z}{\sqrt{k}},\frac{\theta}{k})$. Set
$\gamma=\sum^{n-1}_{t=1}\gamma_td\ol z_t$. Then we have
$\Box^{(1)}_{b,H_n}\gamma=0$ and, by \eqref{scat-e21}, $\norm{\gamma}_{\phi_0}\leqslant  1$. Thus,
\begin{equation} \label{s7-eV}
\abs{(Q^{(1)}_{M,H_n}\gamma)_j(0)}^2\leqslant\frac{\abs{(Q^{(1)}_{M,H_n}\gamma)_j(0)}^2}{\norm{\gamma}^2_{\phi_0}}\leqslant  (Q^{(1)}_{M,H_n}S^{(1)}_{j,H_n}\ol{Q^{(1)}_{M,H_n}})(0),
\end{equation}
where 
\[Q^{(1)}_{M,H_n}\gamma=\sum^{n-1}_{t=1}(Q^{(1)}_{M,H_n}\gamma)_td\ol z_t,\ \ (Q^{(1)}_{M,H_n}\gamma)_t\in C^\infty(H_n),\ \ t=1,\ldots,n-1.\]
We claim that
\begin{equation} \label{s7-eVI}
\lim_{u\To\infty}k_{s_u}^{-n}\abs{(Q^{(1)}_{M,k_{s_u}}\alpha_{k_{s_u}})_j(0)}^2
=\abs{(Q^{(1)}_{M,H_n}\gamma)_j(0)}^2.
\end{equation}
We write
\[\Td\alpha_{k_{s_u}}=\sum^{n-1}_{j=1}\Td\alpha_{k_{s_u},j}e_j=\sum^{n-1}_{j=1}\hat\alpha_{k_{s_u},j}d\ol z_j.\]
Since $e_t=d\ol z_t+O(\abs{(z,\theta)}$, $t=1,\ldots,n-1$, we conclude that for all $t=1,\ldots,n-1$,
\begin{equation}\label{s7-eVI-I}
\lim_{u\To\infty}k_{s_u}^{-\frac{n}{2}}F^*_{k_{s_u}}\Td\alpha_{k_{s_u},t}=
\lim_{u\To\infty}k_{s_u}^{-\frac{n}{2}}F^*_{k_{s_u}}\hat\alpha_{k_{s_u},t}=\gamma_t.
\end{equation}
Moreover, from the definition of $Q^{(1)}_{M,k_{s_u}}$(see \eqref{sp3-eV}), it is easy to see that
\begin{equation}\label{s7-eVI-II}
\abs{(Q^{(1)}_{M,k_{s_u}}\alpha_{k_{s_u}})_j(0)}_{h^{L^k}}=\abs{\int e^{-it\eta}\psi(\eta)\chi_M(t)e^{-\frac{k_{s_u}}{2}(F^*_{k_{s_u}}\phi)(0,t)}F^*_{k_{s_u}}\hat\alpha_{k_{s_u},j}(0,t)dt}.
\end{equation}
Combining \eqref{s7-eVI-II} with \eqref{s7-eVI-I}, \eqref{s7-eI} and notice that 
$-\frac{k}{2}(F^*_k\phi)(0,t)\To-\frac{\beta}{2}t$, as $k\To\infty$, uniformly on ${\rm Supp\,}\chi_M$, \eqref{s7-eVI} follows.
The claim \eqref{s7-eIII} follows from \eqref{s7-eIV}, \eqref{s7-eV} and \eqref{s7-eVI}. Finally, \eqref{s7-eIII} and \eqref{s7-e0-1} imply the conclusion of the theorem.
\end{proof}

In order to estimate $\sum^{n-1}_{j=1}
(Q^{(1)}_{M,H_n}S^{(1)}_{j,H_n}\ol{Q^{(1)}_{M,H_n}})(0)$, we need the some preparation. Put 
\begin{equation}\label{s7-VIII-I}
\Phi_0=\sum^{n-1}_{j,t=1}\mu_{j,\,t}\ol z_jz_t,
\end{equation}
where $\mu_{j,t}$, $j,t=1,\ldots,n-1$, are as in \eqref{scat-e1}. Note that 
\[\phi_0(z,\theta)=\Phi_0(z)+\beta\theta.\]
For $q=0,1,\ldots,n-1$, we denote by $L^2_{(0,q)}(H_n, \Phi_0)$ the completion of $\Omega_0^{0,q}(H_n)$
with respect to the norm $\|\cdot\|_{\Phi_0}$, where
\[
\|u\|^2_{\Phi_0}=\int_{H_n}\abs{u}^2e^{-\Phi_0}dv(z)d\theta\,,\quad u\in\Omega_0^{(0,q)}(H_n)\,.\]

Let $u(z, \theta)\in\Omega^{0,1}(H_n)$ with $\norm{u}_{\phi_0}=1$, $\Box^{(1)}_{b,H_n}u=0$. Put
$v(z, \theta)=u(z, \theta)e^{-\frac{\beta}{2}\theta}$.
We have \[\int_{H_n}\!\abs{v(z,\theta)}^2e^{-\Phi_0(z)}dv(z)d\theta=1.\]
Choose $\chi(\theta)\in C^\infty_0(\Real)$ so that $\chi(\theta)=1$ when $\abs{\theta}<1$ and $\chi(\theta)=0$ when $\abs{\theta}>2$ and set $\chi_j(\theta)=\chi(\theta/j)$, $j\in\mathbb{N}$. Let
\begin{equation} \label{s7-eIX}
\hat v_j(z, \eta)=\int_{\Real}\!v(z,\theta)\chi_j(\theta)e^{-i\theta\eta}d\theta\in\Omega^{0,1}(H_n),\ j=1,2,\ldots.
\end{equation}
From Parseval's formula, we have
\begin{align*}
&\int_{H_n}\!\abs{\hat v_j(z,\eta)-\hat v_t(z,\eta)}^2e^{-\Phi_0(z)}d\eta dv(z)\\
&=2\pi\int_{H_n}\!\abs{v(z,\theta)}^2\abs{\chi_j(\theta)-\chi_t(\theta)}^2e^{-\Phi_0(z)}d\theta dv(z)\To0,\  j,t\To\infty.
\end{align*}
Thus, there is $\hat v(z, \eta)\in L^2_{(0,1)}(H_n, \Phi_0)$ such that $\hat v_j(z,\eta)\To\hat v(z, \eta)$ in $L^2_{(0,1)}(H_n, \Phi_0)$. We have 
\begin{equation}\label{alpc}
\int\abs{\hat v(z,\eta)}^2e^{-\Phi_0(z)}dv(z)d\eta=2\pi.
\end{equation}
We call $\hat v(z, \eta)$ the Fourier transform of $v(z, \theta)$ with respect to $\theta$. Formally,
\begin{equation} \label{s7-eX}
\hat v(z, \eta)=\int_{\Real}\! e^{-i\theta\eta}v(z,\theta)d\theta.
\end{equation} 


The following theorem is one of the main technical results in~\cite{HM09} (see section 3 in~\cite{HM09}, for the proof)

\begin{thm}\label{s7-t3}
With the notations used above.
Let $u(z, \theta)\in\Omega^{0,1}(H_n)$  with $\norm{u}_{\phi_0}=1$, 
$\Box^{(1)}_{b,H_n}u=0$ and let $\hat v(z, \eta)\in L^2_{(0,1)}(H_n, \Phi_0)$ be the Fourier transform of the function $u(z,\theta)e^{-\frac{\beta}{2}\theta}$ with respect to $\theta$(see the discussion before \eqref{s7-eX}). 
Then, for almost all $\eta\in\Real$, we have 
$\hat v(z,\eta)$ is smooth with respect to $z$
and 
\[\int_{\Complex^{n-1}}\abs{\hat v(z,\eta)}^2e^{-\Phi_0(z)}dv(z)<\infty\]
and 
\begin{equation}\label{s7-eXI}
\begin{split}
&\abs{\hat v(z,\eta)}^2\\
&\quad\leq (2\pi)^{-n+1}e^{\Phi_0(z)}\mathds{1}_{\mathbb R_{p,1}}(\eta)\abs{\det(M^\phi_p+2\eta\mathcal{L}_p)}
\int_{\Complex^{n-1}}\abs{\hat v(w,\eta)}^2e^{-\Phi_0(w)}dv(w)
\end{split}
\end{equation}
for all $z\in\Complex^{n-1}$. 
\end{thm}

Now, we can prove

\begin{prop}\label{s7-p1}
Let $u(z, \theta)\in\Omega^{0,1}(H_n)$  with $\norm{u}_{\phi_0}=1$, 
$\Box^{(1)}_{b,H_n}u=0$. We have 
\begin{equation}\label{s7-ea1}
\abs{(Q^{(1)}_{M,H_n}u)(0)}^2\leq\frac{E^2}{n-1},
\end{equation}
where $E$ is as in \eqref{change}.
\end{prop}

\begin{proof}
Let $\varphi\in C^\infty_0(\Complex^{n-1},\Real)$ such that $\int_{\Complex^{n-1}}\!\varphi(z)dv(z)=1$, $\varphi\geqslant0$, $\varphi(z)=0$ if $\abs{z}>1$. Put $g_m(z)=m^{2n-2}\varphi(mz)e^{\Phi_0(z)}$, $m=1,2,\ldots$.
Then, $\int_{\Complex^{n-1}}\!g_m(z)e^{-\Phi_0(z)}dv(z)=1$ and 
\begin{equation}\label{s7-ea2}
\begin{split}
(Q^{(1)}_{M,H_n}u)(0)&=\lim_{m\To\infty}\int e^{-it\eta}\psi(\eta)\chi_M(t)e^{-\frac{\beta}{2}t}e^{-\Phi_0(z)}g_m(z)u(z,t)dtdv(z)\\
&=\lim_{m\To\infty}\int\hat\psi(t)\chi_M(t)e^{-\frac{\beta}{2}t}e^{-\Phi_0(z)}g_m(z)u(z,t)dtdv(z).
\end{split}
\end{equation}
Choose $\chi(t)\in C^\infty_0(\Real)$ so that $\chi(t)=1$ when $\abs{t}<1$ and $\chi(t)=0$ when $\abs{t}>2$ and set $\chi_j(t)=\chi(t/j)$, $j\in\mathbb{N}$. For each $m$, we have 
\begin{equation}\label{GPCeI}
\begin{split}
&\int\hat\psi(t)\chi_M(t)e^{-\frac{\beta}{2}t}e^{-\Phi_0(z)}g_m(z)u(z,t)dtdv(z)\\
&=\lim_{j\To\infty}\int\hat\psi(t)\chi_M(t)e^{-\frac{\beta}{2}t}e^{-\Phi_0(z)}g_m(z)u(z,t)\chi_j(t)dtdv(z).
\end{split}
\end{equation} 
From Parseval's formula, we can check that for each $j$, 
\begin{equation}\label{GPCeII}
\begin{split}
&\int\hat\psi(t)\chi_M(t)e^{-\frac{\beta}{2}t}e^{-\Phi_0(z)}g_m(z)u(z,t)\chi_j(t)dtdv(z)\\
&=\frac{1}{2\pi}\int\alpha(\eta)\hat v_j(z,\eta)g_m(z)e^{-\Phi_0(z)}d\eta dv(z),
\end{split}
\end{equation} 
where $\hat v_j(z,\eta)$ is as in \eqref{s7-eIX} and 
\begin{equation}\label{GPCeIII}
\alpha(\eta)=\int e^{-it\eta}\hat\psi(\eta)\chi_M(t)dt.
\end{equation}
From \eqref{GPCeII} and \eqref{GPCeI}, we obtain for each $m$, 
\begin{equation}\label{GPCeIV}
\begin{split}
&\int\hat\psi(\eta)\chi_M(t)e^{-\frac{\beta}{2}t}e^{-\Phi_0(z)}g_m(z)u(z,t)dv(z)dt\\
&=\frac{1}{2\pi}\int\hat v(z,\eta)g_m(z)\alpha(\eta)e^{-\Phi_0(z)}dv(z)d\eta,
\end{split}
\end{equation}
where $\hat v(z,\eta)$ is as in \eqref{s7-eX}. Now, 
\[\begin{split}
\alpha(\eta)&=\int e^{-it\eta}\hat\psi(t)\chi_M(t)dt\\
&=\int e^{-it\eta}\hat\psi(t)dt+\int e^{-it\eta}\hat\psi(t)(\chi_M(t)-1)dt\\
&=(2\pi)\psi(\eta)+\alpha_1(\eta),
\end{split}\]
where 
\[\alpha_1(\eta)=\int e^{-it\eta}\hat\psi(t)(\chi_M(t)-1)dt.\]
Combining this with \eqref{GPCeIV}, we have 
\begin{equation}\label{GPCeV}
\begin{split}
&\int\hat\psi(t)\chi_M(t)e^{-\frac{\beta}{2}t}g_m(z)u(z,t)e^{-\Phi_0(z)}dv(z)dt\\
&=\int\hat v(z,\eta)g_m(z)\psi(\eta)e^{-\Phi_0(z)}dv(z)d\eta+\frac{1}{2\pi}\int\hat v(z,\eta)g_m(z)\alpha_1(\eta)e^{-\Phi_0(z)}dv(z)d\eta.
\end{split}
\end{equation}
Since $\hat v(z,\eta)\in L^2_{(0,1)}(H_n,\Phi_0)$, it is easy to see that 
\begin{equation}\label{epaI}
\int\abs{\psi(\eta)}\abs{\hat v(z,\eta)}\abs{g_m(z)}e^{-\Phi_0(z)}d\eta dv(z)<\infty,\ \ \forall m>0.
\end{equation} 
From \eqref{s7-eXI}, we see that $\hat v(z,\eta)=0$ almost everywhere on $\Real\setminus\Real_{p,1}$, for every $z\in\Complex^{n-1}$. Since ${\rm Supp\,}\psi\bigcap\Real_{p,1}=\emptyset$(see the discussion after \eqref{sp3-eI}), we conclude that for each $m>0$, 
\begin{equation}\label{epaII}
z\To\int\psi(\eta)\hat v(z,\eta)g_m(z)e^{-\Phi_0(z)}d\eta=0.
\end{equation} 
From \eqref{epaI}, \eqref{epaII} and Fubini's Theorem, we obtain 
\begin{equation}\label{epaIII}
\int\hat v(z, \eta)g_m(z)\psi(\eta)e^{-\Phi_0(z)}d\eta dv(z)=0
\end{equation}
for every $m>0$. From \eqref{epaIII} and \eqref{GPCeV}, we get for each $m$, 
\begin{equation}\label{GPCeVI}
\begin{split}
&\int\hat\psi(t)\chi_M(t)e^{-\frac{\beta}{2}t}g_m(z)u(z,t)e^{-\Phi_0(z)}dv(z)dt\\
&=\frac{1}{2\pi}\int\hat v(z,\eta)g_m(z)\alpha_1(\eta)e^{-\Phi_0(z)}dv(z)d\eta.
\end{split}
\end{equation}
Since $0\leq\chi_M\leq 1$ and $\chi_M=1$ if $-M\leq t\leq M$, we have 
\begin{equation}\label{GPCeVII}
\abs{\alpha_1(\eta)}=\abs{\int_{\abs{t}\geq M}e^{-it\eta}\hat\psi(t)(\chi_M(t)-1)dt}\leq\int_{\abs{t}\geq M}\abs{\hat\psi(t)}dt\leq\frac{2C_0}{M},\ \ \forall\eta\in\Real,
\end{equation}
where $C_0=\sup_{t\in\Real}t^2\abs{\hat\psi(t)}$. Put 
\[f(\eta):=\int_{\Complex^{n-1}}\abs{\hat v(z,\eta)}^2e^{-\Phi_0(z)}dv(z).\]
From \eqref{GPCeVII} and \eqref{GPCeVI}, we have for each $m$, 
\begin{equation}\label{GPCeVIII}
\begin{split}
&\abs{\int\hat\psi(\eta)\chi_M(t)e^{-\frac{\beta}{2}t}g_m(z)u(z,t)e^{-\Phi_0(z)}dv(z)dt}\\
&\leq\frac{2C_0}{M}\frac{1}{2\pi}\int\abs{\hat v(z,\eta)}g_m(z)e^{-\Phi_0(z)}dv(z)d\eta=\frac{2C_0}{M}\frac{1}{2\pi}\int_{\abs{z}\leq1}\abs{\hat v(\frac{z}{m},\eta)}\varphi(z)dv(z)d\eta\\
&\stackrel{by \eqref{s7-eXI}}\leq\frac{2C_0}{M}(2\pi)^{-\frac{n+1}{2}}\int_{\abs{z}\leq1}e^{\Phi(\frac{z}{m})}\mathds{1}_{\Real_{p,1}}(\eta)\abs{\det(M^{\phi}_p+2\eta\mathcal{L}_p)}^{\frac{1}{2}}\sqrt{f(\eta)}\varphi(z)dv(z)d\eta\\
&\leq\frac{2C_0}{M}(2\pi)^{-\frac{n+1}{2}}\sup\{e^{\Phi_0(\frac{z}{m})};\, \abs{z}\leq1\}\Bigr(\int_{\Real_{p,1}}\abs{\det(M^{\phi}_p+2\eta\mathcal{L}_p)}d\eta\Bigr)^{\frac{1}{2}}\Bigr(\int f(\eta)d\eta\Bigr)^{\frac{1}{2}}\\
&\stackrel{by \eqref{alpc}}=\frac{2C_0}{M}(2\pi)^{-\frac{n}{2}}\sup\{e^{\Phi_0(\frac{z}{m})};\, \abs{z}\leq1\}\Bigr(\int_{\Real_{p,1}}\abs{\det(M^{\phi}_p+2\eta\mathcal{L}_p)}d\eta\Bigr)^{\frac{1}{2}}.
\end{split}
\end{equation}
Combining \eqref{GPCeVIII} with \eqref{s7-ea2} and \eqref{sp3-eI2}, we get 
\[\abs{(Q^{(1)}_{M,H_n}u)(0)}\leq\frac{2C_0}{M}\Bigr((2\pi)^{-n}\int_{\Real_{p,1}}\abs{\det(M^{\phi}_p+2\eta\mathcal{L}_p)}d\eta\Bigr)^{\frac{1}{2}}<\frac{E}{\sqrt{n-1}},\]
where $E$ is as in \eqref{change}. \eqref{s7-ea1} follows.
\end{proof}

In view of Proposition~\ref{s7-p1}, we have proved that for all $u(z, \theta)\in\Omega^{0,1}(H_n)$ with $\norm{u}_{\phi_0}=1$, 
$\Box^{(1)}_{b,H_n}u=0$, we have 
\[\abs{(Q^{(1)}_{M,H_n}u)(0)_j}^2\leq\abs{(Q^{(1)}_{M,H_n}u)(0)}^2<
\frac{E^2}{n-1},\]
for all $j=1,\ldots,n-1$, where $(Q^{(1)}_{M,H_n}u)(0)=\sum^{n-1}_{j=1}(Q^{(1)}_{M,H_n}u)_j(0)d\ol z_j$ and $E$ is as in \eqref{change}.
Thus, for every $j=1,\ldots,n-1$, we have 
\[(Q^{(1)}_{M,H_n}S^{(1)}_{j,H_n}\ol{Q^{(1)}_{M,H_n}})(0)
<\frac{E^2}{n-1}\]
and
\begin{equation}\label{s7-eb4}
\sum^{n-1}_{j=1}(Q^{(1)}_{M,H_n}S^{(1)}_{j,H_n}\ol{Q^{(1)}_{M,H_n}})(0)
<E^2.
\end{equation} 
From \eqref{s7-eb4} and Theorem~\ref{s7-t1}, we obtain the main result of this section

\begin{thm}\label{s7-t4}
Let $\nu_k>0$ be any sequence with $\nu_k\To0$ as $k\To\infty$. For each $x\in X$, we have
\begin{equation}\label{s7-eb5}
\limsup_{k\To\infty}k^{-n}(Q^{(1)}_{M,k}\pit^{(1)}_{k,\leq k\nu_k}\ol{Q^{(1)}_{M,k}})(x)<
E^2,
\end{equation}
where $E$ is as in \eqref{change}.
\end{thm}

The proof of the following theorem is essentially the same as the proof of \eqref{s4-e0}. We omit the proof.

\begin{thm}\label{s7-t5}
For any sequence $\nu_k>0$ with $\nu_k\To0$ as $k\To\infty$, there is a constant $C>0$ independent of $k$,
such that
\begin{equation} \label{s7-eb6}
\abs{(Q^{(1)}_{M,k}\pit^{(1)}_{k,\,\leqslant  k\nu_k}\ol{Q^{(1)}_{M,k}})(x)}\leqslant  Ck^n,\ \ \forall x\in X.
\end{equation} 
\end{thm} 

Now, we can prove

\begin{thm}\label{s7-t6}
Let $\nu_k>0$ be any sequence with $\nu_k\To0$ as $k\To\infty$. 
Then there is a $k_0>0$ such that for all $k\geq k_0$, 
\begin{equation}\label{s7-eb7}
\int_X(Q^{(1)}_{M,k}\pit^{(1)}_{k,\leq k\nu_k}\ol{Q^{(1)}_{M,k}})(x)dv_X(x)\leq k^n\int_XE^2dv_X(x),
\end{equation}
where $E$ is as in \eqref{change}.
\end{thm} 

\begin{proof}
In view of Theorem~\ref{s7-t5}, $\sup_k k^{-n}Q^{(1)}_{M,k}\pit^{(1)}_{k,\leq k\nu_k}\ol{Q^{(1)}_{M,k}})(\cdot)$ is integrable on $X$.
Thus, we can apply Fatou's lemma and we get by using Theorem~\ref{s7-t4}:
\begin{equation*}
\begin{split}
&\limsup_{k\To\infty}k^{-n}\int_X(Q^{(1)}_{M,k}\pit^{(1)}_{k,\leq k\nu_k}\ol{Q^{(1)}_{M,k}})(x)dv_X(x)\\
&\quad\leq\int_X\limsup_{k\To\infty}k^{-n}(Q^{(1)}_{M,k}\pit^{(1)}_{k,\leq k\nu_k}\ol{Q^{(1)}_{M,k}})(x)dv_X(x)\\
&\quad<\int_XE^2dv_X(x).
\end{split}
\end{equation*}
The theorem follows.
\end{proof} 

\section{The proof of Theorem~\ref{main}}

Let $\delta_k>0$, $\delta_k\To\infty$ as $k\To\infty$, be as in Theorem~\ref{s6-t1} and let $\nu_k>0$ be any sequence with
$\lim_{k\To\infty}\frac{\delta_k}{\nu_k}=0$ and $\nu_k\To0$ as $k\To\infty$. Let 
$\gamma_{1,k}<\gamma_{2,k}<\cdots<\gamma_{m_k,k}$
be all the distinct non-zero eigenvalues of $\Box^{(0)}_{b,k}$ between $0$ and $k\nu_k$. Thus, $\gamma_{1,k}>0$ and $\gamma_{m_k,k}\leq k\nu_k$. We notice that $\gamma_{j,k}$, $j=1,\ldots,m_k$, are also eigenvalues of $\Box^{(1)}_{b,k}$. 
For $\mu\in\Real$, let $\cH^q_{b,\mu}(X,L^k)$ denote the space spanned by the eigenforms of $\Box^{(q)}_{b,k}$ whose eigenvalues are $\lambda$. For each $j\in\set{1,\ldots,m_k}$, let
$f^1_{j,k},f^2_{j,k},\ldots,f^{d_{j,k}}_{j,k}$
be an orthonormal basis for $\cH^0_{b,\gamma_{j,k}}(X,L^k)$, where $d_{j,k}={\rm dim\,}\cH^0_{b,\gamma_{j,k}}(X,L^k)$.
Let $f^1_{0,k},f^2_{0,k},\ldots,f^{d_{0,k}}_{0,k}$
be an orthonormal basis for $\cH^0_b(X,L^k)$, where $d_{0,k}={\rm dim\,}\cH^0_b(X,L^k)$.

Let $Q^{(0)}_{M,k}$ and $Q^{(1)}_{M,k}$ be as in \eqref{sp3-eII} and \eqref{sp3-eV} respectively. 
By the definition of $(Q^{(0)}_{M,k}\pit^{(0)}_{k,\leq k\nu_k})(x)$(see \eqref{e-seII}), we have
\begin{equation}\label{s8-e1}
\begin{split}
&(Q^{(0)}_{M,k}\pit^{(0)}_{k,\leq k\nu_k})(x)\\
&\quad=\sum^{d_{0,k}}_{t=1}\langle (Q^{(0)}_{M,k}f^t_{0,k})(x)|f^t_{0,k}(x)\rangle_{h^{L^k}}+\sum^{m_k}_{j=1}\sum^{d_{j,k}}_{t=1}\langle(Q^{(0)}_{M,k}f^t_{j,k})(x)|f^t_{j,k}(x)\rangle_{h^{L^k}}.
\end{split}
\end{equation}
From \eqref{s8-e1} and \eqref{s6-e1.1}, we conclude that
\begin{equation}\label{s8-e2}
\begin{split}
&\sum^{d_{0,k}}_{t=1}\abs{(Q^{(0)}_{M,k}f^t_{0,k}\ |\ f^t_{0,k})_{h^{L^k}}}+\sum^{m_k}_{j=1}\sum^{d_{j,k}}_{t=1}\abs{(Q^{(0)}_{M,k}f^t_{j,k}\ |\ f^t_{j,k})_{h^{L^k}}}\\
&\quad\geq\frac{k^n}{2}(2\pi)^{1-n}\int_X\Bigr(\int\psi(\xi)\det(M^\phi_x+2\xi\mathcal{L}_x)
\mathds{1}_{\Real_{x,0}}(\xi)d\xi\Bigr)dv_X(x),
\end{split}
\end{equation}
for $k$ large.
For $j=1,\ldots,m_k$, we put 
\[g^t_{j,k}=\frac{1}{\norm{\ddbar_{b,k}f^t_{j,k}}_{h^{L^k}}}\ddbar_{b,k}f^t_{j,k}=\frac{1}{\sqrt{\gamma_{j,k}}}\ddbar_{b,k}f^t_{j,k}
\in\cH^1_{b,\gamma_{j,k}}(X,L^k),\ \ t=1,\ldots,d_{j,k}.\]
For each $j=1,\ldots,m_k$, 
\begin{equation}\label{s8-e2-3}
\begin{split}
(Q^{(1)}_{M,k}g^t_{j,k}\ |\ g^t_{j,k})_{h^{L^k}}&=\frac{1}{\gamma_{j,k}}(Q^{(1)}_{M,k}\ddbar_{b,k}f^t_{j,k}\ |\ \ddbar_{b,k}f^t_{j,k})_{h^{L^k}}\\
&=\frac{1}{\gamma_{j,k}}(\ddbar_{b,k}Q^{(0)}_{M,k}f^t_{j,k}\ |\ \ddbar_{b,k}f^t_{j,k})_{h^{L^k}}\ \ \mbox{here we used \eqref{sp3-eX}}\\
&=\frac{1}{\gamma_{j,k}}(Q^{(0)}_{M,k}f^t_{j,k}\ |\ \Box^{(0)}_{b,k}f^t_{j,k})_{h^{L^k}}\\
&=(Q^{(0)}_{M,k}f^t_{j,k}\ |\ f^t_{j,k})_{h^{L^k}},\ \ t=1,\ldots,d_{j,k}.
\end{split}
\end{equation}
Hence, 
\begin{equation}\label{s8-e2a}
\begin{split}
&\sum^{m_k}_{j=1}\sum^{d_{j,k}}_{t=1}\abs{(Q^{(0)}_{M,k}f^t_{j,k}\ |\ f^t_{j,k})_{h^{L^k}}}\\
&=\sum^{m_k}_{j=1}\sum^{d_{j,k}}_{t=1}\abs{(Q^{(1)}_{M,k}g^t_{j,k}\ |\ g^t_{j,k})_{h^{L^k}}}\\
&\quad\leq\sqrt{\sum^{m_k}_{j=1}\sum^{d_{j,k}}_{t=1}\norm{Q^{(1)}_{M,k}g^t_{j,k}}^2_{h^{L^k}}}
\sqrt{\sum^{m_k}_{j=1}\sum^{d_{j,k}}_{t=1}\norm{g^t_{j,k}}^2_{h^{L^k}}}.
\end{split}
\end{equation}
Since $\norm{g^t_{j,k}}^2_{h^{L^k}}=\norm{f^t_{j,k}}^2_{h^{L^k}}=1$, for every $j$ and $t$, it is obviously that 
\[\sqrt{\sum^{m_k}_{j=1}\sum^{d_{j,k}}_{t=1}\norm{g^t_{j,k}}^2_{h^{L^k}}}
=\sqrt{\sum^{m_k}_{j=1}\sum^{d_{j,k}}_{t=1}\norm{f^t_{j,k}}^2_{h^{L^k}}}.\]
Combining this with \eqref{s8-e2a} and \eqref{s8-e2}, we get
\begin{equation}\label{s8-e2-1}
\begin{split}
&\sum^{d_{0,k}}_{t=1}\abs{(Q^{(0)}_{M,k}f^t_{0,k}\ |\ f^t_{0,k})_{h^{L^k}}}+\sqrt{\sum^{m_k}_{j=1}\sum^{d_{j,k}}_{t=1}\norm{Q^{(1)}_{M,k}g^t_{j,k}}^2_{h^{L^k}}}
\sqrt{\sum^{m_k}_{j=1}\sum^{d_{j,k}}_{t=1}\norm{f^t_{j,k}}^2_{h^{L^k}}}\\
&\quad\geq\frac{k^n}{2}(2\pi)^{1-n}\int_X\Bigr(\int\psi(\xi)\det(M^\phi_x+2\xi\mathcal{L}_x)
\mathds{1}_{\Real_{x,0}}(\xi)d\xi\Bigr)dv_X(x),
\end{split}
\end{equation}
for $k$ large. 

We can check that for each $j=1,\ldots,m_k$, $g^t_{j,k}$, $t=1,\ldots,d_{j,k}$ is an orthonormal basis of the space 
$\ddbar_{b,k}\cH^0_{b,\gamma_{j,k}}(X,L^k)\subset\cH^1_{b,\gamma_{j,k}}(X,L^k)$.
From this observation and the definition of 
$(Q^{(1)}_{M,k}\pit^{(1)}_{k,\leq k\nu_k}\ol{Q^{(1)}_{M,k}})(x)$(see \eqref{e-seIII}), we conclude that 
\begin{equation}\label{s8-e3}
\sum^{m_k}_{j=1}\sum^{d_{j,k}}_{t=1}\abs{Q^{(1)}_{M,k}g^t_{j,k}}^2_{h^{L^k}}(x)\leq(Q^{(1)}_{M,k}\pit^{(1)}_{k,\leq k\nu_k}\ol{Q^{(1)}_{M,k}})(x).
\end{equation}
Thus, 
\begin{equation}\label{s8-e4}
\sum^{m_k}_{j=1}\sum^{d_{j,k}}_{t=1}\norm{Q^{(1)}_{M,k}g^t_{j,k}}^2_{h^{L^k}}\leq\int_X(Q^{(1)}_{M,k}\pit^{(1)}_{k,\leq k\nu_k}\ol{Q^{(1)}_{M,k}})(x)dv(x).
\end{equation}
Combining \eqref{s8-e4} with \eqref{s7-eb7}, we get
\begin{equation}\label{s8-e5}
\sum^{m_k}_{j=1}\sum^{d_{j,k}}_{t=1}\norm{Q^{(1)}_{M,k}g^t_{j,k}}^2_{h^{L^k}}\\
\leq k^n\int_XE^2dv_X(x)
\end{equation}
for $k$ large, where $E$ is as in \eqref{change}.

From \eqref{s2-e6I} and \eqref{s2-e6}, we conclude that
\begin{equation}\label{s8-e6}
\int_X\pit^{(0)}_{k,\,\leqslant  k\nu_k}(x)dv_X(x)=k^n(2\pi)^{-n}\int_X\Bigr(\int_{\Real_{x,0}}\abs{\det(M^\phi_x+2s\mathcal{L}_x)}ds\Bigr)dv_X(x)+o(k^n),
\end{equation}
for $k$ large. It is obviously the case that 
\[\sum^{m_k}_{j=1}\sum^{d_{j,k}}_{t=1}\norm{f^t_{j,k}}^2_{h^{L^k}}\leq\int_X\pit^{(0)}_{k,\,\leqslant  k\nu_k}(x)dv_X(x).\]
Combining this with \eqref{s8-e6}, we get
\begin{equation}\label{s8-e7}
\begin{split}
&\sum^{m_k}_{j=1}\sum^{d_{j,k}}_{t=1}\norm{f^t_{j,k}}^2_{h^{L^k}}\\
&\quad\leq k^n2(2\pi)^{-n}\int_X\Bigr(\int\det(M^\phi_x+2\xi\mathcal{L}_x)
\mathds{1}_{\Real_{x,0}}(\xi)d\xi\Bigr)dv_X(x)
\end{split}
\end{equation}
for $k$ large. From \eqref{s8-e5}, \eqref{s8-e7}, \eqref{s8-e2-1} and \eqref{change}, we obtain

\begin{thm}\label{s8-t1}
Let $f^1_{0,k},f^2_{0,k},\ldots,f^{d_{0,k}}_{0,k}$
be an orthonormal basis for $\cH^0_b(X,L^k)$, where $d_{0,k}={\rm dim\,}\cH^0_b(X,L^k)$. Then, for $k$ large, we have
\begin{equation}\label{s8-e8}
\begin{split}
&\sum^{d_{0,k}}_{t=1}\abs{(Q^{(0)}_{M,k}f^t_{0,k}\ |\ f^t_{0,k})_{h^{L^k}}}\\
&\quad\geq\frac{k^n}{4}(2\pi)^{1-n}\int_X\Bigr(\int\psi(\xi)\det(M^\phi_x+2\xi\mathcal{L}_x)
\mathds{1}_{\Real_{x,0}}(\xi)d\xi\Bigr)dv_X(x).
\end{split}
\end{equation}
\end{thm}

The following is straightforward 

\begin{lem}\label{s8-l1}
For $k$ large, there is a constant $C>0$ independent of $k$, such that 
\[\norm{Q^{(0)}_{M,k}u}^2_{h^{L^k}}\leq C\norm{u}^2_{h^{L^k}},\ \ \forall u\in C^\infty(X,L^k).\]
\end{lem} 

\begin{proof}
Let $D\Subset D'\Subset D''\Subset X$ be open sets of $X$ and let $s$ be a local section of $L$ on $D''$. We assume that there exist 
canonical coordinates $x=(x_1,\ldots,x_{2n-1})=(z,\theta)$ on $D''$.  
Let $\chi_M$ be as in \eqref{sp3-eM0}. 
For $k$ large, we have 
\[\set{\Phi^{\frac{t}{k}}(x)\in D';\, \forall x\in D, t\in{\rm Supp\,}\chi_M}\] 
and ${\rm Supp\,}f(\Phi^{\frac{t}{k}}x)\subset D'$, $\forall t\in{\rm Supp\,}\chi_M$, $\forall f\in C^\infty_0(D,L^k)$. In cannonical coordinates $x=(z,\theta)$, we have 
$\Phi^{\frac{t}{k}}(x)=(z,\frac{t}{k}+\theta)$. Let $m(z,\theta)dv(z)d\theta$ be the volume form on $D''$, where $dv(z)=2^{n-1}dx_1dx_2\ldots dx_{2n-2}$. Since $m(z,\theta)$ is strictly positive, for $k$ large, there is a constant $C_1>0$ independent of $k$, such that 
\begin{equation}\label{efin1}
m(z,\theta)\leq C_1m(z,\theta+\frac{t}{k}),\ \ \forall (z,\theta)\in D',\ \ t\in{\rm Supp\,}\chi_M.
\end{equation} 
Let $u\in C^\infty_0(D,L^k)$. On $D''$, we write $u=s^k\Td u$, $\Td u\in C^\infty_0(D)$.
From the definition of $Q^{(0)}_{M,k}$ (see \eqref{sp3-eII}), we 
can check that for $k$ large,
\begin{equation}\label{alI}
\begin{split}
&\int\abs{(Q^{(0)}_{M,k}u)(x)}^2_{h^{L^k}}dv_X(x)\\
&=\int_{D'}\abs{\int e^{-it\eta}\psi(\eta)\chi_M(t)e^{-\frac{k}{2}\phi(z,\theta+\frac{t}{k
})}\Td u(z,\theta+\frac{t}{k})dtd\eta}^2m(z,\theta)dv(z)d\theta\\
&\leq \Td C\int_{(z,\theta)\in D'}\chi_M(t)e^{-k\phi(z,\theta+\frac{t}{k})}\abs{\Td u(z,\theta+\frac{t}{k})}^2m(z,\theta)dtd\theta dv(z)\\
&\leq\Td CC_1\int_{(z,\theta)\in D'}\chi_M(t)e^{-k\phi(z,\theta+\frac{t}{k})}\abs{\Td u(z,\theta+\frac{t}{k})}^2m(z,\theta+\frac{t}{k})dtd\theta dv(z)\\
&=\Td CC_1\int_{(z,\lambda-\frac{t}{k})\in D'}\chi_M(t)e^{-k\phi(z,\lambda)}\abs{\Td u(z,\lambda)}^2m(z,\lambda)dtd\lambda dv(z)\\
&\leq C\int e^{-k\phi(z,\theta)}\abs{\Td u(z,\theta)}^2m(z,\theta)dv(z)d\theta=C\norm{u}^2_{h^{L^k}}.
\end{split}
\end{equation}
where $\Td C>0$, $C>0$ are independent of $k$ and $u$ and $C_1$ is as in \eqref{efin1}. From \eqref{alI} and by using partition of unity, 
the lemma follows.
\end{proof}

\begin{proof}[Proof of Theorem~\ref{main}]
From Lemma~\ref{s8-l1} and \eqref{s8-e8}, we see that for $k$ large,
\[\begin{split}
\sqrt{C}d_{0,k}&=\sqrt{C}\sum^{d_{0,k}}_{t=1}\norm{f^t_{0,k}}^2_{h^{L^k}}
\geq\sum^{d_{0,k}}_{t=1}\abs{(Q^{(0)}_{M,k}f^t_{0,k}\ |\ f^t_{0,k})_{h^{L^k}}}\\
&\geq\frac{k^n}{4}(2\pi)^{1-n}\int_X\Bigr(\int\psi(\xi)\det(M^\phi_x+2\xi\mathcal{L}_x)
\mathds{1}_{\Real_{x,0}}(\xi)d\xi\Bigr)dv_X(x),\end{split}\]
where $C>0$ is the constant as in Lemma~\ref{s8-l1} and $d_{0,k}={\rm dim\,}\cH^0_b(X,L^k)$. Theorem~\ref{main} follows.
\end{proof} 

\section{Examples} 

In this section, some examples are collected. The aim is to illustrate the main results in some 
simple situations. 

\subsection{Compact Heisenberg groups} 
Let $\lambda_1,\ldots,\lambda_{n-1}$ be given non-zero integers.
Let $\mathscr CH_n=(\Complex^{n-1}\times\Real)/_\sim$\,, where
$(z, \theta)\sim(\Td z, \Td\theta)$ if
\[\begin{split}
&\Td z-z=(\alpha_1,\ldots,\alpha_{n-1})\in\sqrt{2\pi}\mathbb Z^{n-1}+i\sqrt{2\pi}\mathbb Z^{n-1},\\
&\Td\theta-\theta-i\sum^{n-1}_{j=1}\lambda_j(z_j\ol\alpha_j-\ol z_j\alpha_j)\in\pi\mathbb Z.
\end{split}\]
We can check that $\sim$ is an equivalence relation
and $\mathscr CH_n$ is a compact manifold of dimension $2n-1$. The equivalence class of $(z, \theta)\in\Complex^{n-1}\times\Real$ is denoted by
$[(z, \theta)]$.
For a given point $p=[(z, \theta)]$, we define
$T^{1, 0}_p\mathscr CH_n$ to be the space spanned by
\[
\textstyle
\big\{\frac{\pr}{\pr z_j}+i\lambda_j\ol z_j\frac{\pr}{\pr\theta},\ \ j=1,\ldots,n-1\big\}.
\]
It is easy to see that the definition above is independent of the choice of a representative $(z, \theta)$ for $[(z, \theta)]$.
Moreover, we can check that $T^{1, 0}\mathscr CH_n$ is a CR structure and $T:=\frac{\pr}{\pr\theta}$ is a rigid global real vector field. Thus, $(\mathscr CH_n, T^{1, 0}\mathscr CH_n)$ is a compact generalized Sasakian CR manifold of dimension $2n-1$.
Let $J$ denote the canonical complex structure on $X\times\Real$ given by $J\frac{\pr}{\pr t}=T$, where $t$ denotes the coordinate of $\Real$. We take a Hermitian metric $\langle\,\cdot\,|\,\cdot\,\rangle$ on the complexified tangent bundle $\Complex T\mathscr CH_n$ such that
\[
\Big\lbrace
\tfrac{\pr}{\pr z_j}+i\lambda_j\ol z_j\tfrac{\pr}{\pr\theta}\,, \tfrac{\pr}{\pr\ol z_j}-i\lambda_jz_j\tfrac{\pr}{\pr\theta}\,, -\tfrac{\pr}{\pr\theta}\,;\, j=1,\ldots,n-1\Big\rbrace
\]
 is an orthonormal basis. The dual basis of the complexified cotangent bundle is
\[
\Big\lbrace
dz_j\,,\, d\ol z_j\,,\, \omega_0:=-d\theta+\textstyle\sum^{n-1}_{j=1}(i\lambda_j\ol z_jdz_j-i\lambda_jz_jd\ol z_j); j=1,\ldots,n-1
\Big\rbrace\,.
\]
The Levi form $\mathcal{L}_p$ of $\mathscr CH_n$ at $p\in\mathscr CH_n$ is given by
$\mathcal{L}_p=\sum^{n-1}_{j=1}\lambda_jdz_j\wedge d\ol z_j$.
 
Now, we construct a generalized Sasakian CR line bundle $(L,J)$ over $\mathscr CH_n$. Let $L=(\Complex^{n-1}\times\Real\times\Complex)/_\equiv$ where $(z,\theta,\eta)\equiv(\Td z, \Td\theta, \Td\eta)$ if
\[\begin{split}
&(z,\theta)\sim(\Td z, \Td\theta),\\
&\Td\eta\exp(\Td\theta+i\sum^{n-1}_{j=1}\lambda_j\abs{\Td z_j}^2)=\eta\exp(\theta+i\sum^{n-1}_{j=1}\lambda_j\abs{ z_j}^2)\exp(\sum^{n-1}_{j,t=1}\mu_{j,t}(z_j\ol\alpha_t+\frac{1}{2}\alpha_j\ol\alpha_t)),
\end{split}\]
where $\alpha=(\alpha_1,\ldots,\alpha_{n-1})=\Td z-z$, $\mu_{j,t}=\mu_{t,j}$, $j, t=1,\ldots,n-1$, are given integers. We can check that $\equiv$ is an equivalence relation and
$(L,J)$ is a generalized Sasakian CR line bundle over $\mathscr CH_n$. For $(z, \theta, \eta)\in\Complex^{n-1}\times\Real\times\Complex$ we denote
$[(z, \theta, \eta)]$ its equivalence class.
It is straightforward to see that the pointwise norm
\[
\big\lvert[(z, \theta, \eta)]\big\rvert^2_{h^L}:=\abs{\eta}^2\exp\big(2\theta-\textstyle\sum^{n-1}_{j,t=1}\mu_{j,t}z_j\ol z_t\big)
\]
is well-defined. In local coordinates $(z, \theta, \eta)$, the weight function of this metric is
\[\phi=-2\theta+\sum^{n-1}_{j,t=1}\mu_{j,t}z_j\ol z_t.\] 
We can check that $T\phi=-2$. Thus, $(L,J,h^L)$ is a rigid generalized Sasakian CR line bundle over $\mathscr CH_n$.
Note that
\[
\textstyle\ddbar_b=\sum^{n-1}_{j=1}d\ol z_j\wedge(\frac{\pr}{\pr\ol z_j}-i\lambda_jz_j\frac{\pr}{\pr\theta})\,,\quad
\pr_b=\sum^{n-1}_{j=1}dz_j\wedge(\frac{\pr}{\pr z_j}+i\lambda_j\ol z_j\frac{\pr}{\pr\theta}).
\]
Thus
$d(\ddbar_b\phi-\pr_b\phi)=2\sum^{n-1}_{j,t=1}\mu_{j,t}dz_j\wedge d\ol z_t$ and for any $p\in\mathscr CH_n$, 
\[M^\phi_p=\sum^{n-1}_{j,t=1}\mu_{j,t}dz_j\wedge d\ol z_t.\]
From this and Theorem~\ref{main}, we obtain

\begin{thm} \label{t-ex1}
If the matrix $\left(\mu_{j,t}\right)^{n-1}_{j,t=1}$ is positive definite and $Y(0)$, $Y(1)$ hold on $\mathscr CH_n$, then 
for $k$ large, there is a constant $c>0$ independent of $k$, such that 
\[{\rm dim\,}H^0_b(\mathscr CH_n, L^k)\geq ck^n.\]
\end{thm}

\subsection{Holomorphic line bundles over a complex torus}

Let
\[T_n:=\Complex^n/(\sqrt{2\pi}\mathbb Z^n+i\sqrt{2\pi}\mathbb Z^n)\]
be the flat torus. Let $\lambda=\left(\lambda_{j,t}\right)^{n}_{j,t=1}$, where $\lambda_{j,t}=\lambda_{t,j}$, 
$j, t=1,\ldots,n$, are given integers. Let $L_\lambda$ be the holomorphic
line bundle over $T_n$
with curvature the $(1,1)$-form
$\Theta_\lambda=\sum^n_{j,t=1}\lambda_{j,t}dz_j\wedge d\ol z_t$.
More precisely, $L_\lambda:=(\Complex^n\times\Complex)/_\sim$\,, where
$(z, \theta)\sim(\Td z, \Td\theta)$ if
\[
\Td z-z=(\alpha_1,\ldots,\alpha_n)\in \sqrt{2\pi}\mathbb Z^n+i\sqrt{2\pi}\mathbb Z^n\,,\quad
\Td\theta=\textstyle\exp\big(\sum^n_{j,t=1}\lambda_{j,t}(z_j\ol\alpha_t+\tfrac{1}{2}\alpha_j\ol\alpha_t\,)\big)\theta\,.
\]
We can check that $\sim$ is an equivalence relation and $L_\lambda$ is a holomorphic line bundle over $T_n$.
For $[(z, \theta)]\in L_\lambda$
we define the Hermitian metric by
\[
\big\vert[(z, \theta)]\big\vert^2:=\abs{\theta}^2\textstyle\exp(-\sum^n_{j,t=1}\lambda_{j,t}z_j\ol z_t)
\]
and it is easy to see that this definition is independent of the choice of a representative $(z, \theta)$ of $[(z, \theta)]$. We denote by $\phi_\lambda(z)$ the weight of this Hermitian fiber metric. Note that $\pr\ddbar\phi_\lambda=\Theta_\lambda$.

Let $L^*_\lambda$ be the
dual bundle of $L_\lambda$ and let $\norm{\,\cdot\,}_{L^*_\lambda}$ be the norm of $L^*_\lambda$ induced by the Hermitian fiber metric on $L_\lambda$. Consider the compact CR manifold of dimension $2n+1$: $X=\{v\in L^*_\lambda;\, \norm{v}_{L^*_\lambda}=1\}$; this is the boundary of the Grauert tube associated to $L^*_\lambda$. The manifold $X$ is equipped with a natural $S^1$-action.
Locally $X$ can be represented in local holomorphic coordinates $(z,\eta)$, where $\eta$ is the fiber coordinate, as the set of all $(z,\eta)$ such that $\abs{\eta}^2e^{\phi_\lambda(z)}=1$.
The $S^1$-action on $X$ is given by $e^{i\theta}\circ (z,\eta)=(z,e^{i\theta}\eta)$, $e^{i\theta}\in S^1$, $(z,\eta)\in X$. Let $T$ be the global real vector field on $X$ determined by $Tu(x)=\frac{\pr}{\pr\theta}u(e^{i\theta}\circ x)\big|_{\theta=0}$, for all $u\in C^\infty(X)$. 
We can check that $T$ is a rigid global real vector field on $X$. Thus, $X$ is a compact generalized Sasakian CR manifold of dimension $2n+1$. Let $J$ denote the canonical complex structure on $X\times\Real$ given by $J\frac{\pr}{\pr t}=T$, where $t$ denotes the coordinate of $\Real$.

Let $\pi:L^*_\lambda\To T_n$
be the natural projection from $L^*_\lambda$ onto $T_n$. Let $\mu=\left(\mu_{j,t}\right)^{n}_{j,t=1}$, where $\mu_{j,t}=\mu_{t,j}$, $j, t=1,\ldots,n$, are given integers. Let $L_\mu$ be another holomorphic
line bundle over $T_n$ determined by the constant curvature form
$\Theta_\mu=\sum^n_{j,t=1}\mu_{j,t}dz_j\wedge d\ol z_t$ as above. 
The pullback line bundle $\pi^*L_\mu$ is a holomorphic line bundle over $L^*_\lambda$. If we restict $\pi^*L_\mu$ on $X$, then we can check that $(\pi^*L_\mu,J)$ is a generalized Sasakian CR line bundle over $X$. 

The Hermitian fiber metric on $L_\mu$ induced by $\phi_\mu$ induces a Hermitian fiber metric on $\pi^*L_\mu$ that we
shall denote by $h^{\pi^*L_\mu}$. We let $\psi$ to denote the weight of $h^{\pi^*L_\mu}$. 
The part of $X$ that lies over a fundamental domain of $T_n$ can be represented in local holomorphic coordinates
$(z, \xi)$, where $\xi$ is the fiber coordinate, as the set of all $(z, \xi)$ such that
$r(z, \xi):=\abs{\xi}^2\exp(\sum^n_{j,t=1}\lambda_{j,t}z_j\ol z_t)-1=0$
and the weight $\psi$ may be written as $\psi(z, \xi)=\sum^n_{j,t=1}\mu_{j,t}z_j\ol z_t$. From this we see that $(\pi^*L_\mu,J,h^{\pi^*L_\mu})$ is a rigid generalized Sasakian CR line bundle over $X$. It is straightforward to check that for any $p\in X$, we have 
$M^\psi_p=\frac{1}{2}d(\ddbar_b\psi-\pr_b\psi)(p)|_{T^{1, 0}X}=\sum^n_{j,t=1}\mu_{j,t}dz_j\wedge d\ol z_t$.
From this observation and Theorem~\ref{main}, we obtain

\begin{thm} \label{s7-t3a}
If the matrix $\left(\mu_{j,t}\right)^{n-1}_{j,t=1}$ is positive definite and $Y(0)$, $Y(1)$ hold on $X$, then 
for $k$ large, there is a constant $c>0$ independent of $k$, such that 
\[\dim H^0_b(X, (\pi^*L_\mu)^k)\geq ck^{n+1}.\]
\end{thm}


\begin{thebibliography}{10} 

\bibitem{AndS70}
A.~Andreotti and Y.-T, Siu,
\emph{Projective embedding of pseudoconcave spaces},
Ann. Scuola Norm. Sup. Pisa(3), \textbf{24} (1970), 231--278.

\bibitem{BRT85}
M.-S.~Baouendi and L.-P.~Rothschild and F.-Treves,
\emph{C{R} structures with group action and extendability of {C}{R} functions},
Invent. Math., \textbf{83} (1985), 359--396.

\bibitem{Be05}
R.~Berman, 
\emph{Holomorphic {M}orse inequalities on manifolds with boundary},
  Ann. Inst. Fourier (Grenoble) \textbf{55} (2005), no.~4, 1055--1103.


\bibitem{BdM1:74b}
L.~{B}outet~de {M}onvel, \emph{Int\'egration des \'equations de
  {C}auchy-{R}iemann induites formelles}, S\'eminaire Goulaouic-Lions-Schwartz
  1974--1975; \'Equations aux deriv\'ees partielles lin\'eaires et non
  lin\'eaires, Centre Math., \'Ecole Polytech., Paris, 1975, Exp. no. 9,
  pp.~13.



\bibitem{CS01}
S-C.~Chen and M-C.~Shaw,
\emph{Partial differential equations in several complex variables},
AMS/IP Studies in Advanced Mathematics. 19. Providence, RI: American Mathematical Society (AMS). Somerville, MA: International Press, xii, 380 p., (2001).  

\bibitem{Dav95}
E.-B.~Davies, \emph{{Spectral} {Theory} and {Differential} {Operators}}, 
Cambridge Stud. Adv. Math., vol. \textbf{42}, (1995). 

\bibitem{De:85}
J.-P.~Demailly, \emph{{Champs magn\'etiques et inegalit\'es de Morse pour la
  $d''$--cohomologie}}, Ann. Inst. Fourier (Grenoble) \textbf{35} (1985),
  189--229.

\bibitem{EH00}
E.~Epstein and G.~Henkin,
\emph{Stability of embeddings for pseudoconcave surfaces and their boundaries},
Acta Math.,\textbf{185} (2000), no. 2, 161--237.

\bibitem{FK72}
G.~B. Folland and J.~J. Kohn, \emph{The {N}eumann problem for the
  {C}auchy-{R}iemann complex}, Annals of Mathematics Studies. No.75. Princeton, N.J.: Princeton University Press and University of Tokyo Press, 146 p. (1972).

\bibitem{Hsiao08}
C-Y.~Hsiao,
\emph{Projections in several complex variables}, M\'em. Soc. Math. France, Nouv. S\'er., \textbf{123} (2010), 131 p.

\bibitem{HM09}
C-Y.~Hsiao and G.~Marinescu,
\emph{Szeg\"{o} kernel asymptotics and Morse inequalities on CR manifolds}, 
Math. Z. 271 (2012), Page 509--553. 

\bibitem{Hsiao12}
C-Y.~Hsiao,
\emph{Existence of CR sections for high power of semi-positive rigid Heisenberg line bundles over compact Heisenberg manifolds}, preprint available at arXiv:1204.4810. 


\bibitem{Ko65}
J.J.~Kohn,
\emph{Boundaries of complex manifolds.},
Proc. Conf. Complex Analysis, Minneapolis 1964, 81-94, (1965).

\bibitem{Ma96}
G.~Marinescu, \emph{{Asymptotic Morse Inequalities for Pseudoconcave Manifolds}},
  Ann. Scuola Norm. Sup. Pisa Cl. Sci. \textbf{23} (1996), no.~1, 27--55.
  
\bibitem{Si1:84}
Y.~T. Siu, \emph{{A vanishing theorem for semipositive line bundles over
  non--K\"ahler manifolds}}, J. Differential Geom. \textbf{20} (1984),
  431--452.

\bibitem{Yo80}
K.~Yosida,
\emph{Functional Analysis}, Repr. of the 6th ed.
Berlin: Springer-Verlag. xiv, 506 p., (1994).

\end{thebibliography}
\end{document}